\documentclass{amsart}
\usepackage{amssymb}
\usepackage{amsfonts}
\usepackage{amsmath}
\usepackage{geometry}
\usepackage{mathrsfs}
\usepackage{hyperref}

\usepackage{xcolor}

\definecolor{darkgreen}{rgb}{0.1,0.7,0.1}

\definecolor{darkred}{rgb}{0.7,0.1,0.1}

\hypersetup{
citecolor=blue,
colorlinks=true, %colorise les liens
urlcolor= blue, %couleur des hyperliens
linkcolor= black, %couleur des liens internes
bookmarksopen=true, %si les signets Acrobat sont crees,
pdftitle={Reconstruction}, %informations apparaissant dans
}

\newcommand{\averag}[2]{\bar{#1}^{\,^{#2}}}

\def\${|\!|\!|}

\newcommand{\tun}{\mathbf{1}}

\newcommand{\s}{\mathfrak{s}}

\newcommand{\cA}{\mathcal{A}}
\newcommand{\cB}{\mathcal{B}}
\newcommand{\cC}{\mathcal{C}}
\newcommand{\cD}{\mathcal{D}}
\newcommand{\cE}{\mathcal{E}}
\newcommand{\cF}{\mathcal{F}}
\newcommand{\cG}{\mathcal{G}}
\newcommand{\cH}{\mathcal{H}}
\newcommand{\cI}{\mathcal{I}}
\newcommand{\tcI}{\tilde{\mathcal{I}}}

\newcommand{\cM}{\mathcal{M}}

\newcommand{\cP}{\mathcal{P}}
\newcommand{\cQ}{\mathcal{Q}}
\newcommand{\cR}{\mathcal{R}}
\newcommand{\cS}{\mathcal{S}}
\newcommand{\cT}{\mathcal{T}}
\newcommand{\cU}{\mathcal{U}}

\newcommand{\cW}{\mathcal{W}}

\newcommand{\cY}{\mathcal{Y}}

\newcommand{\ccB}{\mathscr{B}}

\newcommand{\ccN}{\mathscr{N}}

\newcommand{\N}{ {\mathbb{N}} }
\newcommand{\Z}{ {\mathbb{Z}} }
\newcommand{\R}{ {\mathbb{R}} }
\newcommand{\T}{ {\mathbb{T}} }
\newcommand{\E}{ {\mathbb{E}} }
\renewcommand{\P}{ {\mathbb{P}} }

\newcommand{\eps}{\varepsilon}
\newtheorem{theorem}{Theorem}[section]
\theoremstyle{plain}

\newtheorem{assumption}{Assumption}

\newtheorem{definition}[theorem]{Definition}

\newtheorem{lemma}[theorem]{Lemma}

\newtheorem{proposition}[theorem]{Proposition}
\newtheorem{remark}[theorem]{Remark}

\numberwithin{equation}{section}
\usepackage{tikz}

\usetikzlibrary{shapes.misc}
\usetikzlibrary{shapes.symbols}
\usetikzlibrary{snakes}
\usetikzlibrary{decorations}
\usetikzlibrary{decorations.markings}

\def\drawx{\draw[-,solid] (-3pt,-3pt) -- (3pt,3pt);\draw[-,solid] (-3pt,3pt) -- (3pt,-3pt);}
\tikzset{
	root/.style={circle,fill=testcolor,inner sep=0pt, minimum size=2mm},
	dot/.style={circle,fill=black,draw=black, solid,inner sep=0pt,minimum size=0.5mm},
	square/.style={rectangle,fill=black,draw=black, solid,inner sep=0pt,minimum size=1mm},
	empty/.style={circle,fill=white,draw=white, solid,inner sep=0pt,minimum size=0.5mm},
	var/.style={circle,fill=black!10,draw=black,inner sep=0pt, minimum size=
	2mm},
	symb/.style={circle,fill=symbols,draw=symbols, solid,inner sep=0pt,minimum size=0.5mm},
	yy/.style={circle,fill=gray!20,draw=black,inner sep=0pt,minimum size=0.8mm},
	>=stealth,
	dotred/.style={circle,fill=black!50,inner sep=0pt, minimum size=2mm},
	generic/.style={semithick,shorten >=1pt,shorten <=1pt},
	dist/.style={ultra thick,draw=testcolor,shorten >=1pt,shorten <=1pt},
	testfcn/.style={ultra thick,testcolor,shorten >=1pt,shorten <=1pt,<-},
	testfcnx/.style={ultra thick,testcolor,shorten >=1pt,shorten <=1pt,<-,
		postaction={decorate,decoration={markings,mark=at position 0.6 with {\drawx}}}},
	kprime/.style={semithick,shorten >=1pt,shorten <=1pt,densely dashed,->},
	kprimex/.style={semithick,shorten >=1pt,shorten <=1pt,densely dashed,->,
		postaction={decorate,decoration={markings,mark=at position 0.4 with {\drawx}}}},
	kernel/.style={semithick,shorten >=1pt,shorten <=1pt,->},
	multx/.style={shorten >=1pt,shorten <=1pt,
		postaction={decorate,decoration={markings,mark=at position 0.5 with {\drawx}}}},
	kernelx/.style={semithick,shorten >=1pt,shorten <=1pt,->,
		postaction={decorate,decoration={markings,mark=at position 0.4 with {\drawx}}}},
	kernel1/.style={->,semithick,shorten >=1pt,shorten <=1pt,postaction={decorate,decoration={markings,mark=at position 0.45 with {\draw[-] (0,-0.1) -- (0,0.1);}}}},
	kernel2/.style={->,semithick,shorten >=1pt,shorten <=1pt,postaction={decorate,decoration={markings,mark=at position 0.45 with {\draw[-] (0.05,-0.1) -- (0.05,0.1);\draw[-] (-0.05,-0.1) -- (-0.05,0.1);}}}},
	kernelBig/.style={semithick,shorten >=1pt,shorten <=1pt,decorate, decoration={zigzag,amplitude=1.5pt,segment length = 3pt,pre length=2pt,post length=2pt}},
	rho/.style={dotted,semithick,shorten >=1pt,shorten <=1pt},
	renorm/.style={shape=circle,fill=white,inner sep=1pt},
	labl/.style={shape=rectangle,fill=white,inner sep=1pt},
	xi/.style={circle,fill=symbols!10,draw=symbols,inner sep=0pt,minimum size=1.2mm},
	xix/.style={crosscircle,fill=symbols!10,draw=symbols,inner sep=0pt,minimum size=1.2mm},
	xib/.style={circle,fill=symbols!10,draw=symbols,inner sep=0pt,minimum size=1.6mm},
	xibx/.style={crosscircle,fill=symbols!10,draw=symbols,inner sep=0pt,minimum size=1.6mm},
	not/.style={circle,fill=symbols,draw=symbols,inner sep=0pt,minimum size=0.5mm},
	>=stealth,
	}

\colorlet{symbols}{blue!90!black}

\makeatletter
\def\DeclareSymbol#1#2#3{\expandafter\gdef\csname MH@symb@#1\endcsname{\tikz[baseline=#2,scale=0.15]{#3}}%
\expandafter\gdef\csname MH@symb@#1s\endcsname{\scalebox{0.6}{\tikz[baseline=#2,scale=0.15]{#3}}}}
\def\<#1>{\csname MH@symb@#1\endcsname}
\makeatother

\DeclareSymbol{X}{-2.4}{\node[dot] {};}

\DeclareSymbol{1}{0}{\draw (0,0)  -- (0,1.4) node[dot] {};}

\DeclareSymbol{2}{0}{\draw (-0.6, 1.4) node[dot] {} -- (0,0) -- (0.6,1.4) node[dot] {};}

\DeclareSymbol{3}{0}{\draw (0,0) -- (0,1.2) node[dot] {}; \draw (-.7,1) node[dot] {} -- (0,0) -- (.7,1) node[dot] {};}

\DeclareSymbol{31}{-3}{\draw (0,0) -- (0,-1) -- (1,0) node[dot] {}; \draw (0,0) -- (0,1.5) node[dot] {}; \draw (-1,1.5) node[dot] {} -- (0,0.5) -- (1,1.5) node[dot] {};}

\DeclareSymbol{30}{-3}{\draw (0,0) -- (0,-1); \draw (0,0) -- (0,1.5) node[dot] {}; \draw (-1,1.5) node[dot] {} -- (0,0.5) -- (1,1.5) node[dot] {};}

\DeclareSymbol{10}{-3}{\draw (0,0) -- (0,-1); \draw (-.8,1) node[dot] {} -- (0,0);}

\DeclareSymbol{1c0}{-3}{\draw[thick,densely dotted] (0,0.8) -- (0,-1); \draw (-.8,1.5) node[dot] {} -- (0,0.8);}

\DeclareSymbol{32}{-3}{\draw (0,0) -- (0,-1) -- (1,0) node[dot] {}; \draw (0,0) -- (0,-1) -- (-1,0) node[dot] {}; \draw (0,0) -- (0,1.2) node[dot] {}; \draw (-.7,1) node[dot] {} -- (0,0) -- (.7,1) node[dot] {};}
\DeclareSymbol{12}{-3}{\draw (0,0) -- (0,-1) -- (1,0) node[dot] {}; \draw (0,0) -- (0,-1) -- (-1,0) node[dot] {}; \draw (-.8,1) node[dot] {} -- (0,0);}

\DeclareSymbol{22}{-3}{\draw (0,0.8) -- (0,-1) -- (1,0) node[dot] {}; \draw (0,0.8) -- (0,-1) -- (-1,0) node[dot] {};\draw (-.7,1.5) node[dot] {} -- (0,0.8) -- (.7,1.5) node[dot] {};}

\DeclareSymbol{2c2}{-3}{\draw[thick,densely dotted] (0,0.8) -- (0,-1);\draw (0,-1) -- (1,0) node[dot] {}; \draw (0,-1) -- (-1,0) node[dot] {};\draw (-.7,1.5) node[dot] {} -- (0,0.8) -- (.7,1.5) node[dot] {};}

\DeclareSymbol{3c2}{-3}{\draw[thick, densely dotted](0,0.5) -- (0,-1);\draw(0,-1) -- (1,0) node[dot] {}; \draw (0,-1) -- (-1,0) node[dot] {}; \draw (-0.7,1.5) node[dot] {} -- (0,0.5) -- (0.7,1.5) node[dot] {};}

\DeclareSymbol{20}{-3}{\draw (0,0.8) -- (0,-1);\draw (-.7,1.5) node[dot] {} -- (0,0.8) -- (.7,1.5) node[dot] {};}

\DeclareSymbol{2c0}{-3}{\draw[thick,densely dotted] (0,0.8) -- (0,-1);\draw (-.7,1.5) node[dot] {} -- (0,0.8) -- (.7,1.5) node[dot] {};}

\DeclareSymbol{32W}{-3}{\draw  [densely dotted, semithick] (-1.5,0)node[dot] {}  -- (0,-1.5) node[yy]{} -- (1.5,0) node[dot] {}; \draw (.8,-2.3)node{\tiny $y$} ;\draw[semithick] (0,0) -- (0,-1.5) node[yy] {}; \draw [densely dotted, semithick] (0,0) -- (0,1.8) node[dot] {}; \draw[densely dotted, semithick] (-1,1.5) node[dot] {} -- (0,0) -- (1,1.5) node[dot] {};}
\DeclareSymbol{32WW}{-3}{\draw  [densely dotted, semithick] (-1.5,0)node[dot] {}  -- (0,-1.5) node[yy]{} -- (1.5,0) node[dot] {}; \draw (.8,-2.3)node{\tiny $\bar{y}$} ;\draw[semithick] (0,0) -- (0,-1.5) node[yy] {}; \draw [densely dotted, semithick] (0,0) -- (0,1.8) node[dot] {}; \draw[densely dotted, semithick] (-1,1.5) node[dot] {} -- (0,0) -- (1,1.5) node[dot] {};}

\DeclareSymbol{s1}{0}{\draw[white] (-.4,0) -- (.4,0); \draw[symbols] (0,0)  -- (0,1.2) node[symb] {};}
\DeclareSymbol{s2}{0}{\draw[symbols] (-0.5,1.2) node[symb] {} -- (0,0) -- (0.5,1.2) node[symb] {};}
\DeclareSymbol{s3}{0}{\draw[symbols] (0,0) -- (0,1.2) node[symb] {}; \draw[symbols] (-.7,1) node[symb] {} -- (0,0) -- (.7,1) node[symb] {};}
\DeclareSymbol{s31}{-3}{\draw[symbols] (0,0) -- (0,-1) -- (1,0) node[symb] {}; \draw[symbols] (0,0) -- (0,1.2) node[symb] {}; \draw[symbols] (-.7,1) node[symb] {} -- (0,0) -- (.7,1) node[symb] {};}
\DeclareSymbol{s30}{-3}{\draw[symbols] (0,0) -- (0,-1); \draw[symbols] (0,0) -- (0,1.2) node[symb] {}; \draw[symbols] (-.7,1) node[symb] {} -- (0,0) -- (.7,1) node[symb] {};}
\DeclareSymbol{s32}{-3}{\draw[symbols] (0,0) -- (0,-1) -- (1,0) node[symb] {}; \draw[symbols] (0,0) -- (0,-1) -- (-1,0) node[symb] {}; \draw[symbols] (0,0) -- (0,1.2) node[symb] {}; \draw[symbols] (-.7,1) node[symb] {} -- (0,0) -- (.7,1) node[symb] {};}
\DeclareSymbol{s22}{-3}{\draw[symbols] (0,0.3) -- (0,-1) -- (1,0) node[symb] {}; \draw[symbols] (0,0.3) -- (0,-1) -- (-1,0) node[symb] {};\draw[symbols] (-.7,1) node[symb] {} -- (0,0.3) -- (.7,1) node[symb] {};}
\DeclareSymbol{s20}{-3}{\draw[symbols] (0,0) -- (0,-1);\draw[symbols] (-.7,1) node[symb] {} -- (0,0) -- (.7,1) node[symb] {};}
\DeclareSymbol{s21}{-3}{\draw[symbols] (0,0) -- (0,-1);\draw[symbols] (-.7,1) node[symb] {} -- (0,0) -- (.7,1) node[symb] {}; \draw[symbols] (0,0) -- (0,-1) -- (1,0) node[symb] {};}
\DeclareSymbol{s12}{-3}{\draw[symbols] (0,0) -- (0,-1) -- (1,0) node[symb] {}; \draw[symbols] (0,0) -- (0,-1) -- (-1,0) node[symb] {}; \draw (-.8,1) node[symb] {} -- (0,0);}
\DeclareSymbol{s11}{-3}{\draw[symbols] (0,0) -- (0,-1) -- (-1,0) node[symb] {}; \draw (-.8,1) node[symb] {} -- (0,0);}
\DeclareSymbol{s10}{-3}{\draw[symbols] (0,0) -- (0,-1); \draw[symbols] (-.8,1) node[symb] {} -- (0,0);}

\begin{document}

\title[Existence of densities for $\Phi^4_3$]{Existence of densities for the dynamic $\Phi^4_3$ model}

\author{Paul Gassiat}
\address{
Universit\'e Paris-Dauphine, PSL University, UMR 7534, CNRS, CEREMADE, 75016 Paris, France}
\email{gassiat@ceremade.dauphine.fr}

\author{Cyril Labb\'e}
\address{
Universit\'e Paris-Dauphine, PSL University, UMR 7534, CNRS, CEREMADE, 75016 Paris, France}
\email{labbe@ceremade.dauphine.fr}

\begin{abstract}
We apply Malliavin calculus to the $\Phi^4_3$ equation on the torus and prove existence of densities for the solution of the equation evaluated at regular enough test functions. We work in the framework of regularity structures and rely on Besov-type spaces of modelled distributions in order to prove Malliavin differentiability of the solution.  Our result applies to a large family of Gaussian space-time noises including white noise, in particular the noise may be degenerate as long as it is sufficiently rough on small scales.
\end{abstract}

\maketitle

\setcounter{tocdepth}{1}
\tableofcontents

\section{Introduction}

Consider the so-called dynamic $\Phi^4_d$ model
\begin{equation}\label{Eq:Phi4} \partial_t u = \Delta u - u^3 + \xi,\; \;\;\; u(0,\cdot)=u_0,\end{equation}
on the $d$-dimensional torus $\T^d$ of size $1$ and driven by a Gaussian noise $\xi$. In this paper, we focus on $d=3$ and investigate the existence of densities for the solution. Our main result applies to a large family of noises that includes, in particular, the space-time white noise, the precise assumptions on the noise will be specified later on.\\

This equation has been the object of several recent works in the fields of stochastic PDEs, let us give a very brief survey of this literature. In dimension $2$ and when $\xi$ is a space-time white noise, the solution of the equation was constructed by means of Dirichlet forms by Albeverio and R\"ockner~\cite{AlbRoc} and via a change-of-unknown by Da Prato and Debussche~\cite{dPD}. Among several subsequent results, let us mention that solutions were shown to be global-in-time~\cite{MouWeb2} and that existence and uniqueness of an invariant measure together with convergence to equilibrium were studied in~\cite{TsaWeb,ZhuZhu}.\\
In dimension $3$, existence of solutions when $\xi$ is irregular (in particular, a space-time white noise) fell out of reach of classical theories. The theory of regularity structures~\cite{Hairer2014} and the paracontrolled calculus~\cite{Max} provide new frameworks in which existence of solutions of such singular SPDEs can be tackled. In the case of the $\Phi^4_3$ model driven by a white noise, existence of local-in-time solutions was proved by Hairer~\cite{Hairer2014}, and Catellier and Chouk~\cite{CatChouk}. Let us also cite the work of Zhu and Zhu~\cite{ZhuZhu3} that constructs the solution by means of Dirichlet forms. Recently, Mourrat and Weber~\cite{MouWeb3} proved that solutions are actually global-in-time. Notice that the solution has space-time H\"older regularity $-1/2^-$ in the parabolic scale: therefore, it is not a function but only a distribution. In the aforementioned constructions of the solution, one actually renormalizes the equation by means of infinite constants; the equation formally becomes:
$$\partial_t u = \Delta u - u^3 + Cu + \xi,\;\; \;\;\; u(0,\cdot)=u_0$$
with $C=+\infty$.\\

In the present paper, we consider a noise $\xi$ which is obtained by convolving space-time white noise with a kernel $R$ satisfying Assumption \ref{asn:R} and either Assumption \ref{asn:Dense} or Assumption \ref{asn:Rough}. These assumptions are precisely presented in Section \ref{sec:noise}, let us simply mention that Assumption \ref{asn:R} requires the kernel to be regular enough (not worse than a Dirac), Assumption \ref{asn:Dense} asks for the associated Cameron-Martin space to be dense in $L^2$ while Assumption \ref{asn:Rough} ensures that $\xi$ is ``rough enough" (i.e.~of H\"older regularity strictly less than $-2$). The existence of solutions to \eqref{Eq:Phi4} in that setting is essentially granted by \cite{CH2016} and \cite{Hairer2014}.

To illustrate our assumptions, one can write a Paley-Littlewood type decomposition\footnote{The Fourier transform $\hat{K}_n$ of $K_n$ is concentrated on frequencies of order $2^n$, but unlike the standard Paley-Littlewood decomposition, $K_n$ (and not $\hat{K}_n$) is compactly supported.}
 $\zeta= \sum_{n\geq 0} K_n * \zeta$ of space-time white noise such that letting $\xi = R * \zeta$ with
 $$R= \sum_{n\ge 0} \alpha_n K_n\;,$$
where $\alpha_n \in \R$ for each $n$, then one has :
\begin{itemize}
\item $(\alpha_n)$ bounded $\Rightarrow$ Assumption \ref{asn:R} is satisfied,
%\item $\alpha_n>0$ for all $n\geq 0$ $\Rightarrow$ Assumption \ref{asn:Dense} is satisfied,
\item $\limsup_{n \to \infty} 2^{n\beta} |\alpha_n| >0$ for some $\beta < \frac{1}{2}$ $\Rightarrow$ Assumption \ref{asn:Rough} is satisfied,
\end{itemize}
see Proposition \ref{prop:PL} below.

\bigskip

We now state our main result. Let $\varphi_i, i=1\ldots n$ be $n\ge 1$ linearly independent functions in the (parabolically scaled) Besov space $\cB^{1/2+\kappa}_{1,\infty}(\R_+\times \T^3)$, for some $\kappa >0$, and assume that they are all supported in $(0,T)\times\T^3$ for some $T>0$.
\begin{theorem}\label{Th:Main}
Assume that the driving noise $\xi$ satisfies Assumption \ref{asn:R} and either Assumption \ref{asn:Dense} or Assumption \ref{asn:Rough} and that the solution $u$ of \eqref{Eq:Phi4} starting from some $u_0 \in \cC^{-2/3+}$ exists up to time $T$ almost surely. Then, the random variable $X=(\langle u, \varphi_1\rangle,\ldots,\langle u,\varphi_n\rangle)$ admits a density with respect to the Lebesgue measure on $\R^n$.
\end{theorem}

There exists already a substantial literature devoted to proving absolute continuity of densities for SPDEs with degenerate noise (and the often related ergodicity properties), going back to Ocone \cite{Ocone88} for the case of linear equations. In the case of polynomial nonlinearities perturbed by an additive noise which is white in time and smooth in space, a rather complete counterpart to the H\"ormander finite-dimensional theory was developed by Mattingly and coauthors (\cite{MP06,HM06,BM07,HM11}). Of course, we cannot apply these results to \eqref{Eq:Phi4} due to the roughness of the noise, which is actually one of the important technical difficulties we have to overcome. We also note that in the case of space-time white noise, the strong Feller property proved in \cite{HMStrongFeller} implies that for each $t>0$, the law of $u(t,\cdot)$ is absolutely continuous w.r.t.~the invariant measure for \eqref{Eq:Phi4}, which is a stronger statement than simply considering its finite-dimensional projections\footnote{{The absolute continuity of finite-dimensional projections of $u(t,\cdot)$ (still in the case of space-time white noise) has also been proven directly in the recent work of Romito \cite{Romito}.}}. Our result however can be applied to noises which are not white in time, where the Markovian theory is of course not accessible. In addition, we obtain densities for averages of our solution in space \emph{and time}, and not just at a fixed time which is the case considered in virtually all of the literature. (Note that the existence of densities for space-time averages is in principle a strictly stronger statement than densities for a fixed time, as soon as the regularity required for the test functions allows for Dirac masses in $t$. For technical reasons this is however not the case in our theorem).

Let us comment on our assumption on the existence of solutions up to time $T$. Mourrat and Weber~\cite{MouWeb3} showed that the explosion time of the solution is actually infinite when the driving noise is a space-time white noise. Their proof should carry through if we replaced the white noise by a more general driving noise satisfying the hypothesis considered in the present paper: consequently, the assumption on the existence of the solution up to time $T$ is probably not restrictive at all. Actually, we can disregard this assumption and show that the r.v. of the statement of the theorem, conditionally given the event $\{T < T_{\mbox{\tiny explo}}\}$, admits a density. To prove this more general statement, one needs to take care of the differentiability of the r.v. $\tun_{\{T < T_{\mbox{\tiny explo}}\}}X$: this can be done using the same techniques as in~\cite[Section 5.2]{MalliavinReg}. In order not to clutter this article, we preferred not to work in this level of generality.

\subsection*{Outline of the proof.} We rely on the theory of regularity structures~\cite{Hairer2014} to construct solutions to \eqref{Eq:Phi4}. The proof of the theorem is split into two main parts, corresponding to the two main assumptions required by the classical Bouleau-Hirsch criterion for existence of densities. First, we show that the random variable of the statement is Malliavin differentiable. Second, we prove that its Malliavin derivative is almost surely non-degenerate. 

\medskip

To carry out the first task, we start by constructing solutions of \eqref{Eq:Phi4} associated to a shifted noise $\xi+h$:
\begin{equation}\label{eq:shifted} \partial_t u = \Delta u - u^3 + \xi + h\;,\; \;\;\; u(0,\cdot)=u_0\end{equation}
where $h$ lies in the Cameron-Martin space associated to $\xi$ (under our assumptions, this is always a subspace of $L^2(\R_+\times\T^3)$), and of the associated tangent equation (formally obtained by differentiating $u$ w.r.t. $h$):
\begin{equation}\label{eq:tangent} \partial_t v = \Delta v - 3 u^2 v + h\;,\; \;\;\; v(0,\cdot)=0.\end{equation}
Then, we prove that $X$ is Malliavin differentiable and identify its derivative in direction $h$ as being $\big(\langle v,\varphi_1\rangle,\ldots,\langle v,\varphi_n\rangle\big)$.\\
To construct solutions of the above equations in the framework of regularity structures, one can think of two approaches. In the first approach, one adds a new abstract symbol $H$ in the regularity structure and builds the associated enlarged model. In the case of the generalized parabolic Anderson model, this strategy of proof was followed in~\cite{MalliavinReg} since the action of the model on only three new (but similar to each other) symbols needed to be defined. In the case of the $\Phi^4_3$ model, the action of the model on many more new symbols would need to be defined so that a construction by hand would be very tedious. In the second approach, that we actually follow in this paper, one lifts the convolution of $h$ with the heat kernel into an appropriate space of modelled distributions and solves the equation within the original regularity structure.\\
Working with ``classical" $L^\infty$-type spaces of modelled distributions as introduced in~\cite{Hairer2014} would require to view $h$ as an element of $\cC^\alpha=\cB^\alpha_{\infty,\infty}$. Classical embedding theorems show that $\alpha < -5/2$ so that the convolution of $h$ with the heat kernel has negative regularity in these spaces, and it is not possible to lift it as an $L^\infty$-type modelled distribution in the polynomial regularity structure.\\
On the other hand, working with $L^2$-type spaces of modelled distributions as introduced in~\cite{Recons} allows to lift the convolution of $h$ with the heat kernel into the polynomial regularity structure without losing regularity. However, one then needs to solve the equation in such an $L^2$-type space and the interplay of the cubic non-linearity with the $L^2$-type bounds may cause some difficulties. Fortunately, the embedding theorems for spaces of modelled distributions proved in~\cite{Recons} offer the necessary tools to make sense of these non-linear terms.\\
At a technical level, the spaces considered in~\cite{Recons} are not weighted near $t=0$ so that we have to adapt the analytical results presented therein to weighted spaces. Let us also mention that we work with space \textit{and time} $L^2$-type norms: this is problematic for iterating fixed point arguments since one needs to restart the equation from the already obtained solution evaluated at a given time (this requires to embed the solution into an $L^\infty$-type space in the time variable, thus losing too much regularity). This difficulty is circumvented by patching together solutions in a different manner, we refer to the discussion below Proposition \ref{Prop:Y}.

\medskip

To carry out the second task, following \cite{Ocone88} (and also \cite{MP06,BM07}), we work with a backward representation of the Malliavin derivative, namely for a given test function $\varphi \in \cB^{1/2+\kappa}_{1,\infty}$ supported in $(0,T) \times \T^3$, we consider $w$ which is (formally) solution to
\begin{equation}
\label{eq:w}
(-\partial_t -\Delta) w = -3u^2 w + \varphi, \;\;\; w(T,\cdot)=0
\end{equation}
(note that the product $u^2w$ is actually ill-defined, so to make rigorous sense of this equation we work again in a suitable set of modelled distributions), and we are then reduced to proving
$$\left( \left \langle w, h\right\rangle_{L^2([0,T]\times \T^3)} = 0 \;\;\;\forall h \in \cH \right)\; \Rightarrow \; \varphi = 0,$$
where $\cH$ is the Cameron-Martin space associated to the noise. Using the equation satisfied by $w$, a simple induction argument gives the implication
$$w = 0 \Rightarrow \varphi = 0,$$
so that when $\cH$ is dense in $L^2$ (Assumption \ref{asn:Dense}) the result follows immediately. When the noise is degenerate, one has near each point $z$ the local expansion for the r.h.s. of \eqref{eq:w}
$$ -3u^2 w + \varphi = -3 w(z) \<2> + R_z$$
where $\<2>$ is the (renormalized) square $\<2> = (\;\<1>\;)^2$, with $(\partial_t-\Delta)\; \<1> = \xi$, and our roughness assumption (Assumption \ref{asn:Rough}) implies that $R_z$ is of homogeneity near $z$ strictly greater than that of $-3 \<2> w$. By testing against suitable localized elements of $\cH$, we can then separate the contributions of the two terms to obtain that under the orthogonality condition, $w=0$ a.e.. Note that this type of argument based on the separation of scales appears frequently in this context (of proving the non-degeneracy of Malliavin derivatives), and is already present in the classical Malliavin proof of H\"ormander's theorem (via the uniqueness in the decomposition of a continuous semimartingale as the sum of a martingale and a bounded variation process). The precise argument then takes a different form based on the structure of the problem under consideration, for instance in the context of rough differential expansions this led to the notion of ``true roughness", cf. \cite{HP13,FS13,FH14}. The theory of regularity structures is particularly well-suited for this kind of argument, since as soon as the theory is used to solve an equation, it automatically gives a Taylor-like expansion (with terms of successively higher homogeneity) for the solutions.

\subsection*{Other SPDEs.} Our method is not specific to the dynamic $\Phi^4$ model and can in principle be applied to other singular SPDEs. The main requirement is that one can set up a fixed point argument for the shifted and tangent equations \eqref{eq:shifted}, \eqref{eq:tangent} in an $L^2$-type space of modelled distributions. In the case of $\Phi^4$, since the noise is additive, we are able to do that by decomposing the solution to the shifted equation as sum of an $L^\infty$ modelled distribution that captures the most irregular terms in the equation, and an $L^2$ modelled distribution with higher homogeneities, cf. section \ref{sec:shift}. This method would also apply for instance to a generalized KPZ equation of the form:
$$ \partial_t u = \partial^2_x u + f(u) (\partial_x u)^2 + \xi\;,$$
where $\xi$ is space-time white noise on $\R_+\times\T^1$. We could also treat the case of SHE (with the same noise)
$$ \partial_t u = \partial^2_x u +  g(u) \xi\;,$$
by directly solving the shifted equation in an $L^2$-space.
\\
However, it seems that there are SPDEs that fall into the scope of the theory of regularity structures for which our method would not apply, for instance the generalized KPZ equation with multiplicative noise~\cite{Yvain,String} which is the combination of the two equations above. Indeed, in that case one cannot solve the shifted equation directly in $L^2$, while a decomposition as described above does not hold.

\subsection*{Organisation of the paper.} In Section \ref{sec:noise}, we present the assumptions on the driving noise $\xi$. In Section \ref{Sec:Prelim}, we introduce the regularity structure associated with the $\Phi^4_3$ model, together with the appropriate spaces of modelled distributions we will work with. We also state the main analytical tools that we will need, and postpone their technical (but rather classical) proofs to Section \ref{Sec:Techos}. In Section \ref{Sec:Diff}, we prove Malliavin differentiability of the r.v. $X$ of the statement. In Section \ref{Sec:Densities}, we prove that the associated Malliavin derivative is almost surely non-degenerate and thus complete the proof of our main theorem.

\subsection*{Notations.} In this paper, the underlying space will always be the torus $\T^3$ of size $1$. It is convenient to work with functions defined on the whole space $\R^3$ but which have the periodicity of the torus. From now on, we will call periodic any such map. Notice that we will deal with space-time maps: periodicity will always refer to the space variable.\\
Some of our intermediate results hold in arbitrary space dimension so at several occasions in the paper, we will write $d$ for the space dimension.\\
We will be working in the so-called parabolic scaling $\s = (2,1,1,1)$ where $\s_0=2$ refers to time-direction, and $\s_1,\ldots,\s_3$ to space-directions. We set $|\s| = \s_0+\ldots+\s_3$. We consider the so-called $\s$-scaled metric $|z|_{\s}=|z|=\sup_{i=0,\ldots,3} |z_i|^{1/\s_i}$ for all $z\in \R^{4}$. For $\lambda \in \R$ and $z \in \R^4$ we let $\lambda \cdot z = (\lambda^{\s_i} z_i)_{i=0,\ldots,3}$. For a function $\phi$ on $\R^4$ and a given $(\lambda,z)$ we define $\phi^\lambda_z : y \mapsto |\lambda|^{-|\s|} \phi(\lambda^{-1} \cdot (y-z))$ (if $z=0$ we just write $\phi^\lambda$, and if $\lambda=1$ we just write $\phi_z$), note that this transformation preserves the $L^1$-norm. For a multi-index $k \in \N^4$, we let $|k|=|k|_{\s} = \sum_{i=0}^3 \s_i k_i$. For any $k\in\N^{d+1}$, we set $\partial^k = \prod_{i=0}^{d+1} \partial_{z_i}^{k_i}$.\\
We let $\cC^r=\cC^r(\R^{4})$ denote the space of all functions on $\R^4$ that admit continuous derivatives of order $k$, for all $k\in\N^4$ such that $|k| < r$. We further let $\ccB^r$ be the subset of $\cC^r$ whose elements are supported in the $\s$-scaled unit ball and have a $\cC^r$-norm smaller than or equal to $1$.\\
Similarly we let $\cB^\alpha_{p}(\R^4)=\cB^\alpha_{p,\infty}(\R^4)$ be the $\s$-scaled Besov space as defined in~\cite[Def. 2.1]{Recons}: notice that the parameter $q$ in the Besov scale will always be taken equal to $+\infty$ so we omit writing it.\\
For every $n\in\Z$, we define the dyadic grid of $\s$-scaled mesh $2^{-n}$
\begin{equation}
\label{eq:defLambda}
 \Lambda_n := \{(k_0 2^{-2n}, k_1 2^{-n}, k_2 2^{-n}, k_3 2^{-n}):  k=(k_0,\ldots,k_3)\in \Z^4\}\;.
\end{equation}
The Fourier transform of a tempered distribution $f \in \cS'(\R^4)$ is denoted by $\hat{f}$, it is defined by 
$$\left\langle \hat{f}, \phi \right\rangle = \left\langle f, \hat{\phi}\right\rangle$$
with
$$\hat{\phi}(\xi) = \int_{\R^4} \phi(z) e^{-i \xi \cdot z}dz$$
for $\phi$ in $\cS(\R^4)$.

%Recall that we typically identify distributions/functions on $\R\times \T^3$ as functions which are $1$-periodic in the last three arguments (we just say "periodic"). 
The Fourier transform of a smooth function $\phi$ on $\R\times \T^3$ is defined as the function
$$\hat{\phi}(\xi) = \int_{\R \times \T^3} \phi(z) e^{-i \xi \cdot z} dz$$
where the argument $\xi$ takes values in $\R \times (2\pi \Z)^3$. (Note : this is not exactly the same as the Fourier transform of $\phi$ viewed as a distribution on $\R^4$. There will hopefully be no confusion by which transform we mean). %(We could also define the Fourier transform of periodic distributions but we will not need it...).

One then has the isometry
$$\left\|f\right\|_{L^2(\R\times\T^3)} =  \left\| \hat{f} \right\|_{L^2(\R \times (2\pi \Z)^3;\hat{m})}$$
where the measure $\hat{m}$ is defined by
\begin{equation} \label{eq:defm}
 \int \phi(\xi_0,\xi_1,\xi_2,\xi_3) \hat{m}(d\xi) = \frac{1}{2\pi} \int_{\R} d\xi_0 \sum_{(k_1,k_2,k_3)\in \Z^3} \phi(\xi_0, 2\pi k_1, 2\pi k_2, 2\pi k_3)
 \end{equation}

 Given $f$ and $g$ two distributions such that the convolution $f \ast g$ makes sense (say $f$ is  compactly supported), if $g$ is periodic then so is $f\ast g$. Therefore it makes sense to view $f \ast g$ as a distribution on $\R \times \T^3$, and in that case one has
$$\widehat{f\ast g} = \hat{f} \hat{g}.$$

For a function $g : \R^{1+d} \to \R$ decaying sufficiently fast at infinity, its periodization $g^{per}$ is defined by
$$g^{per}(t,x) = \sum_{x_0 \in \Z^d} g(t,x-x_0).$$
One then has for all periodic $f$ (identified with a function on $\R \times \T^d$)
\begin{equation} \label{eq:per}
\left\langle f, g \right\rangle_{L^2(\R\times \R^d)} = \left\langle f, g^{per} \right\rangle_{L^2(\R\times \T^d)}.
\end{equation}
Finally, the notation $A \lesssim B$ means that $A \leq cB$ for some constant $c>0$ which does not depend on the parameters appearing in $A$ and $B$.

\subsection*{Acknowledgements.} The authors are supported by the ANR grant SINGULAR ANR-16-CE40-0020-01. PG would like to thank Yvain Bruned for a helpful discussion.

\section{Assumptions on the noise} \label{sec:noise}

We consider a Gaussian noise $\xi$ with covariance given by
$$\E\left[\left\langle \xi, \phi\right\rangle \left\langle \xi, \psi\right\rangle \right] = \left\langle R\ast \phi, R\ast \psi \right\rangle$$
(i.e. $\xi$ is the convolution of space-time white noise with $R$).

On $R$ we assume the following 
\begin{assumption} \label{asn:R}
One has the decomposition $R = \sum_{n \geq 0} R_n$ (the series is assumed to converge in the sense of distributions) where each $R_n$ is an even smooth function, supported in $\left\{|x|_{\s} \leq C 2^{-n}\right\}$ for some constant $C$. In addition, there exists $\beta \geq 0$, 
s.t. for each multiindex $k$, there exists $C_k>0$ with
\begin{equation}
\left\| \partial_k R_n \right\|_{L^\infty} \leq C_k 2^{n \left(|\s| + |k| - \beta\right)},
\end{equation}
and if $\beta=0$ one further has $\int R_n(x) dx=0$ for $n \geq 1$.
\end{assumption}

The assumption essentially says that $\xi$ has regularity no worse than white noise. The Cameron-Martin space $\cH$ associated to $\xi$ is  then given by $$ \cH =\overline{\left\{R \ast R \ast \phi, \;\;\; \phi \in \cC^\infty_c(\R \times \T^3)\right\}},$$
the closure being taken with respect to the norm
$$\left\|R \ast R \ast \phi \right\|_{\cH} = \left\|R \ast \phi \right\|_{L^2(\R \times \T^3)}.$$
In fact one also has
$$\|\psi \|_{\cH} := \inf\{ \|\phi\|_{L^2(\R\times \T^3)}, \;\; \psi=R\ast \phi\}$$
and
$$\cH = \overline{ \left\{R \ast \phi, \;\;\; \phi \in \cC^\infty_c\right\}}$$
(this can for instance be checked via the Fourier transform).

The next proposition then shows that under our assumption, $\cH$ is a subset of $L^2$.

\begin{proposition} \label{prop:L2H}
Under Assumption  \ref{asn:R}, there exists $C>0$ such that
\begin{equation} \label{eq:asnL2H}
\forall \phi \in L^2(\R \times \T^3), \left\| R\ast \phi \right\|_{L^2} \leq C \left\| \phi \right\|_{L^2}.
\end{equation}

\end{proposition}

\begin{proof}
When $\beta>0$, $R$ is in $L^1$ so that the result is obvious. Hence we now assume $\beta=0$. It is enough to prove that $\hat{R}$ is bounded.

Using the support property of $R_n$, by a Bernstein-type lemma (e.g. one can adapt the proof of \cite[Lemma 2.1]{BCD} to our parabolic setting) one has
$$\left\| \partial^k \hat{R}_n \right\|_{L^\infty} \lesssim 2^{- n |k|} \left\|\hat{R}_n \right\|_{L^\infty} \leq 2^{- n |k|} \left\|R_n\right\|_{L^1} \lesssim  2^{- n |k|} $$
so that, for $n \geq 1$, since $\hat{R}_n(0) = \int R_n = 0$, one obtains
$$ \left| \hat{R}_n(\xi) \right| \lesssim \left(2^{-n} |\xi|\right) \wedge 1.$$
%$$ \left| \hat{R}_n(\xi) \right| \lesssim \left(2^{-n} |\xi|\right) \wedge 1.$$

In addition, for all $i$ $=$ $0,\ldots,3$,
$$\sup_{\xi} \left| \hat{R}_n(\xi) |\xi_i| \right| \lesssim  \left\|  \widehat{ \partial_{z_i} R_n} \right\|_{L^\infty} \leq  \left\|\partial_{z_i} R_n\right\|_{L^1} \lesssim 2^{n \s_i}$$
so that we also have
$$\left| \hat{R}_n(\xi) \right| \lesssim \left(2^n |\xi|^{-1}\right)\wedge 1$$
and combining these two bounds yields that $\hat{R} = \sum_{n \geq 0} \hat{R}_n$ is bounded.
\end{proof}

To obtain non-degeneracy of the Malliavin derivative we will need that one of two additional assumptions holds. The first assumption is a density assumption on the Cameron-Martin space :
\begin{assumption}\label{asn:Dense}
The set $\left\{ h_{| [0,T] \times \T^3} \,\,, h \in \mathcal{H} \right\}$ is dense in $L^2([0,T] \times \T^3)$.
\end{assumption}

To formulate the second assumption we need some notations. For $C>1$ and $n\geq 0$, let
$$A_n^C = \left\{ \xi \in \R^4 : C^{-1} 2^{n} \leq |\xi| \leq  C 2^{n} \right\}$$
and 
$$B_n^C = \left\{(\xi,\xi') \in (\R^4)^2: \; \xi, \xi', \xi+\xi' \in A_n^C\right\}.$$
%Note that $$Vol(B_n)= 2^{2n |\s|} Vol(B_1)$$ with $Vol(B_1)>0$.

The assumption is then written as
\begin{assumption} \label{asn:Rough}
One has $\beta < \frac 1 2$ and for some $C \geq 1$,
\begin{equation} \label{eq:Rough1}
\limsup_{n \to \infty} 2^{3n \beta} \sup_{(\xi,\xi') \in B_n^C}  \left|\hat{R}(\xi)\hat{R}(\xi')\hat{R}(\xi+\xi')\right| > 0.
\end{equation}
\end{assumption}

The following simple lemma will be needed in the proof of Theorem \ref{Th:Main}.

\begin{lemma} \label{lem:RoughL2}
Under Assumption \ref{asn:R}, \eqref{eq:Rough1} is equivalent to
\begin{equation}  \label{eq:Rough2}
\limsup_{n \to \infty} 2^{2n (3\beta - |\s|)} \int_{B_n^C} \hat{m}(d\xi) \hat{m}(d\xi')  \left|\hat{R}(\xi)\hat{R}(\xi')\hat{R}(\xi+\xi')\right|^2 > 0.
\end{equation}
(recall that the measure $\hat{m}$ is defined by \eqref{eq:defm}).
\end{lemma}

\begin{proof}
We only prove that \eqref{eq:Rough1} implies \eqref{eq:Rough2} since the converse implication is immediate (and we will in fact not need it).

One first checks that for any $i$ in $\left\{0,\ldots,3\right\}$, 
\begin{equation} \label{eq:partialkHatR}
\left|\partial_{\xi_i} \hat{R}(\xi) \right| \lesssim|\xi|^{-\s_i}.
\end{equation}
To prove this, we write 
$$\left|\partial_{\xi_i} \hat{R}(\xi) \right|  \leq \sum_{0\leq n \leq n_0} \left|\partial_{\xi_i} \hat{R}_n(\xi)\right| + \sum_{n>n_0} \left|\partial_{\xi_i} \hat{R}_n(\xi)\right| $$
with $n_0$ such that $2^{n_0}\leq  |\xi| \leq 2^{n_0+1}$ . As in the proof of Proposition \ref{prop:L2H}, by Bernstein's lemma it holds that
$$ \sum_{n>n_0} \left\|\partial_{\xi_i} \hat{R}_n\right\|_{L^\infty} \lesssim \sum_{n>n_0} 2^{-n \s_i} \lesssim |\xi|^{-\s_i}$$
For the other sum, we note that for $N > \s_i$, 
\begin{eqnarray*}
 \sup_{\xi} \left| \partial_{\xi_i} \hat{R}_n(\xi) |\xi_j|^{N}\right| &\leq&  \sup_{\xi} \left| \partial_{\xi_i} \left( \hat{R}_n(\xi) (\xi_j)^{N}\right) \right| + \sup_\xi  \left| \hat{R}_n(\xi) \partial_{\xi_i} (\xi_j)^{N} \right|\\
 &\lesssim& 2^{-n \s_i} \left\|\partial_{z_j}^N R_n \right\|_{L^1} + 1_{i=j} \left\| \partial_{z_i}^{N-1} R_n \right\|_{L^1} \\
 &\lesssim& 2^{n (N \s_j - \s_i)},
 \end{eqnarray*}
 from which we deduce
 \[ \left|\partial_{\xi_i} \hat{R}_n(\xi) \right|\lesssim 2^{-n \s_i} \left( |\xi|^{-N} 2^{nN} \wedge 1\right).\]
 Hence we obtain
 $$ \sum_{n\leq n_0} \left|\partial_{\xi_i} \hat{R}_n(\xi)\right| \lesssim  \sum_{n\leq n_0} 2^{n(N-\s_i)} |\xi|^{-N} \lesssim  |\xi|^{-\s_i},$$
which concludes the proof of \eqref{eq:partialkHatR}.

From \eqref{eq:partialkHatR} and Taylor's formula (\cite[Prop. A.1]{Hairer2014}), if $(\zeta_n, \zeta_n') \in B_n^C$ is such that
$$2^{3n \beta}  \left|\hat{R}(\zeta_n)\hat{R}(\zeta_n')\hat{R}(\zeta_n+\zeta_n')\right| \geq K >0 $$
then there exists $\delta>0$ such that 
$$2^{3n \beta} |\hat{R}(\xi)\hat{R}(\xi')\hat{R}(\xi+\xi')| \geq \frac{K}{2} $$
whenever
$$(\xi,\xi') \in B'_n:=\left\{(\xi,\xi') \mbox{ s.t. } |\xi -\zeta_n| \leq 2^n \delta, |\xi'-\zeta'_n|\leq 2^n \delta\right\}.$$
and we note that
$$ \limsup_{n \to \infty}2^{-2n |\s|} \left(\hat{m} \otimes \hat{m} \right)\left(B'_n \cap B_{n}^C\right)= K'>0.$$
It follows that
$$\limsup_{n \to \infty} 2^{2n (3\beta - |\s|)} \int_{B_{n}^C} \hat{m}(d\xi) \hat{m} (d\xi') \left|\hat{R}(\xi)\hat{R}(\xi')\hat{R}(\xi+\xi')\right|^2 \geq \frac{K^2 K'}{4}>0.$$
\end{proof}

It is difficult to give examples of kernels $R$ satisfying Assumption \ref{asn:R} which have simple expressions as Fourier multipliers (in particular, because of the assumption that $R$ is compactly supported).

However, in the following proposition, we give an example of a Littlewood-Paley type decomposition $\sum_{n \geq 0} \zeta_n$ of space-time white noise $\zeta$, such that if one considers the noise $\xi = \sum_{n \geq 0} \alpha_n \zeta_n$, there exist simple sufficient conditions on the sequence $(\alpha_n)_{n \geq 0}$ for the previous assumptions to be fulfilled.

\begin{proposition} \label{prop:PL}
Assume that $\rho$ is a smooth, even, compactly supported function with $\int_{\R^4} \rho =1$, such that, letting $\eta(x) = \rho(x) - 2^{-|\s|} \rho(2^{-1}\cdot x)$, %one has
%\begin{equation} \label{eq:HatEtaPositive}
% \hat{\eta} \geq 0 \mbox{ on } \R^4,
% \end{equation}
%and that
 there exists $\xi_0$ in $\R^4$ such that
\begin{equation} \label{eq:CondPL2}
 \left|\hat{\eta}(\xi_0)\right| - \sum_{n \in \Z \setminus \{0\}} \;\;2^{-n \beta} \left|\hat{\eta}(2^{-n}\cdot\xi_0)\right| > 0.
\end{equation}
Then for any bounded sequence $(\alpha_n)_{n \geq 0}$, the kernel
$$R: = \alpha_0 \rho +  \sum_{n \geq 1} \alpha_n \eta^{2^{-n}}.$$
satisfies Assumption \ref{asn:R} for any $\beta$ such that $\limsup_{n\to \infty} 2^{n \beta} \left|\alpha_n \right|< \infty.$
In addition,
%\begin{enumerate}
%\item if $\alpha_n>0$ for all $n \geq 0$, then Assumption \ref{asn:Dense} is satisfied,
%\item 
if $\beta< \frac{1}{2}$ and $\limsup_{n \to \infty} 2^{n \beta} \left|\alpha_n\right|  \in (0,\infty)$ then Assumption \ref{asn:Rough} is satisfied.
%\end{enumerate}
\end{proposition}

Note that if $\alpha_n \equiv 1$ then $R=\delta_0$ which corresponds to space-time white noise. 
\begin{proof}

The fact that $R$ satisfies Assumption \ref{asn:R} with $R_0=\alpha_0 \rho$ and $R_n =  \alpha_{n} \eta^{2^{-n}}$ for $n \geq 1$ is a straightforward consequence of scaling.

%We now prove (1). To this end, we first note that Assumption \ref{asn:Dense} is equivalent to the condition
%\begin{equation} \label{eq:DenseFourier}
%\forall (k_0,k_1,k_2,k_3) \in \Z^4, \;\;\hat{R}\left(\frac{2\pi k_0}{T}, 2\pi k_1, 2\pi k_2, 2\pi k_3\right) \neq 0.
%\end{equation}
%Hence it is enough to prove that $\hat{R}>0$ on $\R^4$. Note that by the definition of $\eta$ it holds that
%\begin{equation} \label{eq:HatEta1}
% \forall \xi \in \R^4\setminus\{0\}, \;\; \sum_{n \in \Z} \hat{\eta}(2^{-n}\cdot \xi)  = \hat{\rho}(\xi) + \sum_{n>0} \hat{\eta}(2^{-n}\cdot \xi) = 1
%\end{equation}
%If $\alpha_n >0$ for all $n$, then one has $\hat{R}(0)= \alpha_0>0$
%and for all $\xi \neq 0$,
%$$\hat{R}(\xi) = \alpha_0 \hat{\rho}(\xi) + \sum_{n \geq 1} \alpha_n \hat{\eta}(2^{-n} \cdot \xi) >0$$
%by \eqref{eq:HatEtaPositive} and\eqref{eq:HatEta1}. 

We now prove the second point. We note $C:= \limsup_{n \to \infty} 2^{n \beta} |\alpha_n|$ so that  $2^{n \beta} |\alpha_n| = C + \varepsilon_n$ with $\limsup_{n\to \infty} \varepsilon_n = 0$. We then have 
\begin{align*}
2^{n \beta} \left| \hat{R}(2^n \cdot \xi_0) \right| & \geq 2^{n \beta} \left| \alpha_n \right| \left| \hat{\eta} (\xi_0)\right|   - \sum_{m\ge 1-n; m\ne 0}  2^{n \beta} \left| \alpha_{n+m} \right| \left| \hat{\eta}(2^{-m} \cdot \xi_0) \right| - 2^{n \beta} |\alpha_0| \hat{\rho}(2^{n} \cdot \xi_0)|\\
& \geq C \delta + \varepsilon_n \left| \hat{\eta} (\xi_0)\right| - \sum_{m \geq 1-n; m\ne 0} 2^{-m\beta}\varepsilon_{n+m} \left| \hat{\eta} (2^{-m} \cdot \xi_0)\right| - 2^{n \beta} |\alpha_0| \hat{\rho}(2^{n} \cdot \xi_0)|
\end{align*}
where we have let $\delta = \left|\hat{\eta}(\xi_0)\right| - \sum_{m \in \Z \setminus \{0\}} \;\;2^{-m \beta} \left|\hat\eta(2^{-m}\cdot\xi_0)\right|  >0.$ 

Note that the last term in the inequality above goes to $0$ as $n \to \infty$ since $\hat{\rho}$ decays rapidly at infinity. For the sum $\sum_{m } 2^{-m\beta}\varepsilon_{n+m} \left| \hat{\eta} (2^{-m} \cdot \xi_0)\right|$, we note that by similar arguments as in the proof of Proposition \ref{prop:L2H}, we have that 
$$\hat{\eta}(\zeta) \lesssim |\zeta| \wedge |\zeta|^{-N}$$
for $N$ arbitrarily large $(N>\beta$ will suffice), so that we can bound uniformly (in $n$) each summand by a term of an absolutely convergent series (in $m$). Hence we can interchange limits and obtain
\begin{align*}
&\limsup_{n \to \infty}  \sum_{m \geq 1-n} 2^{-m\beta}\varepsilon_{n+m} \left| \hat{\eta} (2^{-m} \cdot \xi_0)\right| - 2^{n \beta} |\alpha_0| \hat{\rho}(2^{n} \cdot \xi_0)| \\
 = &\sum_{m \neq 0} \limsup_{n \to \infty} \varepsilon_{n+m} 2^{-m\beta} \left| \hat{\eta} (2^{-m} \cdot \xi_0)\right| 
= 0.
\end{align*}
Finally we have
\[\limsup_{n \to \infty} 2^{n\beta} \left|\hat{R}(2^n \cdot \xi_0) \right| >0,\]
and \eqref{eq:Rough1} will be satisfied for any $C$ such that $(\xi_0,\xi_0)$ $\in$ $B_0^C$.
\end{proof}

For completeness we prove that such dyadic partitions of unity (with compact support in the space variable) exist.

\begin{lemma}
There exists $\rho$ satisfying the assumptions of Proposition \ref{prop:PL}.
\end{lemma}

\begin{proof}
We first proceed as for a standard dyadic partition of unity and define $\rho_0$ to be an even Schwartz function s.t.
$$\hat{\rho}_0 \equiv 1 \mbox{ on } |\xi|\leq 1, \;\;\;\;\hat{\rho}_0 \equiv 0 \mbox{ on } |\xi|\geq  2, \;\;\; \hat{\rho} \in[0,1] \mbox{ everywhere}.$$
Letting $\hat{\eta}_0(\xi) = \hat{\rho}_0(\xi) - \hat{\rho}_0(2\xi)$, it is clear that \eqref{eq:CondPL2} holds since in fact if we fix $\xi_0$ such that
 $|\xi_0|=1$, then
 $$ \hat{\eta}_0(\xi_0)=1, \mbox{ and } \;\;\; \forall n \in \Z\setminus \{0\}, \;\;\hat{\eta}_0(2^{-n} \cdot \xi_0) = 0.$$
Of course the issue is that $\rho_0$ is not compactly supported, so we fix $\phi$ smooth, even, compactly supported with $\phi(0)=0$, and let $\rho_\delta := \rho_0 \phi(\delta \cdot),$
and $\hat{\eta}_\delta(\xi) = \hat{\rho}_\delta(\xi) - \hat{\rho}_\delta(2\xi)$. Then we note that 
$$\hat{\rho}_\delta = \hat{\rho}_0 \ast \hat{\phi}^\delta,$$
and since $\left\|\hat{\phi}^\delta\right\|_{L^1}$ does not depend on $\delta$, we have for all multi-indices $k$
$$ \sup_{\delta \in (0,1]}\left\|\partial^k \hat{\rho}_\delta \right\|_{L^\infty} \lesssim \sup_{\delta \in (0,1]} \left\|\partial^k \hat{\rho}_0 \right\|_{L^\infty} \left\|\hat{\phi}^\delta\right\|_{L^1} \lesssim 1.$$
In addition, using that for all multi-indices $k$ one has for $\delta \leq 1$ $\left\| \partial^k \rho_\delta \right\|_{L^1} \lesssim \sum_{l \leq k} \left\| \partial^l \rho_0 \right\|_{L^1}$,
we obtain (using the same arguments as in the proof of Proposition \ref{prop:L2H}) that for arbitrary $N>0$
$$\sup_{\delta \in (0,1]} \left| \hat{\rho}_\delta(\xi)\right| \lesssim |\xi|^{-N} \wedge 1$$
and combining these two bounds yields
$$\sup_{\delta \in (0,1]} \left|\hat{\eta}_\delta(2^{-n}\cdot \xi_0)\right| \lesssim 2^{-n} \wedge 2^{nN}.$$
Hence we can pass to the limit and obtain
$$\lim_{\delta \to 0}  \;\;\; \left|\hat{\eta}_\delta(\xi_0)\right| - \sum_{n \in \Z \setminus \{0\}} \;\;2^{-n \beta} \left|\hat{\eta}_\delta(2^{-n}\cdot\xi_0)\right| = 1.$$
It then suffices to take $\delta$ small enough and let $$\rho = \frac{\rho_\delta}{\int_{\R^4} \rho_\delta}.$$

%TODO What about \eqref{eq:HatEtaPositive} ??? NOT CLEAR....
\end{proof}

\section{The regularity structure setting}\label{Sec:Prelim}

\subsection{Regularity structures}\label{Subsec:RegStruct}

A regularity structure is a triplet $(\cA,\cT,\cG)$ where:\begin{itemize}
\item $\cA$, the so-called set of homogeneities, is a subset of $\R$ assumed to be locally finite and bounded from below,
\item $\cT = \oplus_{\zeta \in \cA} \cT_\zeta$ is a graded normed vector space,
\item $\cG$ is a group of continuous linear maps on $\cT$ which is such that, for every $\Gamma \in \cG$, we have $\Gamma\tau - \tau \in \cT_{<\nu} :=  \oplus_{\zeta \in \cA,\zeta < \nu} \cT_\zeta$ whenever $\tau \in \cT_\nu$ for some $\nu \in \cA$.
\end{itemize}
We will denote by $\cQ_{\zeta}$ the projection from $\cT$ to $\cT_{\zeta}$ and we will use the notation $\left| \tau \right|_{\zeta} = \left|\cQ_{\zeta} \tau \right|$.
\medskip

The regularity structure $\mathcal{T}$ we consider is an extension of the usual regularity structure for $(\Phi_3^4)$ defined in \cite{Hairer2014}, since we need additional symbols to solve the dual backward equation \eqref{eq:W}. In particular, we need an abstract integration operator $\tilde{\cI}$ associated to the backward heat kernel. $\mathcal{T}$ is then given by all the formal linear combinations of elements of $\mathcal{F}$ where
$$\mathcal{F} =  \mathcal{U} \cup \mathcal{R}_U  \cup \mathcal{W}\cup \mathcal{R}_W$$
and $\mathcal{U}, \mathcal{R}_U, \mathcal{W}, \mathcal{R}_W$ are the smallest sets of symbols such that 
$$X^k \in \mathcal{U} \cap \mathcal{R}_U  \cap \mathcal{W}\cap \mathcal{R}_W \;\; \forall k \in \N^4$$
$$\Xi \in \mathcal{R}_U, $$
$$\tau_1, \tau_2, \tau_3 \in \mathcal{U} \Rightarrow \tau_1 \tau_2 \tau_3 \in \mathcal{R}_\mathcal{U}$$
$$\tau \in \mathcal{R}_U \Rightarrow \cI(\tau) \in \mathcal{U}$$
$$\tau_1, \tau_2 \in \mathcal{U}, \rho \in \mathcal{W} \Rightarrow \tau_1 \tau_2 \rho \in \mathcal{R}_W$$
$$\tau \in \mathcal{R}_W \Rightarrow \tilde{\cI}(\tau) \in \mathcal{W}$$
where as usual we take $\cI(X^k)=\tilde{\cI}(X^k)=0$ for all multiindices $k$.

The homogeneity of elements of $\cT$ is defined by letting
$$\left|\Xi\right| = -\frac{|\s|}{2} + \beta - \kappa, \;\; \left|1 \right|=0,$$
for some fixed positive $\kappa$ small enough, and then recursively
$$\left|\cI(\tau)\right| = \left|\tau\right| + 2, \;\;\left|\tcI(\tau)\right| = \left|\tau\right|+2, \;\; \left|\tau_1 \tau_2 \right|= \left|\tau_1\right| + \left|\tau_2 \right|.$$

To save space we omit the details of the construction of the structure group $\cG$ since we will not need them here, and refer instead to \cite{Hairer2014}.

We will frequently use the tree notation to describe elements of $\cT$: $\Xi$ is represented by a dot, the integration maps $\cI$ and $\tcI$ are represented by respectively straight lines and dotted lines, and the product of two symbols is represented by joining the corresponding trees at the root. For example :
$$\cI(\Xi)^2 = \<2>\;, \;\;\; \tcI(\cI(\Xi)^2) \cI(\Xi)^2 = \<2c2>.$$

\subsubsection{The heat kernel}
From now on, $r$ is an arbitrary integer larger than $5/2$. Recall that we denote by $P$ the usual heat kernel defined on the whole space $\R\times\R^3$. By~\cite[Lemma 7.7]{Hairer2014}, for any given $T>0$ there exists a collection of smooth compactly supported functions $P_-$ and $(P_m)_{m\ge 0}$ which vanish at all negative times and satisfy the following properties:\begin{enumerate}
\item $P*f(z) = P_-*f(z) + \sum_{m\ge 0} P_m*f(z)$ for all $z\in (-\infty,T]\times\R^3$ and every periodic map $f$,
\item $P_0$ is supported in $B(0,1)$,
\item we have for all $z=(t,x) \in \R\times\R^3$ and all $m\ge 1$
$$ P_m(z) = 2^{m(|\s|-2)} P_0(t 2^{2m},x_1 2^m,\ldots,x_3 2^m)\;,$$
\item $P_0$ annihilates polynomials of scaled degree $r$.
\end{enumerate}

\noindent Roughly speaking, $P_-$ stands for the smooth part of the heat kernel while for every $m\ge 0$, $P_m$ essentially coincides with $P$ in an annulus of radius $2^{-m}$ around $0$ and vanishes elsewhere.

As a consequence of these properties, we deduce that for any $k\in \N^{d+1}$, there exists a constant $C'>0$ such that for all $m\ge 0$ and all $z\in (0,\infty)\times \R^d$ we have
$$ \big| \partial^k P_m(z) \big| \le C' 2^{m(d+|k|)} \;.$$

In the sequel we will denote $P_+ = \sum_{m \geq 0} P_m$. We will also denote by $\tilde{P}(t,x) := P(-t,x)$ the backward heat kernel and define similarly the functions $\tilde{P}_-$ and $\tilde{P}_+$.

\subsubsection{Admissible models} \label{subsec:Model}
Let us now recall the notion of \textit{admissible model}. A pair $(\Pi,\Gamma)$ is called an admissible model if it satisfies the following assumptions:\begin{itemize}
\item For every $z\in \R^{d+1}$, $\Pi_z$ is a linear map from $\cT$ into the space of Schwartz distributions $\cD'(\R^{d+1})$ and we have the bound
\begin{equation} \label{eq:DefPi}
 \| \Pi \|_K := \sup_{z\in K} \sup_{\lambda \in (0,1]} \sup_{\zeta\in \cA_{<\gamma}} \sup_{\tau\in \cT_\zeta} \sup_{\eta\in\ccB^r} \frac{\big|\langle \Pi_z\tau , \eta^\lambda_z\rangle\big|}{|\tau| \lambda^\zeta} < \infty\;,
 \end{equation}
 for every bounded domain $K \subset \R^{d+1}$.
\item For every $z,z' \in \R^{d+1}$, $\Gamma_{z,z'}$ is an element of $\cG$ and we have
$$ \| \Gamma \|_K := \sup_{z,z'\in K, |z-z'| \le 1} \sup_{\nu \le \zeta} \sup_{\tau\in\cT_\zeta} \frac{\big| \Gamma_{z,z'} \tau\big|_\nu}{|\tau| |z-z'|^{\zeta-\nu}} < \infty\;,$$
for every bounded domain $K \subset \R^{d+1}$.
\item For all $z,z'\in \R^{d+1}$, $\Pi_{z}\Gamma_{z,z'} = \Pi_{z'}$.
\item For every multiindex $k\in \N^{d+1}$, it holds
\begin{align*}
\Pi_z X^k (z') &= (z'-z)^k\;,\\
\Pi_z \cI \tau (z') &= \langle \Pi_z \tau, P_+(z'-\cdot) \rangle - \sum_{|k|<|\tau|+2} \frac{(z'-z)^k}{k!} \langle \Pi_z \tau , \partial^k P_+(z-\cdot)\rangle\;,\\
\Pi_z \tilde\cI \tau (z') &= \langle \Pi_z \tau, \tilde{P}_+(z'-\cdot) \rangle - \sum_{|k|<|\tau|+2} \frac{(z'-z)^k}{k!} \langle \Pi_z \tau , \partial^k \tilde{P}_+(z-\cdot)\rangle\;.
\end{align*}
\end{itemize}

In order to set up a solution theory for \eqref{Eq:Phi4} we also need the following conditions (cf \cite[Section 9.4]{Hairer2014})

\begin{equation} \label{eq:Modelsup1}
 \sup_{s \in \R} \left\| \mathbf{1}_{\{t>s\}}\Pi \Xi\right\|_{\cC^{-\frac{|\s|}{2} - \kappa}(K)}  < \infty
 \end{equation}
\begin{equation} \label{eq:Modelsup2} \sup_{t \in [0,T]} \left\| (P_+ \ast \Pi\Xi)(t,\cdot) \right\|_{\cC^{-\frac12 - \kappa}(K')} < \infty
\end{equation}
for all compact sets $K\subset \R\times\R^3$ and $K'\subset \R^3$, all $T>0$ and any $\kappa >0$. Here $\cC^{\alpha}(K)$ stands for the space of $\s$-scaled $\alpha$-H\"older distributions on $K$.

We let $\cM$ be the space of admissible models satisfying all of the above, and we equip it with the topology associated with the corresponding system of semi-norms.

We will consider a model $\Pi$ obtained by renormalizing smooth canonical models as described in full generality in \cite{CH2016} (note however that our case is only a simple extension of the one in \cite{Hairer2014}, since our noise does not have worse regularity than white noise and the only additional tree of negative regularity in our structure compared to \cite{Hairer2014} is $\<2c2>$ which is renormalized in the exact same way as $\<22>$).

For a given smooth function $\zeta$ on $\R\times \T^3$, the canonical model associated to $\zeta$ is the unique admissible model defined by letting $\Pi_z \Xi = \zeta$, and then recursively by letting $\Pi_z(\tau \tau') = (\Pi_z \tau) (\Pi_z \tau')$, for $\tau$, $\tau'$ such that $\tau \tau'$ $\in$ $\cT$.

We now fix a sequence of mollifiers $\rho^\varepsilon = \varepsilon^{-|\s|} \rho(\varepsilon^{-1} \cdot)$, where $\rho$ is a smooth compactly supported function integrating to $1$, and we let $\xi^\varepsilon = \xi \ast \rho^\varepsilon$ be regularizations of our noise. The models $\Pi^\varepsilon$ are then defined as the canonical models associated to $\xi^\varepsilon$.

It is well known that the sequence $\Pi^\varepsilon$ does not converge as $\varepsilon \to 0$, and that one needs to introduce a renormalization. In our case, the renormalization we need to consider can be described by the group $\mathfrak{R}$ of transformations on $\cT$ of the form $\exp\left( - C_1 L_1 - C_2 L_2 - C_3 L_3 \right)$, where the $L_i$'s are determined by the substitutions
$$L_1 : \<2> \mapsto 1, \;\;\; L_2 : \<22> \mapsto 1, \;\;\; L3 : \<2c2> \mapsto 1$$
(i.e. each $L_i$ acts on a tree $\tau$ in $\cT$ by replacing formally each occurence of the associated tree in $\tau$ by $1$).

One can then define a renormalized model $\Pi^M$ for each $M$ in $\mathfrak{R}$, and for each admissible model $\Pi$. We will not need the precise definition of $\Pi^M$, but in our case one has for each smooth model $\Pi$ the relation
\begin{equation} \label{eq:PiM}
\left(\Pi^M_z\tau \right) (z) = \left( \Pi_z (M\tau)\right)(z), \;\;\;\;\;\forall z \in \R \times \T^3
\end{equation}
which is useful to determine the equations satisfied by (reconstructions of the) solutions to the abstract fixed point equations. 

In our case, we let $M^\varepsilon$ correspond to the constants
$$C_1^\varepsilon = \int (K_\varepsilon(z))^2 dz,$$
$$ C_2^\varepsilon = \int dz_1 dz_2 dz_3 P_+(z_1) K_\varepsilon(z_3-z_1) K_\varepsilon(z_2-z_1) K_\varepsilon(z_3) K_\varepsilon(z_2)$$
$$C_3^\varepsilon = \int dz_1 dz_2 dz_3 P_+(- z_1) K_\varepsilon(z_3-z_1) K_\varepsilon(z_2-z_1) K_\varepsilon(z_3) K_\varepsilon(z_2)$$
(where $K_\varepsilon = R \ast P_+ \ast \rho^\varepsilon$).

In fact, it holds that
\begin{equation} \label{eq:C2C3}
C_2^\varepsilon = C_3^\varepsilon
\end{equation}
as can be seen by the change of variables $(z'_1,z'_2,z'_3) = (-z_1,z_2-z_1,z_3-z_1)$

The renormalized models then converge as $\varepsilon \to 0$ :

\begin{theorem} \label{thm:model}
Let $\hat{\Pi}^\varepsilon := (\Pi^\varepsilon)^{M_\varepsilon}$, then there exists an admissible model $\Pi$ such that $\hat{\Pi}^\varepsilon \to \Pi$ in probability in $\mathfrak{M}$. In addition, $\Pi$ does not depend on the particular choice of $\rho$.
\end{theorem}

\begin{proof}
Convergence of the models is a special case of the results in \cite{CH2016} (and is essentially already contained in \cite[section 10.5]{Hairer2014}, although our noise is not exactly the same). 

The fact that convergence also holds w.r.t.~the bounds \eqref{eq:Modelsup1} \eqref{eq:Modelsup2} follows exactly as in \cite[Proposition 9.5]{Hairer2014} :
\begin{itemize}
\item The first bound follows by the exact same proof as in \cite{Hairer2014}, using \eqref{eq:asnL2H}.
\item The second bound again follows from the proof in \cite{Hairer2014}, once we note that $ P_+ \ast R$ satisfies the same scaling assumptions as $P_+$ (cf. e.g. \cite[Lemma 10.14, Lemma 10.16]{Hairer2014}), and these are the only properties of $P_+$ used in the proof).
\end{itemize}

\end{proof}

\subsection{Weighted spaces of Besov modelled distributions}\label{Subsec:Dgamma}

We first recall the definition of the (unweighted) spaces of Besov modelled distributions introduced in~\cite{Recons}.

\begin{definition}\label{Def:Dgamma}
Take $\gamma \in \R$. We let $\cD^\gamma_{p}$ be the Banach space of all periodic maps $f: \R\times\R^3 \rightarrow \cT_{<\gamma}$ such that for all $\zeta\in\cA_{<\gamma}$, we have:
\begin{align*}
\Big\| \big| f(z) \big|_\zeta \Big\|_{L^p(\R\times\T^3,dz)} &< \infty\;,\quad\sup_{h\in B(0,1)} \bigg\| \frac{\big| f(z+h)-\Gamma_{z+h,z} f(z) \big|_\zeta}{|h|^{\gamma-\zeta}} \bigg\|_{L^p(\R\times\T^3,dz)} &< \infty\;.
\end{align*}
We let $\$f\$$ be the corresponding norm.
\end{definition}

\begin{remark}
For the sake of consistency with~\cite{Recons}, we should have denoted our spaces by $\cD^\gamma_{p,q}$ with $q=\infty$. Since the parameter $q$ will always be taken equal to $+\infty$ in the present work, we omit writing it.
\end{remark}

Let us now introduce spaces of modelled distributions with weights near $t=0+$. The reason for considering such weights is twofold: first, it allows to start the equation at stake from some irregular initial condition, second playing around with the weight parameter $\eta$ will eventually allow us to obtain contractivity of the fixed point map.

\begin{definition}\label{Def:DgammaW}
Let $\gamma \in \R$, $T > 0$ and $\eta \le \gamma$. We let $\cD^{\gamma,\eta,T}_{p}$ be the space of all periodic maps $f: (0,T)\times\R^3 \rightarrow \cT_{<\gamma}$ such that for all $\zeta\in\cA_{<\gamma}$:
\begin{equation}\begin{split}\label{Eq:BdDgamma}
\Big\| \frac{\big| f(z) \big|_\zeta}{t^{\frac{\eta-\zeta}{2}}} \Big\|_{L^p((0,T)\times \T^3,dz)} &< \infty\;,\\
\sup_{h\in B(0,1)} \bigg\| \frac{\big| f(z+h)-\Gamma_{z+h,z} f(z) \big|_\zeta}{|h|^{\gamma-\zeta} \, t^{\frac{\eta-\gamma}{2}}} \bigg\|_{L^p((3|h|^2,T-|h|^2)\times \T^3,dz)} &< \infty\;.
\end{split}\end{equation}
We let $\$f\$$ be the corresponding norm.
\end{definition}

Notice that the exponent of the weight of the local terms is $\frac{\eta-\zeta}{2}$, and not $\frac{(\eta-\zeta)\wedge 0}{2}$ as in~\cite[Section 6]{Hairer2014}. The reason for this choice is simple: the forthcoming embedding theorems would not hold true with $\frac{(\eta-\zeta)\wedge 0}{2}$. Let us mention here that this has some technical consequences: some arguments in the original proof of the convolution with a singular kernel~\cite[Thm 6.16]{Hairer2014} need to be adapted, see in particular the refined bound \eqref{Eq:ReconsConvol} that we will need.

When we are given two models $(\Pi,\Gamma)$ and $(\bar\Pi,\bar\Gamma)$, we will need to compare elements $f$ and $\bar f$ that belong respectively to the spaces $\cD^{\gamma,\eta,T}_p$ and $\bar{\cD}^{\gamma,\eta,T}_p$. To that end, we set:
$$ \|f-\bar f\| := \sup_{\zeta \in \cA_{<\gamma}} \Big\| \frac{\big| f(z)-\bar f(z) \big|_\zeta}{t^{\frac{\eta-\zeta}{2}}} \Big\|_{L^p((0,T)\times \T^3,dz)}\;,$$
as well as
$$ \$ f;\bar f \$ := \|f-\bar f\| + \sup_{h\in B(0,1)} \bigg\| \frac{\big| f(z+h) - \bar{f}(z+h) - \Gamma_{z+h,z} f(z) + \bar\Gamma_{z+h,z} \bar{f}(z) \big|_\zeta}{|h|^{\gamma-\zeta} \, t^{\frac{\eta-\gamma}{2}}} \bigg\|_{L^p((3|h|^2,T-|h|^2)\times \T^3,dz)}\;.$$

\medskip

We now present the main analytical tools associated to these spaces that are needed in the construction of the solution as well as for the proof of the Malliavin differentiability. In order not to clutter the presentation, we postpone the proofs of the forthcoming statements to Section \ref{Sec:Techos} but make some comments on the differences with their original versions (in the H\"older setting) in~\cite{Hairer2014}.

Recall that $r$ is an integer taken larger than $5/2$.

\subsubsection{Reconstruction}

Since we are dealing with modelled distributions defined on the time interval $(0,T)$ only, we need to introduce appropriate Besov-type spaces. To that end, let us introduce some notations. We let $\ccB^r$ be the set of all $\cC^r$ functions from $\R\times\R^d$ into $\R$, which are compactly supported in $B(0,1)$ and whose $\cC^r$-norm is less than or equal to $1$. We then define $\ccB^r_n$ as the subset of $\ccB^r$ whose elements annihilate all polynomials of scaled degree at most $n\in \N$.

\begin{definition}
Fix $T>0$. If $\nu < 0$, we let $\cB^{\nu,T}_{p}$ be the set of all periodic distributions $\xi$ acting on test functions supported in $(-\infty,T)\times\R^d$ such that
\begin{equation}\label{Eq:BgammaT}
\sup_{\lambda \in (0,1]} \Big\| \sup_{\varphi\in\ccB^r} \frac{\big|\langle \xi,\varphi^\lambda_z \rangle\big|}{\lambda^{\nu}} \Big\|_{L^p((-\infty, T-\lambda^2)\times \T^d,dz)} < \infty\;.
\end{equation}
If $\nu \ge 0$, we let $\cB^{\nu,T}_{p}$ be the set of all periodic distributions $\xi$ acting on test functions supported in $(-\infty,T)\times\R^d$ such that
\begin{equation}\label{Eq:BgammaTPos}
\Big\| \sup_{\varphi\in\ccB^r} \big|\langle \xi,\varphi_z \rangle\big| \Big\|_{L^p((-\infty, T-1)\times \T^d,dz)} < \infty\;,\quad \sup_{\lambda \in (0,1]} \Big\| \sup_{\varphi\in\ccB^r_{\lfloor \nu \rfloor}} \frac{\big|\langle \xi,\varphi^\lambda_z \rangle\big|}{\lambda^{\nu}} \Big\|_{L^p((-\infty, T-\lambda^2)\times \T^d,dz)} < \infty\;.
\end{equation}
\end{definition}

The following fundamental result asserts the existence and uniqueness of a linear operator that associates to every modelled distribution a genuine distribution.

\begin{theorem}[Reconstruction]\label{Th:ReconstructionW}
Let $\gamma >0$, $\eta \in \R$, $T>0$ and set $\alpha := \min \cA\backslash\N$. Assume that $\alpha\wedge \eta > -2(1-\frac1{p})$. We set $\bar\alpha = \alpha\wedge \eta$ if $\alpha \wedge \eta \ne 0$, otherwise we let $\bar \alpha$ be any arbitrary negative value. There exists a unique continuous linear map $\cR:\cD^{\gamma,\eta,T}_{p}\rightarrow \cB^{\bar\alpha,T}_{p}$ such that $\langle \cR f, \varphi \rangle = 0$ whenever the support of $\varphi$ lies in $(-\infty,0]\times \T^d$, and such that
\begin{equation}\label{Eq:ReconsBound}
\sup_{\lambda\in (0,1]} \bigg\| \sup_{\varphi\in \ccB^r} \frac{\big|\langle \cR f-\Pi_z f(z), \varphi_z^\lambda \rangle\big|}{\lambda^{\gamma} \, t^{\frac{\eta-\gamma}{2}}} \bigg\|_{L^p((3\lambda^2,T-\lambda^2)\times\T^d,dz)} \lesssim \$f\$\;,
\end{equation}
uniformly over all $f\in \cD^{\gamma,\eta,T}_{p}$.\\
In the case where there are two models $(\Pi,\Gamma)$, $(\bar\Pi,\bar\Gamma)$, we get the following counterpart of \eqref{Eq:ReconsBound}
\begin{equation}\label{Eq:ReconsBoundTwo}\begin{split}
&\sup_{\lambda\in (0,1]} \bigg\| \sup_{\varphi\in \ccB^r} \frac{\big|\langle \cR f-\bar{\cR} \bar f - \Pi_z f(z) + \bar{\Pi}_z \bar f(z), \varphi_z^\lambda \rangle\big|}{\lambda^{\gamma} \, t^{\frac{\eta-\gamma}{2}}} \bigg\|_{L^p((3\lambda^2,T-\lambda^2)\times\T^d,dz)}\\
&\lesssim \$f;\bar f\$ \|\Pi\| (1+\|\Gamma\|) + \$\bar f\$\big( (\|\Pi-\bar\Pi)(1+\|\Gamma\|) + \|\bar\Pi\| \|\Gamma-\bar\Gamma\| \big)\;,
\end{split}\end{equation}
uniformly over all $f,\bar f$ in $\cD^{\gamma,\eta,T}_{p}, \bar{\cD}^{\gamma,\eta,T}_{p}$. 
\end{theorem}
The proof of this result relies on the reconstruction theorem for unweighted modelled distributions, see~\cite[Thm 3.1]{Recons}. Indeed, since the spaces of weighted and unweighted modelled distributions are locally similar, and since the reconstruction operator is local, one can apply the operator constructed in the aforementioned reference to weighted modelled distributions when tested against test functions supported away from the hyperplane $t=0$. Then, one needs to patch together in a consistent way these distributions in order to get an element in $\cB^{\nu,T}_p$: this raises the restriction $\eta\wedge \alpha > -2(1-1/p)$ of the statement. Notice that this is in line with the restriction $\alpha\wedge\eta > -2$ of~\cite[Prop 6.9]{Hairer2014} in the setting of H\"older distributions ($p=\infty$).

\subsubsection{Embedding}

Classical Besov spaces enjoy embedding properties: in particular, one can improve the integrability of a function/distribution at the cost of losing some regularity. In~\cite{Recons}, embedding theorems were established for unweighted Besov-type modelled distributions (Definition \ref{Def:Dgamma}). Notice that in the context of regularity structures, the regularity parameter which is traded off against integrability is no longer the actual regularity of the function/distribution but rather the parameter $\gamma$ (which stands for the order of the generalized Taylor expansion at stake). Below, we extend the scope of the embedding theorems of~\cite{Recons} to the case of weighted spaces near $t=0$ (since the parameter $q$ is set to $+\infty$ in this paper, we do not treat the embedding properties associated with $q$). The main difference with the original version is that one also needs to decrease the value of $\eta$ (by the same amount as $\gamma$) in order to improve integrability.

\begin{theorem}[Embeddings with weights]\label{Th:EmbeddingW}
The space $\cD^{\gamma,\eta,T}_{p}$ is continuously embedded into $\cD^{\gamma',\eta',T}_{p'}$ in any of the following situations:
\begin{enumerate}
\item $p'<p$, $\gamma'=\gamma$ and $\eta'=\eta$,
\item $p'>p$, $\gamma'<\gamma - |\s|\big(\frac1{p}-\frac1{p'}\big)$ and $\eta'=\eta + \gamma'-\gamma$.
\end{enumerate}
\end{theorem}

\smallskip

\subsubsection{Product}

The notion of sector is introduced in~\cite[Def 2.5]{Hairer2014}: roughly speaking, a sector $V$ of regularity $\alpha$ is a ``sub-regularity structure" whose smallest level is $\alpha$. Two sectors $V_1$ and $V_2$ are said to be $\gamma$-regular if $\Gamma(\tau_1\tau_2) = \Gamma \tau_1 \Gamma \tau_2$ for all $\Gamma\in\cG$ and all $\tau_i \in \cT_{\zeta_i}\cap V_i$ such that the $\zeta_i$'s satisfy $\zeta_1+\zeta_2 < \gamma$. We denote by $\cD^{\gamma,\eta,T}_p(V)$ the subspace of $\cD^{\gamma,\eta,T}_p$ whose elements take values in $V$.

\begin{theorem}[Multiplication]\label{Th:Mult}
Let $f_i \in \mathcal{D}^{\gamma_i,\eta_i,T}_{p_i}(V_i)$, $i=1,2$, where $V_1,V_2$ are $\gamma$-regular sectors of regularity $\alpha_1$, $\alpha_2$. Then $f:= f_1 f_2$ belongs to $\mathcal{D}^{\gamma,\eta,T}_p$, where
$$\gamma = (\gamma_1 +\alpha_2)\wedge (\gamma_2 + \alpha_1)\;,\quad \eta = \eta_1+\eta_2\;,\quad\frac{1}{p} = \frac{1}{p_1} +\frac{1}{p_2} \;.$$
If we are given two models $(\Pi,\Gamma)$ and $(\bar\Pi,\bar\Gamma)$, then we have the bound
\begin{align*}
\$ f_1 f_2 ; g_1 g_2\$ &\lesssim \|\Gamma\|^2 \| f_1-g_1\| \$f_2\$ + \|\Gamma-\bar\Gamma\| \Big(\$g_1\$ \$f_2\$ + \$g_2\$ \$f_1\$\Big) + \$f_1;g_1\$ \$f_2\$\\
&+ \|\Gamma\| \$f_1\$ \$f_2;g_2\$ + \$g_1\$ \|f_2-g_2\| + \|\bar\Gamma\| \|f_1-g_1\| \$ g_2\$\;,
\end{align*}
uniformly over all $f_i \in \cD^{\gamma_i,\eta_i}_{p_i}(V_i)$ and all $g_i \in \bar{\cD}^{\gamma_i,\eta_i}_{p_i}(V_i)$.
\end{theorem}

This result is in the flavour of~\cite[Prop 6.12]{Hairer2014}. The main difference is that $\eta$ is not given by the infimum of $\eta_1+\alpha_2$, $\eta_2+\alpha_1$ and $\eta_1+\eta_2$ but is equal to the latter. This is a consequence of our choice of exponents for the weights in the local terms: $\frac{\eta-\zeta}{2}$, and not $\frac{(\eta-\zeta)\wedge 0}{2}$ as in~\cite{Hairer2014}.

\subsubsection{Convolution with the heat kernel}

Recall the decomposition of the heat kernel introduced in Section \ref{Sec:Prelim}. For convenience, we set $P_+ := \sum_{m\ge 0} P_m$ and call this function the singular part of the heat kernel, by opposition to $P_-$ that we call the smooth part of the heat kernel. The goal of the present section is to lift the convolution with the heat kernel at the level of the spaces $\cD^{\gamma,\eta,T}_{p}$. This will be carried out separately for the singular part and the smooth part.

We start with the former, which is the most involved. We set for any $f\in \cD^{\gamma,\eta,T}_p$
\begin{align*}
\cP_+^\gamma f(z) &:= \cI(f(z)) + \sum_{\zeta\in \cA_\gamma} \sum_{k\in\N^{d+1}:|k| < \zeta + 2} \frac{X^k}{k!} \langle \Pi_z \cQ_\zeta f(z),\partial^k P_+(z-\cdot)\rangle\\
&+ \sum_{k\in\N^{d+1}:|k| < \gamma + 2} \frac{X^k}{k!} \langle \cR f - \Pi_z f(z) , \partial^k P_+(z-\cdot)\rangle\;.
\end{align*}

\begin{theorem}[Convolution - singular part]\label{Th:ConvSing}
Let $\alpha := \min\cA\backslash\N$. Fix $\gamma >0$ and $\eta \le \gamma$, and let $\gamma' = \gamma + 2$ and $\eta' < \eta + 2$. Assume that $\gamma' \notin \N$ and $\alpha\wedge \eta > -2(1-\frac1{p})$. Then, the operator $\cP_+^\gamma$ is a continuous linear map from $\cD^{\gamma,\eta,T}_{p}$ into $\cD^{\gamma',\eta',T}_{p}$ and we have for all $f\in\cD^{\gamma,\eta,T}_{p}$
$$ \cR \cP_+^\gamma f = P_+ * \cR f\;.$$
If $(\bar \Pi,\bar \Gamma)$ is another admissible model, then uniformly over all $f,\bar f$ in $\cD^{\gamma,\eta}_{p}, \bar{\cD}^{\gamma,\eta}_{p}$ we have
$$ \$ \cP_+^\gamma f, \cP_+^\gamma \bar f \$ \lesssim \| \Pi\| (1+\|\Gamma\|) \$f,\bar f\$ + (\|\Pi-\bar\Pi\| (1+\|\bar\Gamma\|) + \|\bar\Pi\|\|\Gamma-\bar\Gamma\|)\$ \bar f \$\;.$$
\end{theorem}
\begin{remark}
If $\nu := \min \cA$ and if $f\in \cD^{\gamma,\eta,T}_{p}$, then $\cP_+^\gamma f$ takes values in a sector of regularity $(\nu+2)\wedge 0$.
\end{remark}
\begin{remark}
Notice that in the original version of the convolution theorem~\cite[Th. 5.12]{Hairer2014}, the parameter $\eta$ is sent onto $(\eta\wedge\alpha) +2$ after convolution. As we will be working in situations where $\eta > \alpha$, our result provides a better weight index. Let us also mention that if we had chosen an $L^\infty$-norm in time in our spaces of modelled distributions (as this is the case in the original version of the theorem), then we would have improved the weight index by $2$ and not $2^-$.
\end{remark}

We turn to the convolution with the smooth part of the heat kernel. For every $f\in \cD^{\gamma,\eta,T}_{p}$, we set
$$ \cP_-^\gamma f(z) := \sum_{k\in \N^{d+1}:|k| < \gamma+2} \frac{X^k}{k!} \langle \cR f, \partial^k P_-(z-\cdot) \rangle\;.$$

\begin{theorem}[Convolution - smooth part]\label{Th:ConvSmooth}
In the context of Theorem \ref{Th:ConvSing}, the operator $\cP_-^\gamma$ is a continuous linear map from $\cD^{\gamma,\eta,T}_p$ into $\cD^{\gamma',\eta',T}_p$ and we have
$$ \cR \cP_-^\gamma f = P_- * \cR f\;.$$
\end{theorem}

We then define the operator $\cP^\gamma := \cP_+^\gamma + \cP_-^\gamma$ which is a continuous linear map from $\cD^{\gamma,\eta,T}_p$ into $\cD^{\gamma',\eta',T}_p$ such that
$$ \cR \cP^\gamma f = P * \cR f\;.$$

\subsubsection{Convolution of the shift}

For any function $h\in L^2((0,T)\times\T^d)$, we define
$$ \cP h (z) := \sum_{k\in\N^{d+1}:|k| < 2} \frac{X^k}{k!} \langle h, \partial^k P(z-\cdot) \rangle\;.$$
Although an operator $\cP$ acting on $\cD^{\gamma,\eta,T}_p$ was already introduced in the previous subsection, we prefer to keep the same notation for the present operator as they both refer to the convolution with the heat kernel.
\begin{lemma}\label{Lemma:ConvolShift}
For any $h\in L^2((0,T)\times\T^3)$ and any $\kappa > 0$, the restriction of $\cP h$ to $\cT_{< 2-\kappa}$ defines an element of $\cD^{2-\kappa,2-\kappa,T}_{2}$.
\end{lemma}

\section{Malliavin differentiability}\label{Sec:Diff}

\subsection{The Bouleau-Hirsch criterion}
Let $\Omega$ be a separable Banach space, let $\P$ be the law of a zero-mean Gaussian field on $\Omega$ and let $\cH$ be the associated Cameron-Martin space. We also let $\cF$ be the Borel $\sigma$-field associated with $\Omega$, completed with $\P$-null sets.

\begin{definition}\label{Def:Diff}
A random variable $X$ on $(\Omega,\cF,\P)$ is said to be locally $\cH$-differentiable if there exists an almost surely positive r.v.~$q$ such that $h\mapsto X(\omega+h)$ is Fr\'echet differentiable on $\{h\in\cH: \|h\|_\cH < q(\omega)\}$. For all $\omega$ such that $q(\omega) > 0$, we call $DX(\omega)$ the differential at $h=0$ of the above map.
\end{definition}

In our context, $\Omega$ is taken to be the H\"older space $\cC^\alpha((0,T)\times\T^d)$ with $\alpha < - 5/2$ and $\P$ is the law of the noise $\xi$.

We then have the following result, essentially due to Bouleau and Hirsch: we refer to~\cite[Section 2]{MalliavinReg} for details and references.
\begin{theorem}
Let $X$ be an $\R^n$-valued random variable on $(\Omega,\cF,\P)$. Assume that $X$ is locally $\cH$-differentiable and that $\P$ almost surely, $DX : \cH \to \R^n$ is onto.
Then $X$ admits a density w.r.t. Lebesgue measure on $\R^n$.
\end{theorem}

\subsection{Fixed points}

\subsubsection{Solution theory for $(\Phi^4_3)$}

Let us first recall the theory for the fixed point equation for $U$

\begin{equation} \label{eq:fpU}
U = \cP \left( - U^3 + \mathbf{1}_{\{t>0\}} \Xi \right) + G u_0
\end{equation}

Here $u_0 \in \cC^\eta(\T^3)$ is the initial condition, and $G u_0$ is the lift into the polynomial regularity structure of the convolution with the heat kernel of the initial condition:
$$ G u_0(t,x) := \sum_{k\in\N^{d+1}:|k| < \gamma'} \frac{X^k}{k!} \langle u_0, \partial^k P(t,x-\cdot) \rangle\;,$$
see~\cite[Lemma 7.5]{Hairer2014}.

By \cite[Proposition 9.8]{Hairer2014}, for any model $(\Pi,\Gamma)  \in \cM$, any $u_0 \in \cC^\eta(\T^3)$ with $\eta \in (-2/3,-1/2-\kappa)$ and any $\gamma > 1$, there exists a time $T_{\mbox{\tiny explo}} > 0$ such that \eqref{eq:fpU} admits a solution $U \in \cD^{\gamma,\eta,T}_\infty$ for each $T < T_{\mbox{\tiny explo}}$, and if $T_{\mbox{\tiny explo}} < +\infty$ one has
$$ \lim_{t \to T_{\mbox{\tiny explo}}} \left\|(\cR U)(t,\cdot)\right\|_{\cC^\eta(\T^3)} = +\infty . $$
In addition, the map $(\Pi,\Gamma) \mapsto U$ is continuous in the sense that if $\Pi^\varepsilon \to \Pi$ in $\cM$, then
$$\liminf_{\varepsilon} T_{\mbox{\tiny explo}}(\Pi^\varepsilon) \geq T_{\mbox{\tiny explo}}(\Pi)$$
and for each $T< T_{\mbox{\tiny explo}}(\Pi)$,
$$\lim_{\varepsilon} \$ U^\varepsilon ; U \$_{\cD^{\gamma,\eta,T}_{\infty}} =0,$$
where $U^\varepsilon = U(\Pi^\varepsilon)$, $U=U(\Pi)$.

Furthermore, if $\hat{\Pi}^\varepsilon$ is the model obtained from $\xi$ by renormalization and regularization as described in Section \ref{subsec:Model}, then one can show that $u_\varepsilon:= \cR U(\hat{\Pi}^\varepsilon)$ is the solution to
$$(\partial_t - \Delta) u_\varepsilon = C^\varepsilon u_\varepsilon - u_\varepsilon^3 + \xi_\varepsilon, \;\;\; u_\varepsilon(0,\cdot)=u_0,$$
with $C^\varepsilon = 3 C_1^\varepsilon - 9 C_2^\varepsilon$.

For a fixed sequence $(\varepsilon_k)_{k \geq 0}$ converging to $0$, we let
$$\Omega_0 := \left\{ \xi, \exists \Pi, \mbox{ s.t. }\;\;  \hat{\Pi}^{\varepsilon_k}(\xi) \to \Pi \mbox{ in } \cM\right\}$$
and note that by Theorem \ref{thm:model}, at least for a certain choice of $(\varepsilon_k)$, one has $\P(\Omega_0) = 1$.

We further let
$$\Omega_{0,T} = \Omega_0 \cap \left\{ T < T_{\mbox{\tiny explo}}(\Pi(\xi))\right\}$$
and we will assume throughout that
$$\P(\Omega_{0,T}) = 1.$$

We further let
$$\Omega_{1,T} := \left\{ \xi, \exists u, \mbox{ s.t. }\;\;  u_{\varepsilon_k} \to u \mbox{ in }C([0,T], C^\eta(\T^3))\right\}$$
and note that $\Omega_{0,T} \subset \Omega_{1,T}$. Throughout the rest of the paper, by $u$ we will mean the random variable defined on $\Omega_{1,T}$ as the limit of the $u_{\varepsilon}$'s and arbitrarily (say $0$) on $\Omega_{1,T}^c$.

\subsubsection{Shifted equations} \label{sec:shift}

From now on, we let $U_0 \in\cD^{\gamma_0,\eta_0,T}_\infty$ denote the solution of \eqref{eq:fpU} associated to a model $(\Pi,\Gamma)$ and an initial condition $u_0 \in \cC^{\eta_0}(\T^3)$ with
$$ \gamma_0 > 1\;, \quad -1/2-\kappa > \eta_0 > -2/3\;.$$
Notice that the homogeneity of the lowest level on which $U_0$ takes values is $-1/2 - \kappa$.

Our first goal is to extend the solution theory to the shifted equation and then to the tangent equation. We start with the former; we aim at solving the following equation:
\begin{equation}\label{Eq:Y}
Y_h = -\cP(3 U_0^2 Y_h + 3 U_0 Y_h^2 + Y_h^3) + \cP(h)\;.
\end{equation}
Notice that $U_h = U_0 + Y_h$ is then the solution of \eqref{eq:fpU} where $\Xi$ is replaced by $\Xi+h$.
\begin{proposition}\label{Prop:Y}
Take $\gamma \in (\frac74 + 2\kappa , 2-2\kappa)$. Then, for any admissible model $\Pi$ such that $T_{\mbox{\tiny explo}}(\Pi) > T$ there exists $q=q(\Pi) > 0$ such that for all $h$ in the ball of radius $q$ in $L^2(0,T)$, there exists a unique solution $Y_h \in \cD^{\gamma,\gamma,T}_{2}$ to \eqref{Eq:Y}. Furthermore, the map $(\Pi,h) \mapsto Y_h$ is locally Lipschitz.
\end{proposition}
Although the statement of this proposition is classical in the framework of the theory of regularity structures, the proof presents a specific difficulty. Indeed, we are working with spaces of modelled distributions of $L^p$-type in space \textit{and} time so that to evaluate the solution at some given time $t$, one needs to embed it into a space of modelled distributions of $L^\infty$-type in time: this decreases tremendously the regularity of the corresponding function of space, and unfortunately, prevents us from iterating a fixed point argument. To circumvent this difficulty, we build the iterations in a different manner: after having obtained a fixed point on $(0,T_*)$ for some $T_* >0$, we iterate the fixed point argument on the interval $(T_*/2,3T_*/2)$ but with a slightly different map which takes into account the fixed point already obtained on $(0,T_*)$. The key point is that everything depends continuously on the $L^2$-norm of the shift $h$, so that reducing the latter one can always obtain contractivity.
\begin{proof}
For $S \le T$, consider the map
$$ \cM_S: Y \mapsto -\cP(3 U_0^2 Y + 3 U_0 Y^2 + Y^3) + \cP(h)\;.$$
Let $\cY,\cR_Y$ be the smallest sets of symbols such that $X^k \in \cY$ for all $k\in \N^4$ and, for all $\tau_1,\tau_2,\tau_3 \in \cU$ and all $\rho_1,\rho_2,\rho_3 \in \cY$, we have
$$ \tau_1 \tau_2 \rho_1 \in \cR_Y\;,\quad  \tau_1 \rho_1 \rho_2 \in \cR_Y\;,\quad \rho_1\rho_2\rho_3 \in\cR_Y\;,$$
and for all $\rho \in \cR_Y$ we have $\cI(\rho) \in \cY$. It is simple to check that $\cY$ then contains only symbols with non-negative homogeneities. The proof is now split into five steps.\\

\textit{Step 1.} Let us show that $\cM_S$ goes from $\cD^{\gamma,\gamma,S}_{2}$ into $\cD^{\gamma+\eps,\gamma+\eps,S}_{2}$ for some $\eps>0$. To that end, we look at every single term appearing in $\cM_S(Y)$ and show that it belongs to some $\cD^{\gamma',\eta',S}_2$ with $\gamma' >\gamma$ and $\eta' > \gamma$. Since $\cD^{\gamma',\eta',S}_2$ can be continuously embedded into $\cD^{\gamma'',\eta'',S}_2$ if $\gamma''\le\gamma'$, $\eta'' \le \eta'$ and $\gamma'' \notin \cA$, this is enough to obtain the desired property.\\

Applying successively the Embedding Theorem (Th. \ref{Th:EmbeddingW}), the Multiplication Theorem (Th. \ref{Th:Mult}) and the Convolution Theorems (Th. \ref{Th:ConvSing} and \ref{Th:ConvSmooth}), we obtain the following: (for the sake of readability, we drop the superscript $S$)
$$ Y \in \cD^{\gamma,\gamma}_{2} \Rightarrow Y \in \cD^{\gamma-\frac{5}{3}-\kappa,\gamma-\frac{5}{3}-\kappa}_{6} \Rightarrow Y^3 \in \cD^{\gamma-\frac{5}{3}-\kappa,3\gamma-5-3\kappa}_{2} \Rightarrow \cP(Y^3) \in \cD^{\gamma + \frac{1}{3}-\kappa,3\gamma-3-4\kappa}_{2}\;,$$
as well as
\begin{align*}
Y \in \cD^{\gamma,\gamma}_{2} &\Rightarrow Y\in \cD^{\gamma-\frac5{4}-\kappa,\gamma-\frac5{4}-\kappa}_{4} \Rightarrow Y^2 \in \cD^{\gamma-\frac5{4}-\kappa,2\gamma-\frac52-2\kappa}_{2}\\
&\Rightarrow Y^2 U_0 \in \cD^{\gamma-\frac{7}{4}-2\kappa,2\gamma-\frac52-2\kappa+\eta_0}_{2} \Rightarrow \cP(Y^2 U_0) \in \cD^{\gamma+\frac1{4}-2\kappa,2\gamma-\frac12-3\kappa+\eta_0}_{2}\;,
\end{align*}
and
$$ Y \in \cD^{\gamma,\gamma}_{2} \Rightarrow YU_0^2 \in \cD^{\gamma-1-2\kappa,\gamma+2\eta_0}_{2} \Rightarrow \cP(YU_0^2) \in \cD^{\gamma+1-2\kappa,\gamma+2\eta_0 + 2 - \kappa}_{2}\;.$$
Finally, $\cP(h) \in \cD^{2-\kappa,2-\kappa}_{2}$ by Lemma \ref{Lemma:ConvolShift}.\\
This shows that $\cM_S$ goes from $\cD^{\gamma,\gamma,S}_{2}$ into $\cD^{\gamma+\eps,\gamma+\eps,S}_{2}$ for some $\eps>0$. Notice that the embedding from $\cD^{\gamma+\eps,\gamma+\eps,S}_{2}$ into $\cD^{\gamma,\gamma,S}_{2}$ has a norm of order $S^\epsilon$: we will use this fact below to get a fixed point.

\medskip

\textit{Step 2.}  In the forthcoming equations, $\$\cdot\$$ will refer to the $\cD^{\gamma,\gamma,S}_2$-norm. By the analytical results of Subsection \ref{Subsec:Dgamma}, there exists $C>0$ such that
$$ \$ \cM_{S}(Y)\$ \le C S^\epsilon \big( \$Y\$ + \$Y\$^2 + \$Y\$^3 + \| h \|_{L^2((0,S)\times\T^3)}\big)\;,$$
and
$$ \$ \cM_{S}(Y)-\cM_{S}(Y')\$ \le C S^\epsilon \$Y-Y'\$ \big( 1 + \$Y-Y'\$ + \$Y-Y'\$^2\big)\;,$$
uniformly over all $S \le T$ and all $Y,Y' \in \cD^{\gamma,\gamma,S}_2$. Consequently, for any $R>0$, one can choose $T_*$ and $q$ small enough so that $\cM_{T_*}$ maps the centered ball of radius $R$ into the centered ball of radius $R/2$ in $\cD^{\gamma,\gamma,T_*}_{2}$ and is $1/2$-Lipschitz there, uniformly over all $\|h\|_{L^2(0,T)} \le q$. Fix $R >0$ and let $Y_* = Y_*(h)$ be the corresponding fixed point for any $h$ such that $\|h\|_{L^2} \le q$. A simple computation shows that
$$ \$ Y_*(h) \$ \le C' {T_*}^\epsilon \|h\|_{L^2}\;.$$

%We can also ask for such a property to be true not only on the interval $[0,T_*]$ but on any interval of length $T_*$ in the compact set of $[0,T]$ (this requires to take bounds on $U_0$ and on the model on the whole compact set).
%we deduce that for any two admissible models $(\Pi,\Gamma)$ and $(\bar\Pi,\bar\Gamma)$, we have
%\begin{align*}
%\$ \cM_T(W) \$ \le C \big(T^\epsilon + \|h\|_{L^2((0,T)\times\T^3)}\big)\;,\\
%\$ \cM_T(W) ; \cM_T(\bar W) \$ \le C T^\epsilon(\$ W ; \bar W\$ + \| \Pi-\bar\Pi\| + \|\Gamma-\bar\Gamma\|)\;,
%\end{align*}
%uniformly over all $W,\bar W$ and all $(\Pi,\Gamma), (\bar\Pi,\bar\Gamma)$ whose norms are smaller than some constant $R>0$.
\textit{Step 3.} We ``extend" $Y^*$ into an element $Y^{\mbox{\tiny ext}}$ of $\cD^{\gamma,\gamma,3T_*/2}_{2}$ that satisfies:
$$ Y^{\mbox{\tiny ext}}(t,\cdot) = Y_*(t,\cdot)\;,\quad \forall t\in (0,2T_*/3)\;,$$
and
$$ Y^{\mbox{\tiny ext}}(t,\cdot) = 0\;,\quad \forall t\in (3T_*/4,3T_*/2)\;.$$
To do so, we consider a smooth function of $t$ which equals $1$ on $(0,2T_*/3)$ and $0$ after $3T_*/4$, and we lift it into the polynomial regularity structure up to level $\lfloor \gamma + 2 \rfloor$. The $\cD^{\gamma+2,0,3T_*/2}_{\infty}$-norm of such a function is of order $T_*^{-\gamma-2}$ and we can apply Theorem \ref{Th:Mult} to take the product of this function with $Y_*$. Notice that $\$Y^{\mbox{\tiny ext}}\$ \lesssim T_*^{-\gamma-2} \$ Y_*(h) \$$.\\

\textit{Step 4.} Let us now iterate this fixed point procedure. We introduce the space $\hat{\cD}^{\gamma,\gamma,T_*}_{2}$ of maps on $(T_*/2,3T_*/2)\times\T^3$ which vanish on $(T_*/2,2T_*/3)\times\T^3$ and satisfy the bounds of the ${\cD}^{\gamma,\gamma,T_*}_{2}$-norm but shifted by $T_*/2$ in time: in particular the weights are given by powers of $t-T_*/2$ instead of $t$. We then consider the following map defined on $\hat{\cD}^{\gamma,\gamma,T_*}_{2}$:
\begin{align*}
\hat{\cM}_{T_*}:Y \mapsto &- \cP(3 U_0^2 Y + 3 U_0 Y^2 + Y^3) - \cP(6 U_0Y^{\mbox{\tiny ext}}Y + 3 {Y^{\mbox{\tiny ext}}}^2 Y + 3 Y^{\mbox{\tiny ext}} Y^2 ) + \cP(h\tun_{t\ge 2T_*/3})\\
&-\cP(3 U_0^2 Y^{\mbox{\tiny ext}}  + 3 U_0 {Y^{\mbox{\tiny ext}}}^2 + {Y^{\mbox{\tiny ext}}}^3) - Y^{\mbox{\tiny ext}} + \cP(h\tun_{t< 2T_*/3})\;.
\end{align*}
Notice that the second line vanishes on $(T_*/2,2T_*/3]$ since $Y^{\mbox{\tiny ext}}$ coincides with the fixed point $Y_*$ of $\cM_{T_*}$ on this interval. The map $\hat{\cM}_{T_*}$ then takes values in $\hat{\cD}^{\gamma,\gamma,T_*}_{2}$.\\
Up to diminishing the value of $q$, one can check that for any arbitrary $\delta > 0$, $\hat{\cM}_{T_*}$ maps the centred ball of radius $R$ in $\hat{\cD}^{\gamma,\gamma,T_*}_{2}$ into the centered ball of radius $R/2 + \delta$ in $\hat{\cD}^{\gamma,\gamma,T_*}_{2}$ and is $(1/2 + \delta)-$Lipschitz there. Indeed, compared to $\cM_{T_*}$, the norms of the additional terms appearing in $\hat{\cM}_{T_*}$ all depend on the $L^2$-norm of $h$ so that their contributions can be made as small as desired by simply diminishing the latter.\\
This yields a fixed point $Y_{**}$. Since it vanishes on $(T_*/2,2T_*/3]$, it can be extended into an element of $\cD^{\gamma,\gamma,3T_*/2}_{2}$ by simply setting its value to $0$ before time $T_*/2$, and one can check that $Y_{\mbox{\tiny ext}} + Y_{**}$ is a fixed point of the map $\cM_{3T_*/2}$. Iterating this procedure $k$ times, one obtains a fixed point on the interval $[0,(k+1)T_*/2]$ so that we can get as close as desired to time $T$.\\

\textit{Step 5.} The lower-semicontinuity of the maximal time $(\Pi,\Gamma)\mapsto T$ ensures that we can find a neighbourhood of $(\Pi,\Gamma)$ where the maximal time is uniformly larger than $T$. The local Lipschitz continuity of the solution map $(\Pi,\Gamma,h)\mapsto Y_h$ is then a consequence of the bounds obtained in Theorems \ref{Th:ReconstructionW}, \ref{Th:EmbeddingW}, \ref{Th:Mult} and \ref{Th:ConvSing} applied to the fixed points associated with the model $(\Pi,\Gamma)$ and some close-by model $(\bar\Pi,\bar\Gamma)$.
\end{proof}
In the sequel, we let $U_h=U_0+Y_h$. As these two terms do not live in the same space, we will need to treat them separately in the next analytical bounds.

%
%For any given $h_0,h\in L^2$, we consider the equation
%\begin{equation}\label{Eq:Y2}
%Y = -\cP(3 U_{h_0}^2 Y + 3 U_{h_0} Y^2 + Y^3) + \cP(h)\;.
%\end{equation}
%
%\begin{proposition}\label{Prop:Y2}
%Take $\gamma \in (\frac74 + 2\kappa , 2-2\kappa)$. Then, for any admissible model $\Pi$ such that $T_{\mbox{\tiny explo}}(\Pi) \ge T$, for all $h_0$ in the centered ball of radius $q$ (where $q$ is taken from the previous proposition) and for all $h$ in the centered ball of radius $\delta=q-\|h_0\|$, there exists a unique solution $Y_{h_0,h} \in \cD^{\gamma,\gamma,T}_{2}$ to \eqref{Eq:Y2}. Furthermore, the map $(\Pi,h_0,h) \mapsto Y_{h_0,h}$ is locally Lipschitz.
%\end{proposition}
%\begin{proof}
%It suffices to observe that $Y_{h_0,h}$ is solution to \eqref{Eq:Y2} if and only if $Y_{h_0}+Y_{h_0,h}$ is solution to the equation \eqref{Eq:Y} associated to the shift $h_0+h$. The statement of the proposition follows from Proposition \ref{Prop:Y}.
%\end{proof}
%
%We further let $U_{h_0,h} := U_{h_0} + Y_{h_0,h}$ and
We then have the following consistency result:

\begin{proposition}
For all $\xi$ $\in$ $\Omega_{0,T}$, for all $h$ $\in$ $\cH$ s.t. $\|h\|< q(\xi)$ where $q$ is given by Proposition \ref{Prop:Y} for the model $\Pi(\xi)$, one has $\xi+h \in \Omega_{1,T}$ and 
$$u(\xi+h) = \cR U_{h} (\Pi(\xi))).$$
\end{proposition}

\begin{proof}
By applying the same arguments identifying the equation satisfied by $u_\varepsilon$ (cf. \cite[Proposition 9.10]{Hairer2014}), one can show that $u_{h}^\varepsilon := \cR U_{h} (\hat{\Pi}^\varepsilon)$ satisfies the equation
$$(\partial - \Delta) u_{h}^\varepsilon = - (u_{h}^\varepsilon)^3 + (3 C_1^\varepsilon - 9 C_2^\varepsilon) u_{h}^\varepsilon + \xi_{\varepsilon} + h, \;\;\;\;\; u_{h}^\varepsilon(0,\cdot)= u_0.$$
Comparing with the equation satisfied by $u_\varepsilon$, we obtain that
$$u_\varepsilon(\xi+h) = u^\varepsilon_{h^\varepsilon}$$
where $h^\varepsilon = h \ast \rho^\varepsilon$. Taking the limit and using continuity of $(\Pi,h) \mapsto U_{h}$, we obtain the result.
\end{proof}

We then consider the \textit{tangent equation}
\begin{equation}\label{Eq:V}
V := \cP(-3U_{h_0}^2V) + \cP h\;.
\end{equation}

\begin{proposition}\label{Prop:V}
Take $\gamma \in (\frac74 + 2\kappa , 2-2\kappa)$. Then, for any admissible model $\Pi$ such that $T_{\mbox{\tiny explo}}(\Pi) \ge T$ for all $h_0$ in the ball of radius $q$ in $L^2(0,T)$ and for all $h\in L^2(0,T)$, there exists a unique solution $V_{h_0,h} \in \cD^{\gamma,\gamma,T}_{2}$ to \eqref{Eq:V}. Furthermore, the map $(\Pi,h_0,h) \mapsto V_{h_0,h}$ is locally Lipschitz.
\end{proposition}
\begin{proof}
The arguments are essentially the same as those presented in the proof of Proposition \ref{Prop:Y}. The main difference is that \eqref{Eq:V} is linear in $V$ so that the solution is itself linear in $h$: therefore, there is no constraint on the $L^2$-norm of $h$ for existence of solutions.
\end{proof}

We also identify the equation satisfied by $V_{h_0,h}$ in the case of a regularized model :

\begin{proposition} \label{prop:vheps}
$v^{h_0,h,\varepsilon}:= \cR V_{h_0,h} (\hat{\Pi}^\varepsilon)$ satisfies
\begin{equation} \label{eq:vheps}
(\partial_t - \Delta) v^{h_0,h,\varepsilon} = -3 (u^{h_0,\varepsilon})^2 v^{h_0,h,\varepsilon} + (3C_1^\varepsilon - 9 C_2^\varepsilon) v^{h_0,h,\varepsilon} + h, \;\;\;\; v^{h_0,h,\varepsilon}(0,\cdot) = 0.
\end{equation}
\end{proposition}
\begin{proof}
We fix $h_0$ and $h$, and for ease of notation we let $U=U^\varepsilon_{h_0}, V= V^\varepsilon_{h_0,h}$ and $u$, $v$ their respective reconstructions.
Recall that from the definition of $U$ one gets that 
$$U = \<1> + u_{\mathbf{1}} \mathbf{1} - \<30> + (>1-),$$
where by $(>1-)$ we mean a sum of symbols of degree $1-$ or higher. In particular
$$u = (P \ast \xi_\varepsilon) + u_{\mathbf{1}}$$
and also 
$$U^2 = \<2> +2  u_{\mathbf{1}} \<1> - 2\<31> + u_{\mathbf{1}}^2 \mathbf{1} + (>0).$$
From \eqref{Eq:V} we get
$$V = v \mathbf{1} - 3 v \<20> + \sum_{i=1}^3 v_{i} X_i + (>3/2-),$$
and 
$$-3U^2 V= -3 v  \<2> - 6 u_{\mathbf{1}} v  \<1> + 6 v  \<31> - 3u^2_{\mathbf{1}} v \mathbf{1} + 9 v \; \<22> - 3\sum_{i=1}^3 v_{i} X_i  \<2> + (>0).$$
Since $\cR(-3U^2 V)(z) = \hat{\Pi}^\varepsilon_z(-3 U^2 V)(z)$, using \eqref{eq:PiM} we obtain
$$\cR(-3U^2 V) = -3 v \left( (P\ast \xi_{\varepsilon})^2 - C^\eps_1 \right) - 6 u_{\mathbf{1}} v (P\ast \xi_\varepsilon) - 3 u^2_{\mathbf{1}} v - 9 v C_2^\varepsilon = - 3 u^2 v + 3 C^\varepsilon_1 v - 9 C_2^\varepsilon v$$
and since $(\partial_t-\Delta)(\cR \cP f) = \cR f$ for any modelled distribution $f$, it follows that $\cR V$ satisfies \eqref{eq:vheps}.
\end{proof}

\subsection{Malliavin differentiability of the solutions}

Let $\varphi_i$, $i=1,\ldots,n$ be functions in the Besov space $\cB^{1/2+}_{1}$ and assume that they are compactly supported in $(0,T)$.

\begin{remark}
The regularity of the test functions has to be larger than $1/2$ by a quantity related to $\kappa$: we prefer not to write precisely the relationship with $\kappa$ to avoid dealing with many different constants times $\kappa$ in the arguments.
\end{remark}

We claim that for all $\xi \in \Omega_{0,T}$, for all $h$ lying in the ball of radius $q(\xi)$, the random variable $\langle u(\xi+h), \varphi_i\rangle$ makes sense. Indeed, recall that we have $u(\xi+h) = \cR U_0 + \cR (Y_{h})$. The first term lies in the space $\cB^{-1/2-,T}_{\infty}$ (at least far from $t=0+$) so that it can be tested against $\varphi_i$ (which vanishes near $t=0+$). On the other hand, $Y_{h} \in \cD^{2-,2-,T}_2 \subset \cD^{0+,0+,T}_{10-}$ so that $\cR (Y_{h}) \in \cB^{0+,T}_{10-}$ while $\varphi_i \in \cB^{1/2+}_{1} \subset \cB^{0+}_{\frac{10}{9}+}$ so that $\langle \cR (Y_{h}) , \varphi_i\rangle$ makes sense.
\begin{proposition}
The random variable $\xi\mapsto (\langle u(\xi),\varphi_1\rangle,\ldots,\langle u(\xi),\varphi_n\rangle)\in \R^n$ is locally $\cH$-differentiable in the sense of Definition \ref{Def:Diff}.
\end{proposition}
%The proof is an adaptation of~\cite[Prop 4.7]{MalliavinReg}.
\begin{proof}
Fix $\xi\in \Omega_{0,T}$. For any $h_0$ whose $L^2((0,T)\times \T^d)$ norm is smaller than the parameter $q(\xi)$ of Proposition \ref{Prop:Y}, we introduce the map
\begin{align*}
F: L^2 \times \cD^{\gamma,\gamma,T}_2&\longrightarrow\quad \cD^{\gamma,\gamma,T}_2\\
(h,Y) &\longmapsto\quad Y + \cP\big((Y+U_{h_0})^3 - U_{h_0}^3\big) - \cP(h)\;.
\end{align*}
It can be checked that $F$ is Fr\'echet differentiable and that its partial derivatives satisfy:
\begin{align*}
\partial_1 F (h,Y)(g) &= - \cP g\;,\\
\partial_2 F (h,Y)(Z) &= Z + 3 \cP( (Y+U_{h_0})^2 Z)\;.
\end{align*}
Notice that the differential depends continuously on $h,Y$. Let us prove that $Z\mapsto \partial_2 F (h,Y)(Z)$ is a bounded, linear isomorphism.\\
The linearity is immediate and the boundedness is a consequence of the analytic results on products and convolution, see Theorems \ref{Th:Mult}, \ref{Th:ConvSing} and \ref{Th:ConvSmooth}. To prove that $Z\mapsto \partial_2 F (h,Y)(Z)$ is a bijection, we proceed as follows. Let $X\in\cD^{\gamma,\gamma,T}_2$ be given and consider the fixed point equation:
\begin{equation}\label{Eq:Z}
Z = -3 \cP((Y+U_{h_0})^2 Z) + \lambda X\;.\end{equation}
If we prove that this equation admits a unique solution for all $\lambda > 0$ small enough, then the linearity in $Z$ of the equation ensures that uniqueness holds for all $\lambda > 0$ and, in turn, we deduce that $Z\mapsto \partial_2 F (h,Y)(Z)$ is a bijection (surjectivity follows from the existence of solutions, injectivity follows from the uniquess of the solution for $X=0$). The proof of the claim then follows from exactly the same arguments as in the proof of Proposition \ref{Prop:Y}: the tuning parameter in that proof was the $L^2$-norm of $h$ while, in the present case, one decreases the parameter $\lambda$ in order to get contractivity of the solution map.\\
We know that $F(0,0) = 0$ so that we can apply the Implicit Function Theorem that ensures the existence of a ball $B(0,\epsilon)$ in $L^2((0,T)\times \T^d)$ and of a differentiable function $\theta : B(0,\epsilon) \mapsto \cD^{\gamma,\gamma,T}_2$ such that $F(h,\theta(h)) = 0$ for all $h \in B(0,\epsilon)$. Furthermore,
\begin{equation}\label{Eq:Difftheta} D\theta(h) = - \Big(\partial_2 F (h,\theta(h))\Big)^{-1} \big( \partial_1 F(h,\theta(h))\big)\;.\end{equation}
Comparing the identities $F(h,\theta(h)) = 0$ and \eqref{Eq:Y}, we deduce from the uniqueness part of Proposition \ref{Prop:Y} that $\theta(h) = Y_{h_0+h}-Y_{h_0}$ for all $h\in B(0,\epsilon)$. Similarly, comparing the identities \eqref{Eq:Difftheta} and \eqref{Eq:V}, we deduce from the uniqueness part of Proposition \ref{Prop:V} that $D\theta(0)(h)=V_{h_0,h}$ for all $h\in L^2(0,T)$. Putting everything together, we get
\begin{align*}
\$ Y_{h_0, h} - V_{h_0,h} \$ = \$ \theta(h) - \theta(0) - D\theta(0)(h) \$ = o(\|h\|)\;
\end{align*}
uniformly over all $h$ whose $L^2$-norm is smaller than $\epsilon$.\\
Recall that $u(\xi+h_0+h)= \cR U_0 + \cR Y_{h}$ and that $v(\xi)^{h_0,h} = \cR V_{h_0,h}$. Consequently,
$$ u(\xi+h_0+h) - u(\xi+h_0) - v(\xi)^{h_0,h} = \cR(Y_{h_0+h}-Y_{h_0} - V_{h_0,h})\;.$$
Composing with the continuous map
$$ f \mapsto (\langle f,\varphi_1\rangle, \ldots,\langle f,\varphi_n\rangle)\;,$$
this yields the asserted differentiability at any point $\xi+h_0$ such that $\xi\in \Omega_{0,T}$ and the $L^2$-norm of $h_0$ is smaller than $q(\xi)$. The differential at $\xi+h_0$ is equal to
$$ h\mapsto \Big(\langle v(\xi)^{h_0,h}, \varphi_1 \rangle ,\ldots, \langle v(\xi)^{h_0,h}, \varphi_n\rangle\Big)\;.$$
By Proposition \ref{Prop:V} this map is continuous in $h_0$ (within the ball of radius $q(\xi)$), we deduce the asserted Fr\'echet differentiability, thus concluding the proof.
\end{proof}

\section{Existence of densities}\label{Sec:Densities}

We want to study the density of $X = \left( \langle u, \phi_1 \rangle, \ldots, \langle u, \phi_n  \rangle \right)$ with $\phi_1, \ldots, \phi_n$ linearly independent elements of $\cB^{1/2+}_{1}$ with support in $(0,T) \times \T^3$.

Recall that by the results of the previous section, the Malliavin derivative of $X$ is given by
$$ DX:h \mapsto \left( \langle v^h, \phi_1 \rangle, \ldots, \langle v^h, \phi_n  \rangle \right)$$
where $v^h= \cR V^h$ with $V^h=V_{0,h}$ defined as above.

The a.s. surjectivity of $DX$ is then equivalent to proving that a.s. one has
$$\langle v^h, \sum_{i} \lambda_i \phi_i \rangle = 0\quad \forall h \in \mathcal{H} \;\; \Longrightarrow \lambda_1=\cdots=\lambda_n=0.$$

\subsection{The dual equation}

Fix $T>0$. We need to introduce spaces of modelled distributions which are amenable to solving SPDEs going backward in time from time $T$.\\
For $\gamma>0$, $\eta \le \gamma$, $p\in [1,\infty]$ and $T' \le T$, we introduce the space $\tilde{\cD}^{\gamma,\eta,T'}_p$ as the set of all modelled distributions $f:(T',T)\times\T^d\rightarrow \cT_{<\gamma}$ that satisfy the bounds \eqref{Eq:BdDgamma} with $(0,T)$ and $(3|h|^2,T-|h|^2)$ replaced by $(T',T)$ and $(T'-|h|^2,T-3|h|^2)$, and with the weights $t^{(\eta-\zeta)/2}$ and $t^{(\eta-\gamma)/2}$ replaced by $(T-t)^{(\eta-\zeta)/2}$ and $(T-t)^{(\eta-\gamma)/2}$. The calculus presented in Subsection \ref{Subsec:Dgamma} finds naturally its counterpart backward in time: we denote by $\tilde{\cR}$ and $\tilde{\cP}$ the associated operators acting on the spaces $\tilde{\cD}$.

For $\delta >1/2$, let $\varphi\in \cB^{\delta}_1(\R\times\T^3)$ be a function whose support lies in $(0,T) \times \T^3$. One can naturally lift this function into the polynomial regularity structure by setting
$$ \Phi(z) := \sum_{k\in \N^{d+1}:|k| < \delta} \frac{X^k}{k!} \partial^k \varphi(z)\;.$$
This defines an element of $\tilde{\cD}^{\delta,\delta,T}_{1}$: notice that we can take the parameter $\eta$ in this space to be equal to $\delta$ since our function $\varphi$ vanishes before time $T$. Applying Theorems \ref{Th:ConvSing} and \ref{Th:ConvSmooth}, we deduce that $\tilde{\cP} \Phi \in \cD^{\delta+2,\delta+2-\kappa,T}_{1}$.

We want to study the backward equation which is dual to that of $V^h$,  given by
\begin{equation} \label{eq:W}
W = \tilde{\cP} (-3 U_0^2 W) + \tilde{\cP} \Phi\;.
\end{equation}
(Note that $U_0$ still denotes the solution to the \emph{forward} equation \eqref{eq:fpU}).
\begin{proposition}
Let $\Pi$ be an admissible model such that $T_{\mbox{\tiny explo}}(\Pi) \ge T$ and set $\tilde\gamma=\tilde\eta=2+\delta-\kappa$. For every $T'\in (0,T)$, there exists a unique solution $W\in \tilde{\cD}^{\tilde\gamma,\tilde\eta,T'}_1$ to \eqref{eq:W}. Furthermore, the map $\Pi\mapsto W$ is locally Lipschitz. Finally, we have the following estimate:
$$ \$ W\$_{\tilde{\cD}^{\tilde\gamma,\tilde\eta,T'}_1} \lesssim (T')^{\eta_0-\gamma_0}\;,$$
uniformly over all $T'\in (0,T)$.
\end{proposition} 

\begin{proof}
Observe that the restriction of $U_0$ to $(T',T)$ belongs to $\tilde{\cD}^{\gamma_0,\eta_0,T'}_\infty$ and that its norm satisfies:
$$ \$U_0\$_{\tilde{\cD}^{\gamma_0,\eta_0,T'}_\infty} \lesssim (T')^{\frac{\eta_0-\gamma_0}{2}} \$U_0\$_{\cD^{\gamma_0,\eta_0,T}_\infty}\;,$$
uniformly over all $T' \in (0,T)$. Indeed, the norm on the left hand side is not weighted near $0$, while the norm on the right hand side is weighted by exponents no worse than $\frac{\eta_0-\gamma_0}{2}$.  \\
Applying the analytical results of Subsection \ref{Subsec:Dgamma}, we deduce the following:
$$ W \in \tilde{\cD}^{2+\delta-\kappa,2+\delta-\kappa,T'}_1 \Rightarrow W U_0^2 \in \tilde{\cD}^{2+\delta-1-3\kappa,2+\delta-\kappa+2\eta_0,T'}_1 \Rightarrow \tilde{\cP}(W U_0^2) \in \tilde{\cD}^{2+\delta+1-3\kappa,4+\delta-2\kappa+2\eta_0,T'}_1\;,$$
so that the map $\cM_{T'}:W \mapsto \tilde{\cP} (-3 U_0^2 W) + \tilde{\cP} \Phi$ goes from $\tilde{\cD}^{\tilde\gamma,\tilde\eta,T'}_1$ into $\tilde{\cD}^{\tilde\gamma+\epsilon,\tilde\eta+\epsilon,T'}_1$ for some $\epsilon > 0$.\\
At this point, we introduce a parameter $\lambda >0$ and consider the modified map $\cM_{T'}:W \mapsto \tilde{\cP} (-3 U_0^2 W) + \lambda \tilde{\cP} \Phi$. Following the same steps as in the proof of Proposition \ref{Prop:Y}, we obtain a fixed point to this map: the only difference is that instead of decreasing the norm of $h$ in order to get contractivity, here we decrease the value of $\lambda$. Hence we obtain a unique solution with $\Phi$ replaced by $\lambda \Phi$ in the definition of $\cM_{T'}$ for some $\lambda > 0$ potentially very small. Since the equation is actually linear in $W$, we easily recover the solution for $\Phi$.
\end{proof}

Since $W$ takes values in non-negative levels only, $\tilde\cR W$ is a function on $(T',T)\times\T^3$ with $T' > 0$ arbitrarily small: it therefore defines a function on $(0,T)\times\T^3$ and we have the following result.

\begin{proposition} \label{prop:dual}
Let $v^h = \cR V_h$, $w^\phi = \cR W$, then it holds that
\begin{equation} \label{eq:duality}
\forall (h,\phi) \in L^2 \times \cB_{1}^{\frac12 +}\;, \quad \langle v^h, \phi \rangle = \langle h , 1_{(0,T)} w^\phi \rangle.\end{equation}
In particular, $w^\phi$ is in $L^2((0,T)\times\T^3)$.
\end{proposition}

\begin{proof}
By continuity of everything w.r.t the model, it is enough to prove this identity for the approximating smooth renormalized models. Recall that by Proposition \ref{prop:vheps} we know that in that case the equation satisfied by $v^h$ is given by
$$(\partial_t - \Delta) v^h_\varepsilon =  (-3 u_\varepsilon^2 + C_\varepsilon) v^h_\varepsilon + h, \;\;\; v^h(0,\cdot)=0$$
with $C_\varepsilon = 3 C^\varepsilon_1 - 9 C^\varepsilon_2$. To identify the equation satisfied by $w^\phi$, we proceed as in the proof of Proposition \ref{prop:vheps} and obtain that
$$W = w^\phi \mathbf{1} - 3 w^\phi\; \<2c0> + \sum_i w_{i} X_i + (>3/2-)$$
and
$$- 3 U^2 W = -3 w  \<2> - 6 u_{\mathbf{1}} w\;  \<1> + 6 w \; \<31> - 3u^2_{\mathbf{1}} w \mathbf{1} + 9 w\;  \<2c2> - 3\sum_{i=1}^3 v_{i} X_i  \<2> + (>0).$$
As in the proof of Proposition \ref{prop:vheps} this yields that $w$ satisfies
$$(-\partial_t - \Delta) w^\phi =  (-3 u_\varepsilon^2 + \tilde{C}_\varepsilon) w^\phi + \phi, \;\;\; w^\phi(T,\cdot)=0.$$
where $\tilde{C}_\varepsilon = 3 C_1^\varepsilon - 9 C_3^\varepsilon = C_\varepsilon$ (recalling \eqref{eq:C2C3}), and from there a standard integration by parts leads to \eqref{eq:duality}.
\end{proof}

%
%Then we transform this classicaly using that by integration by parts
%$$\langle v^h, \phi \rangle = \langle h ,w^\phi \rangle$$
%where
%$$(-\partial_t -\partial_{xx}) w^\phi = f'(u) (u_x)^2 w^\phi - \partial_x( (2 f(u)  u_x) w^\phi) + \phi, \;\;\; w^\phi(T,\cdot)=0,$$
%which we rewrite as an equation for a modelled distribution $W$ :
%\begin{equation} \label{eq:W}
%W = \overleftarrow{P} (A W + B \nabla W) + \overleftarrow{P} \phi,
%\end{equation}
%$\overleftarrow{P}$ being the abstract integration operator associated to the backward heat kernel.
%Here we have noted
%$$A = F'(U) \nabla U - 2 F(U) \nabla^2 U, \;\;\; B=- 2 F(U) \nabla U$$
%which are suitable elements of some $\mathcal{D}^\gamma$ spaces.
%
%(Note : here since we consider $\nabla^2 U$ which does not appear in the equation for $U$, we need to consider an extended regularity structure containing also $\mathcal{U} \cdot \nabla^2 \mathcal{U}$, where $\mathcal{U}$ are all the symbols used to describe $U$. Hopefully this is not a problem ?? Actually : also the backward heat kernel would also introduce a different integration operator etc......)
%

\subsection{Existence of densities}

We will now prove the main result of this paper, Theorem \ref{Th:Main}. The proof relies on the following crucial result.
%
%\begin{theorem} \label{thm:main}
%Let $\xi$ satisfy Assumption \ref{asn:R} and either assumption \ref{asn:Dense} or assumption \ref{asn:Rough}. Fix $T>0$ and assume that $T< T_{\mbox{\tiny explo}}$ almost surely. Let $\phi_1,\ldots,\phi_n$ be linearly independent elements of $\cB^{1/2+\kappa}_1$ for some $\kappa>0$, supported in $(0,T) \times \T^3$. Then the law of the random variable 
%$$\left(\left\langle u,\phi_1 \right\rangle, \ldots, \left\langle u, \phi_n \right\rangle \right)$$
%is absolutely continuous w.r.t. Lebesgue measure.
%\end{theorem}

\begin{proposition} \label{prop:wW}
Let $\phi$ $\in$ $\cB^{1/2+}_1$ supported in $(0,T)\times\T^3$, $W$ the solution to the fixed point equation \eqref{eq:W} and $w = \cR W$. If $w= 0$ a.e. on $[0,T] \times \T^3$, then one has $\phi =0$.
\end{proposition}

\begin{proof}
Assume that $W=0$, then by \eqref{eq:W} we obtain $\tilde P*\phi=0$, hence $\phi=0$. We therefore only need to prove that $W=0$ as soon as $w = \mathcal{R} W = \langle W,\mathbf{1} \rangle=0$, which we now assume.

Let $$\delta = \inf\{|\tau|, \tau \in \mathcal{W} \;\; | \langle W,\tau \rangle| \mbox{ is not a.e. } 0\}$$
and assume $\delta <\gamma$. Each symbol in $\mathcal{W}$ which is not a polynomial is of the form $\tau = \tcI(\rho_1 \rho_2 \tau')$, with $\rho_1, \rho_2 \in \cU$ and $\tau' \in \cW$, so that (recalling that the lowest homogeneity in $\cU$ is $-\frac{1}{2}+\beta-\kappa$)
$$|\tau| = 2 + |\rho_1|+|\rho_2|+|\tau'| \geq 1 +2\beta-2\kappa +|\tau'|>|\tau'|,$$
and in addition by the fixed point equation one has
$$\langle W,\tau \rangle =-3 \langle U_0,\rho_1 \rangle\langle U_0,\rho_2 \rangle \langle W,\tau ' \rangle$$
so that if $\langle W,\tau ' \rangle = 0$ a.e. the same holds for $\langle W,\tau \rangle$. Hence it follows that $\delta$ is an integer. But then by letting $W'=\cQ_{<\gamma'} W$  where  $\sup (0,\gamma')\cap \cA = \delta$, $W'$ is an element of $\cD^{\gamma'}_1$ taking value in the polynomial regularity structure, so that by \cite[Prop.3.4]{Recons} $\cR W' = 0 \Rightarrow W'=0$, a contradiction.
\end{proof}

Recall that in order to prove Theorem \ref{Th:Main}, by the Bouleau-Hirsch criterion, it suffices to check that a.s., for all $(\lambda_1,\ldots,\lambda_n),$
$$\left( \forall h \in \cH, \left\langle v^h, \sum_i \lambda_i \phi_i \right\rangle =0\right) \Rightarrow \lambda_1 = \ldots = \lambda_n = 0.$$

By Proposition \ref{prop:dual} (and linearity of $\phi \mapsto w^\phi$), this is equivalent to proving that a.s.
\begin{equation}
\label{eq:PhiH}
\forall \phi \in \operatorname{span}(\phi_1,\ldots,\phi_n), \;\;\left( \forall h \in \cH, \left\langle h, w^\phi 1_{[0,T]}  \right\rangle=0 \right) \Rightarrow \phi=0.
\end{equation}

We will in fact prove the stronger fact that a.s., the above implication holds for all $\phi \in \cB^{1/2+}_1$ with support in $(0,T) \times \T^3$.

\begin{proof}[Proof of Theorem \ref{Th:Main} under Assumption \ref{asn:Dense}]
This is immediate combining the fact that under Assumption \ref{asn:Dense},
$$\left( \forall h \in \cH, \left\langle h, w^\phi 1_{[0,T]} \right\rangle=0 \right) \;\; \Rightarrow \;\; w^\phi 1_{[0,T]}=0$$
and Proposition \ref{prop:wW}.
\end{proof}

The proof under Assumption \ref{asn:Rough} is more involved and we prepare it with a preliminary result.

\begin{lemma} \label{lem:L}
Let $\varphi$ be a smooth, compactly supported function with vanishing moments up to order one, then one has for all $z$
\begin{equation}
\E \left[ \left\langle \Pi_z \<2c0>, \varphi_z \right\rangle ^2 \right] =2  \| L(\varphi) \|^2_{L^2}
\end{equation}
where
$$L(\varphi)(z_1, z_2) = \int dz (P_+ \ast \varphi)(z) (P_+ \ast R)(z-z_1) (P_+ \ast R)(z-z_2).$$
The Fourier transform of $L(\varphi)$ is equal to
$$\widehat{L(\varphi)}(\xi_1,\xi_2) = (\hat{P}_+ \hat{\varphi}) (\xi_1+\xi_2)  (\hat{P}_+ \hat{R}) (-\xi_1) (\hat{P}_+ \hat{R})(-\xi_2).$$
\end{lemma}

\begin{proof}
By stationarity of the model, it suffices to take $z=0$. For the first assertion, note that since $\varphi$ has vanishing moments, one has
$$ \left\langle \Pi \<2c0>, \varphi \right\rangle =  \left\langle P_+(-\cdot) \ast \Pi(\<2>), \varphi \right\rangle =  \left\langle \Pi(\<2>), P_+ \ast \varphi \right\rangle.$$
One then checks via standard Gaussian computations and the definition of $\Pi$ that for any test function $\psi$ one has
$$\E \left[\left\langle \Pi(\<2>), \psi \right\rangle^2 \right]  = 2  \int dz_1 dz_2 \psi(z_1)\psi(z_2) \left\langle P_+\ast R(z_1- \cdot), P_+\ast R(z_2- \cdot) \right\rangle^2  $$
and the result follows.

For the second assertion, we write
\begin{eqnarray*}
\widehat{L(\varphi)}(\xi_1,\xi_2) &=& \int dz_2 dz_1 dz \;\; e^{-i \left\langle \xi_1, z_1\right\rangle} e^{-i \left\langle \xi_2, z_2\right\rangle}  (P_+ \ast \varphi)(z) (P_+ \ast R)(z-z_1) (P_+ \ast R)(z-z_2) \\
&=& \int dz_1 dz \;\; e^{-i \left\langle \xi_1, z_1\right\rangle} (P_+ \ast \varphi)(z) (P_+ \ast R)(z-z_1)  e^{-i \left\langle \xi_2, z\right\rangle} (\hat{P}_+ \hat{R})(- \xi_2) \\
&=& \int dz \;\; e^{-i \left\langle \xi_1+ \xi_2, z\right\rangle}  (P_+ \ast \varphi)(z) (\hat{P}_+\hat{R})(-\xi_1) (\hat{P}_+\hat{R})(- \xi_2)  \\
&=& (\hat{P}_+ \hat{\varphi})(\xi_1+\xi_2) (\hat{P}_+\hat{R})(-\xi_1) (\hat{P}_+\hat{R})(- \xi_2) .
\end{eqnarray*}
\end{proof}

\begin{proof}[Proof of Theorem \ref{Th:Main} under Assumption \ref{asn:Rough}]
Recall that it suffices to prove that $\P$-a.s., for all $\phi$ $\in$ $\cB^{1/2+}_{1}$ supported in $(0,T)\times\T^3$, if $W$ is the solution to
$$W = -\tilde{\cP} (3 U^2 W) + \tilde\cP(\Phi)$$ 
and $w:= \cR W$ is in  $\cH^\perp$, then $w=0$ (which implies $W=0$ and $\Phi=0$ by Proposition \ref{prop:wW}).

By Assumption \ref{asn:Rough} and Lemma \ref{lem:RoughL2}, there exists a sequence $n_k \to \infty$ s.t.,
\begin{equation} \label{eq:RoughSubseq}
\lim_{k \to \infty} 2^{2n_k (3\beta - |\s|)}\int_{B_{n_k}}  \hat{m}(d\xi_1)  \hat{m}(d\xi_2) \left|\hat{R} (\xi_1+\xi_2)  \hat{R}(\xi_1) \hat{R}(\xi_2)\right|^2 >0.
\end{equation}
Now letting $\lambda_k=2^{-n_k}$, we fix $\varphi^k_z = R \ast \eta_z^{\lambda_k}$ where $\eta$ is compactly supported, with vanishing moments up to order $r'\in\N$ with $r' > r+ 4 + \beta$ and such that 
$$ \inf_{A_0} \left| \hat{\eta}\right| > 0.$$
%such functions exist: take $\eta$ compactly supported with vanishing moments up to $r'$ and observe that $\hat\eta$ is analytic so that the $0$ at the origin is necessarily isolated. By rescaling this function, we can make guarantee that an annulus around $0$ does not carry any $0$ of $\hat\eta$.
Then by Lemma \ref{lem:L}, one has
\begin{eqnarray*}
\E \left[ \left\langle \Pi_z \<2c0>, \varphi^k_z \right\rangle ^2 \right] &=&  2 \| L(\varphi^k) \|^2_{L^2}=  2 \| \widehat{L(\varphi^k)} \|^2_{L^2} \\
&\geq& 2 \int_{B_{n_k}}  \hat{m}(d\xi_1)  \hat{m}(d\xi_2)  \left|(\hat{P}_+ \hat{R}) (\xi_1+\xi_2)  (\hat{P}_+ \hat{R}) (-\xi_1) (\hat{P}_+ \hat{R})(-\xi_2) \hat{\eta}(\lambda_k(\xi_1 + \xi_2)) \right|^2 \\
 &\gtrsim&  2^{-12n_k} \int_{B_{n_k}} \hat{m}(d\xi_1)  \hat{m}(d\xi_2) \left|\hat{R} (\xi_1+\xi_2)  \hat{R}(\xi_1) \hat{R}(\xi_2)\right|^2 \\
&\gtrsim&2^{2n_k( - 1 - 3 \beta)}\left(2^{2n_k ( 3\beta - |\s|)}\int_{B_{n_k}}  \hat{m}(d\xi_1)  \hat{m}(d\xi_2) \left|\hat{R} (\xi_1+\xi_2)  \hat{R}(\xi_1) \hat{R}(\xi_2)\right|^2\right)\\ &\gtrsim& 2^{2n_k( - 1 - 3 \beta)}.
\end{eqnarray*}
 
where we have used in the second inequality that $\big|\hat{P}_+(\xi)\big|\gtrsim |\xi|^{-2}$ and \eqref{eq:RoughSubseq} in the last inequality.

We now fix $\gamma$ and $\gamma'$ such that
\begin{equation} \label{eq:gammadef}
1 + 2 \beta < \gamma'<\gamma < \frac{3}{2} + \beta - \kappa
\end{equation}
(This is possible for $\kappa$ small enough since $\beta < \frac{1}{2}$.) 

Now since $\left\langle \Pi_z \<2c0>, \varphi^k_z \right\rangle$ is an element of a Gaussian chaos of order $2$, by the Carbery-Wright inequality \cite[Theorem 8]{CW2001} there exists $C>0$ such that for each $A>0$,
$$\P\left( 2^{n_k (\gamma' + \beta)} \left|\left\langle \Pi_z \<2c0>, \varphi^k_z \right\rangle \right| \leq A \right) \leq \frac{C A^{1/2} 2^{ - \frac{n_k}{2} (\gamma' + \beta)}}{\E \left[  \left\langle \Pi_z \<2c0>, \varphi^k_z \right\rangle ^2 \right]^{1/4} } \lesssim A^{1/2} 2^{- \frac{n_k}{2}  (\gamma' -1-2\beta)}.$$ 
It follows by the Borel-Cantelli Lemma that for any $z$,
$$\P\left(\lim_{k \to \infty}\left| \left\langle \Pi_{z} \<2c0>, \varphi_{z}^{k} \right\rangle \right| 2^{n_k(\gamma'+\beta)}=+\infty\right) =1$$ 
and by Fubini's Theorem that almost surely
\begin{equation} \label{eq:trulyrough}
\left\{ z\in (0,T)\times\T^3 : \lim_{k \to \infty}\left| \left\langle \Pi_{z} \<2c0> , \varphi_{z}^{k} \right\rangle \right| 2^{n_k(\gamma'+\beta)}=+\infty \right\} \mbox{ has full measure.}
\end{equation}

Note that $\varphi^k_z= (R^{\lambda_k^{-1}} \ast \eta)^{\lambda_k}_z$ and we will use the following inequality (the proof of which is deferred to Lemma \ref{lem:Rasteta} below) :
\begin{equation} \label{eq:Rasteta}
\Big|(R^{\lambda_k^{-1}} \ast \eta )(z)\Big| \lesssim \lambda_k^{\beta} (|z|+1)^{-|\s| - r '+ \beta}.
\end{equation}

We now need to localize $W$ to obtain an element of an unweighted space. Let $\theta^\varepsilon=\theta^\varepsilon(t)$ be a smooth function equal to $0$ on $(-\infty,0)$ and to $1$ on $[\varepsilon,+\infty)$, and such that $\|\theta^\varepsilon\|_{\cC^\gamma}\lesssim \varepsilon^{-\gamma}$. We can then lift it to a modelled distribution $\Theta^\varepsilon$ and consider the product $W \cdot \Theta^\varepsilon$. It is simple to check that $W \cdot \Theta^\varepsilon$ belongs to the unweighted space $\cD^\gamma_1=\cD^\gamma_{1,\infty}$ from Definition \ref{Def:Dgamma} and that we have the bound: $\|W \cdot \Theta^\varepsilon\|_{\cD^\gamma_1}\lesssim \varepsilon^{-\delta}$ uniformly over all $\varepsilon > 0$, where $\delta=\gamma+\gamma_0-\eta_0$.

Now for each $z \in (0,T) \times \T^d$ one has using  \eqref{eq:per} that
\[ 
 \left\langle w1_{(0,T)}, \varphi^k_z  \right\rangle_{L^2(\R^4)} = \left\langle w1_{(0,T)}, \left(\varphi^k_z\right)^{per}  \right\rangle_{L^2(\R \times \T^3)} = 0
\] 
since $ \left(\varphi^k_z\right)^{per} =  \left(  R \ast\eta_z^{\lambda_k} \right)^{per} = R \ast \left( \eta_z^{\lambda_k} \right)^{per}$ is in $\mathcal{H}$, so that
\begin{eqnarray*}
0 &=& \left\langle w, 1_{(0,T)}\varphi^k_z  \right\rangle \\
&=&  \left\langle w, \left(1_{(0,T)}-\theta^{\varepsilon_k}\right)\varphi^k_z  \right\rangle \;\;+ \;\;  \left\langle \cR(W\Theta^{\varepsilon_k}) - \Pi_z (W\Theta^{\varepsilon_k})(z),\varphi^k_z  \right\rangle \;\;+\;\; \left\langle \Pi_z (W\Theta^{\varepsilon_k})(z),\varphi^k_z  \right\rangle \\
&=:& A_1^k(z) + A_2^k(z) + A_3^k(z)
\end{eqnarray*}

where we will take $\varepsilon_k =|\log(\lambda_k)|^{-1}$. Hence we have 
\begin{equation} \label{eq:A}
0 = \liminf_{k \to \infty} \int_{(0,T)\times\T^3} dz 1_{\{t \geq 2 \varepsilon_k \}} \lambda_k^{-\gamma'-\beta}  \left| A_1^k(z) + A_2^k(z) + A_3^k(z) \right|  
\end{equation}

We will bound the first two terms from above and the third from below to obtain that $w=0$ a.e.. First note that letting $z=(t,x)$, by \eqref{eq:Rasteta} one has that
$$|\varphi^{k}_z|\lesssim \lambda_k^{-|\s|+\beta}\left( 1+ \frac{(t-\varepsilon_k)_+^{1/2}}{\lambda_k}\right)^{-|\s|-r'+\beta} \mbox{ on the support of }1_{(0,T)} - \theta^\varepsilon,$$ so that for $t \geq 2 \varepsilon_k$ we have by the Cauchy-Schwarz inequality
\[\left|A_1^k(z)\right|\; \leq \; \left\|w \right\|_{L^2} \left\|\left(1_{(0,T)}-\theta^{\varepsilon_k}\right)\varphi^{k}_z \right\|_{L^2} \; \lesssim \; \lambda_k^{r'} \varepsilon_k^{\frac{-|\s|-r'+\beta+1}{2}} \]
and we obtain
\begin{equation} \label{eq:A1}
\limsup_{k \to \infty} \int_{(0,T)\times \T^3} dz 1_{\{t \geq 2 \varepsilon_k \}} \lambda_k^{-\gamma'-\beta }  \left| A_1^k(z) \right|  \lesssim \limsup_{k \to \infty} \lambda_k^{-\gamma'-\beta+r'} \varepsilon_k^{\frac{-|\s|-r'+\beta+1}{2}} = 0.
\end{equation}
%For the second term, we use Lemma \ref{lem:ReconsNoncompact} below (with $p=1$). Indeed, by Lemma \ref{lem:Rasteta} the function $R^{\lambda_k^{-1}} \ast \eta$ belongs to $\ccN_{\lambda_k}^{r,r'-r-\beta}$ and since $r' - r- \beta> 4 > \gamma - |\Xi|$, Lemma \ref{lem:ReconsNoncompact} yields

For the second term, we observe that the function $R^{\lambda_k^{-1}} \ast \eta$ is supported in a centered ball of radius of order $\lambda_k^{-1}$ and satisfies by Lemma \ref{lem:Rasteta}
$$ \| R^{\lambda_k^{-1}} \ast \eta\|_{\cC^r (B(z,1))} \lesssim \lambda_k^{\beta}  \left(1+|z|\right)^{-(|\s| + c)}\;$$
with $c=r'-r-\beta$. Since $r' - r- \beta> 4 > \gamma - |\Xi|$, Lemma \ref{lem:ReconsNoncompact} yields

\begin{equation} \label{eq:A2}
\limsup_{k \to \infty} \int_{(0,T)\times \T^3}  dz 1_{\{t \geq 2 \varepsilon_k \}} \lambda_k^{-\gamma'-\beta }  \left| A_2^k(z) \right| \lesssim \limsup_{k \to \infty} \|W \cdot \Theta^\varepsilon\|_{\cD^\gamma_1} \lambda_k^{\gamma-\gamma'} \lesssim \limsup_{k \to \infty} \varepsilon_k^{-\delta} \lambda_k^{\gamma-\gamma'} = 0.
\end{equation}
For the third term, the expansion of $W$ up to order $\gamma$ gives
$$W(z) = w(z) \mathbf{1} -3 w(z)\; \<2c0> + \sum_{i=1}^3 w_i(z) X_i $$
(the next term in the expansion would be a factor of $\<1c0>$ which is of homogeneity $\frac{3}{2} + \beta - \kappa > \gamma$). Since $\varphi^k_z$ has vanishing moments, it holds that 
$$ \left\langle \Pi_z (W\Theta^{\varepsilon_k})(z),\varphi^k_z  \right\rangle = -3w(z) \theta^\varepsilon(z) \left\langle \Pi_{z} \<2c0>, \varphi_{z}^{k} \right\rangle$$
 and by Fatou's Lemma and \eqref{eq:trulyrough}
 \begin{eqnarray*}
\liminf_{k \to\infty}  \int dz  1_{\{t \geq 2 \varepsilon_k \}} \lambda_k^{-\gamma'-\beta}  \left| A_3^k(z) \right| &\geq& 3 \int dz  |w(z)| \left(\liminf_{k\to\infty} \lambda_k^{-\gamma'-\beta}\left| \left\langle \Pi_{z} \<2c0> , \varphi_{z}^{k} \right\rangle\right|\right) \\
&=& \int dz  |w(z)| \cdot \infty\;.
 \end{eqnarray*}
Combining this inequality with \eqref{eq:A}, \eqref{eq:A1} and \eqref{eq:A2}, we see that necessarily $w=0$ a.e., which finishes the proof.
\end{proof}

\begin{lemma} \label{lem:Rasteta}
Let $R$ satisfy Assumption \ref{asn:R}, and $\eta$ be a compactly supported, smooth function, with vanishing moments up to order $\rho > \beta - |\s|$. Then it holds that for $\lambda \in (0,1]$, $z \in \R^4$,
$$ \big|(R^{\lambda^{-1}} \ast \eta)(z)\big| \lesssim \lambda^{\beta} (1+|z|)^{-|\s|-\rho+\beta}.$$
\end{lemma}

\begin{proof}
We first prove that 
$$ \|R^{\lambda^{-1}} \ast \eta \|_{L^\infty} \lesssim \lambda^{\beta},$$
distinguishing between the cases $\beta=0$ and $\beta>0$. When $\beta=0$, we note that for any $L^2$ function $g$, one has (using Proposition \ref{prop:L2H})
$$\left\|R^{\lambda^{-1}} \ast g \right\|_{L^2} = \lambda^{|\s|/2} \left\|R\ast g^\lambda \right\|_{L^2} \lesssim \lambda^{|\s|/2} \left\| g^\lambda \right\|_{L^2} = \left\|g\right\|_{L^2},$$
so that by Gagliardo-Nirenberg inequality
$$\|R^{\lambda^{-1}} \ast \eta \|_{L^\infty} \; \lesssim \; \|R^{\lambda^{-1}} \ast \eta \|_{L^2}^{1/2} \|R^{\lambda^{-1}} \ast D^4 \eta \|_{L^2}^{1/2} \; \lesssim \; \|\eta \|_{L^2}^{1/2} \|D^4 \eta \|_{L^2}^{1/2} \;\lesssim \;1.$$

In the case $\beta>0$, we fix $n_0$ s.t. $2^{-n_0} \sim \lambda$ and we write
\begin{eqnarray*}
\left(R^{\lambda^{-1}} \ast \eta \right)(z) = \sum_{0 \leq n \leq n_0}  \left(R_n^{\lambda^{-1}} \ast \eta \right)(z) +   \sum_{n > n_0}   \left(R_n^{\lambda^{-1}} \ast \eta \right)(z).
\end{eqnarray*}
The second sum is easily found to be of order $\lambda^{\beta}$, using that
$$ \left(R_n^{\lambda^{-1}} \ast \eta \right)(z) \lesssim \left\| R_n^{\lambda^{-1}} \right\|_{L^1} \left\| \eta\right\|_{L^\infty} \lesssim 2^{-n\beta}.$$
We set for all $z,y$ and all smooth $\phi$:
$$ T_{\rho,z} (\phi)(y) = \sum_{k\in\N^{d+1}:|k|<\rho} \partial^k \phi(z) \frac{y^k}{k!}\;.$$ 
For $0\leq n \leq n_0$, we note that since $\eta$ has vanishing moments up to order $\rho$, one has
$$\left(R_n^{\lambda^{-1}} \ast \eta \right)(z) = \int \left(R_n^{\lambda^{-1}}(z-y)  - T_{\rho,z}(R_n^{\lambda^{-1}})(-y)\right)\eta(y) dy $$
and for $|y| \lesssim 1$, by the Taylor formula from \cite[Proposition A.1]{Hairer2014}, it holds that
 $$ \left|R_n^{\lambda^{-1}}(z-y)  - T_{\rho,z}(R_n^{\lambda^{-1}})(-y)\right| \lesssim \sup_{\rho \leq |l| \leq \rho+2} \|\partial^l R_n^{\lambda^{-1}}\|_{L^\infty} \lesssim \lambda^{|\s|+\rho} 2^{n(|\s|+\rho-\beta)}.$$
Hence 
$$\sum_{n \leq n_0}   \Big|\left(R_n^{\lambda^{-1}} \ast \eta \right)(z)\Big| \lesssim \lambda^{|\s|+\rho}  \sum_{n \leq n_0} 2^{n(|\s|+\rho-\beta)} \lesssim \lambda^{\beta}.$$

Finally, let $C_\eta$ be such that $\eta$ is supported in $B(0,C_\eta)$, and $C>0$ be as given in Assumption \ref{asn:R}, we remark that when $|z| \geq 2C_\eta$, 
$$C 2^{-n}<  \frac{\lambda|z|}{2} \Rightarrow C 2^{-n} <   \lambda( |z| - C_{\eta})$$
so that for such $n$, for all $y$ in $supp(\eta)$, $R_n^{\lambda^{-1}}(z-y) = 0$, and $\left(R_n^{\lambda^{-1}} \ast \eta\right)(z)=0$.

Hence it holds that
\begin{align*}
\left(R^{\lambda^{-1}} \ast \eta \right)(z) &= \sum_{n : 2^{-n} \gtrsim \lambda |z|}  \int R_n^{\lambda^{-1}}(z-y) \eta(y) dy 
\end{align*}
so that the same estimate as above gives  $$\left(R^{\lambda^{-1}} \ast \eta \right)(z) \lesssim \lambda^{|\s|+\rho}   \sum_{n : 2^{-n} \gtrsim \lambda |z|} 2^{n(|\s|+\rho-\beta)} \lesssim \lambda^{\beta} |z|^{-|\s|-\rho+\beta}.$$
\end{proof}

\begin{lemma} \label{lem:ReconsNoncompact}
Take $\ccN^{r,c}_\lambda$ as the space of all functions $\varphi$ supported in $B(0,\lambda^{-1})$ such that $\| \varphi \|_{\cC^r (B(z,1))} \leq  \left(1+|z|\right)^{-(|\s| + c)}$, with $c > p\left( \gamma - \inf \cA\right)$. Then, we have
$$ \sup_{\lambda\in (0,1]} \bigg\| \sup_{\varphi\in \ccN^{r,c}_\lambda} \frac{\big|\langle \cR f-\Pi_z f(z), \varphi_z^\lambda \rangle\big|}{\lambda^{\gamma}} \bigg\|_{L^p(\R_+\times\T^d,dz)} \lesssim \$f\$_{\cD^\gamma_p}\;,$$
uniformly over all $f \in \cD^\gamma_p$.
\end{lemma}

\begin{proof}
Let $\psi:\R^{d+1}\to\R$ be a smooth function, supported in $B(0,1)$, that defines a partition of unity $\sum_{y\in\Lambda_0} \psi_y(\cdot) = 1$. We can decompose 
$$\varphi = \sum_{y\in \Lambda_0:|y| \le \lambda^{-1}+1} \psi_y \varphi =: \sum_{y\in \Lambda_0:|y| \le \lambda^{-1}+1} \eta^y(\cdot-y)\;,$$
where each $\eta^y$ is supported in $B(0,1)$ and such that $$\|\eta^y\|_{\cC^r} \leq \left(1+|y|\right)^{-(|\s| + c)}.$$
We now rely on the reconstruction theorem~\cite[Th 3.1]{Recons} in unweighted spaces. We have 
\begin{eqnarray*}
\left\langle \cR f - \Pi_z f(z), \varphi^\lambda_z \right\rangle &=& \sum_{y\in \Lambda_0} \left\langle \cR f - \Pi_z f(z), (\eta^y)^\lambda_{z+\lambda y} \right\rangle \\
&=& \sum_{y\in \Lambda_0:|y| \le \lambda^{-1}+1} \left\langle \cR f - \Pi_{z+\lambda y} f(z+\lambda y), (\eta^y)^\lambda_{z+\lambda y} \right\rangle\\
&&  + \sum_{y\in \Lambda_0:|y| \le \lambda^{-1}+1} \left\langle \Pi_{z+\lambda y}\left( f(z+\lambda y) - \Gamma_{z+\lambda y ,z} f(z)\right), (\eta^y)^\lambda_{z+\lambda y} \right\rangle .
\end{eqnarray*}
We then bound 
$$\int dz \sup_{\varphi\in \ccN^{r,c}_\lambda}\left|\left\langle \cR f - \Pi_z f(z), \varphi^\lambda_z \right\rangle\right|^p$$
by bounding separately the two terms above. For the first one, we have
\begin{eqnarray*}
&& \int dz  \sup_{\varphi\in \ccN^{r,c}_\lambda}\left|\sum_{y\in \Lambda_0} \left\langle \cR f - \Pi_{z+\lambda y} f(z+\lambda y), (\eta^y)^\lambda_{z+\lambda y} \right\rangle\right|^p \\
&\leq& \int dz  \left|\sum_{y\in \Lambda_0} \frac{1}{ \left(1+|y|\right)^{|\s|+c}} \sup_{\phi \in \cB^r} \big|\left\langle \cR f - \Pi_{z+\lambda y} f(z+\lambda y),\phi^\lambda_{z+\lambda y} \right\rangle\big|\right|^p \\
&\lesssim &  \int \sum_{y\in \Lambda_0} dz'   \frac{1}{ \left(1+|y|\right)^{|\s|+c}}  \left| \sup_{\phi \in B^r} \big|\left\langle \cR f - \Pi_{z'} f(z'),\phi^\lambda_{z'} \right\rangle\big|\right|^p \\
&\lesssim& \lambda^{\gamma p}
\end{eqnarray*}
where we have used Jensen's inequality and the change of variables $z'=z+\lambda y$. For the second one, note that 
$$\left\langle \Pi_{z+\lambda y}\left( f(z+\lambda y) - \Gamma_{z+\lambda y ,z} f(z)\right), (\eta^y)^\lambda_{z+\lambda y} \right\rangle \lesssim \sum_{\zeta < \gamma} \left|f(z+\lambda y) - \Gamma_{z+\lambda y ,z} f(z)\right|_\zeta  \left(1+|y|\right)^{-|\s|-c} \lambda^\zeta$$
(the sums being taken over elements $\zeta$ in $\cA$), so that we obtain
\begin{eqnarray*}
&&\sum_{\zeta<\gamma} \int dz \left|\sum_{y\in \Lambda_0:|y| \le \lambda^{-1}+1} \left|f(z+\lambda y) - \Gamma_{z+\lambda y ,z} f(z)\right|_\zeta  \left(1+|y|\right)^{-|\s|-c} \lambda^\zeta\right|^p \\
&\lesssim& \sum_{\zeta<\gamma} \int dz \sum_{y\in \Lambda_0:|y| \le \lambda^{-1}+1} \frac{\lambda^{p\zeta}}{(1+|y|)^{|\s|+c}}  \left|f(z+\lambda y) - \Gamma_{z+\lambda y ,z} f(z)\right|_\zeta^p \\
&\lesssim&  \lambda^{p\gamma} \sum_{y\in\Lambda_0} \frac{|y|^{p(\gamma-\zeta)}}{(1+|y|)^{|\s|+c}}  \$f\$_{\cD^\gamma_p}^p \;\;\lesssim \;\;\lambda^{p\gamma}
\end{eqnarray*}
where we have again used Jensen's inequality and the fact that we assumed $c> p(\gamma - \beta)$ for all $\beta \in \cA$. \end{proof}

\section{Some technical proofs}\label{Sec:Techos}

\subsection{Reconstruction}

In addition to the Reconstruction Theorem, we will need (in the proof of the Convolution Theorem) a technical result on the regularity of the image of the reconstruction operator near the hyperplane $t=0$. For simplicity, we let $L^p(n_0)$ denote the space $L^p((2^{-2n_0}\wedge T,2^{-2(n_0-1)}\wedge T)\times \T^d,dz)$ for every $n_0\in \Z$. Let $n_T\in \Z$ be the unique integer such that $2^{-2n_T} \le T < 2^{-2(n_T-1)}$.\\
\begin{proposition}\label{Prop:Recons}
In the context of Theorem \ref{Th:ReconstructionW}, for any given $\epsilon >0$ and for every multiindex $k\in \N^{d+1}$, we have the following bounds:
\begin{equation}\label{Eq:ReconsConvol}
\bigg(\sum_{n_0 \ge n_T} \Big\| \sum_{0\le m\le n_0+4} \frac{\big|\langle \cR f,\partial^k P_m(z-\cdot) \rangle\big|}{2^{-n_0(\eta+2-\epsilon-|k|)}} \Big\|_{L^p(n_0)}^p\bigg)^{\frac1{p}} \lesssim \$ f\$\;,
\end{equation}
and
\begin{equation}\label{Eq:ReconsConvol2}
\bigg(\sum_{n_0 \ge n_T} \Big\| \frac{\big|\langle \cR f,\partial^k P_-(z-\cdot) \rangle\big|}{2^{-n_0(\eta+2-\epsilon-|k|)}} \Big\|_{L^p(n_0)}^p\bigg)^{\frac1{p}} \lesssim \$ f\$\;,
\end{equation}
uniformly over all $f\in \cD^{\gamma,\eta,T}_p$. In the case of two models, bounds similar to \eqref{Eq:ReconsBoundTwo} hold.
\end{proposition}

We now present the proofs of Theorem \ref{Th:ReconstructionW} and Proposition \ref{Prop:Recons} jointly. For notational simplicity, we take $T=1$, although the arguments carry through if $T>0$ is arbitrary. The proof of Theorem \ref{Th:ReconstructionW} and Proposition \ref{Prop:Recons} relies on three main arguments. First, we use the reconstruction theorem in unweighted spaces of modelled distributions~\cite[Thm 3.1]{Recons}. Indeed, any element $f\in\cD^{\gamma,\eta,T}_{p}$, restricted (by a localisation argument) to a ball $B(z,\lambda)$ with $z=(t,x)$ and $3\lambda^2 < t$, belongs to $\cD^\gamma_{p}$ and the norm of the injection is of order $t^{\frac{\eta-\gamma}{2}}$. This allows to reconstruct $f$ away from the hyperplane $t=0$.\\
Second, we show that a distribution on the set of test functions supported away from the hyperplane $t=0$ can uniquely be extended to test functions supported on this hyperplane as soon as the regularity index is not too low.\\
Third, we obtain the specific regularity index near the hyperplane $t=0$ by an accurate analysis of the interplay between the regularity of the model and the growth/decay of the weights.\\

The second and third steps are the contents of the following lemmas.
\begin{lemma}\label{Lemma:Extension}
Take $\nu \ne 0$ and assume that $\nu  > - 2(1-\frac1{p})$. Let $\xi$ be a distribution on the set of all test functions whose supports lie in $\big((-\infty,0)\cup (0,T)\big) \times \R^d$. Assume that $\langle \xi,\varphi\rangle = 0$ whenever the support of $\varphi$ lies in $(-\infty,0)\times\R^d$. In addition, assume that:\begin{itemize}
\item if $\nu < 0$, $\xi$ satisfies the bound \eqref{Eq:BgammaT} with $(-\infty,T-\lambda^2)$ replaced by $(3\lambda^2, T-\lambda^2)$
\item if $\nu > 0$, $\xi$ satisfies the bounds \eqref{Eq:BgammaTPos} with $(-\infty,T-\lambda^2)$ replaced by $(3\lambda^2, T-\lambda^2)$, and for any $c'>0$ we have the additional bound
\begin{equation*}\label{Eq:SpecificConvol}
\sup_{m\ge 0} \Big\| \sup_{\varphi\in \ccB^r} \frac{\big|\langle \xi, \varphi^{2^{-m}}_z \rangle\big|}{2^{-m\nu}} \Big\|_{L^p\big((3\cdot 2^{-2m}\wedge (T-2^{-2m}),c' \cdot 2^{-2m}\wedge (T-2^{-2m}))\times \T^d,dz\big)} < \infty\;.
\end{equation*}
\end{itemize}
Then, there exists a unique extension of $\xi$ that belongs to the space $\cB^{\nu,T}_{p}$.
\end{lemma}
\begin{proof}
The proof consists in two steps: first we show uniqueness of the extension and second we construct the extension. For further use, we let $\chi:\R\rightarrow\R$ be a smooth function, supported in a compact subset $[a,A]$ with $a>15$ and such that for all $t>0$
$$\sum_{n\in \Z} \chi(2^{2n} t) = 1\;.$$

\textit{Uniqueness.} If $\nu > 0$, then any element $\xi$ of $\cB^{\nu,T}_p$ is a function in $L^p$, see for instance~\cite[Lemma 2.7]{Recons}; consequently, $\xi$ is completely determined by its evaluations away from $t=0$.\\
Let us now consider the case $\nu < 0$. Let $I_{n_0}(t) := 1 - \sum_{n\le n_0} \chi(2^{2n}t) - \sum_{n\le n_0} \chi(-2^{2n}t)$ and observe that $I_{n_0}$ is a smooth function, supported in $[- 2^{-2n_0} R,2^{-2n_0} R]$ for some $R>0$ that does not depend on $n_0$. If we show that
\begin{equation}\label{Eq:BoundExtension}
\big|\langle \xi, \varphi_{(0,x_0)} \cdot I_{n_0} \rangle \big| \lesssim 2^{-2n_0(1-\frac1{p}) - n_0\nu} \$ \xi\$_{\cB^{\nu,T}_{p}}\;,
\end{equation}
uniformly over all $n_0$ large enough, all $x_0\in \T^d$, all $\varphi \in \ccB^r$ and all $\xi \in \cB^{\nu,T}_{p}$, then we deduce that any $\xi \in \cB^{\nu,T}_{p}$ is completely characterised by its evaluations away from the hyperplane $t=0$ as soon as $\nu > -2\big(1-\frac1{p}\big)$.\\
Therefore, we are left with proving (\ref{Eq:BoundExtension}). We consider a smooth function $\psi:\R^{d+1}\rightarrow\R$, supported in $B(0,1)$, that defines a partition of unity:
$$ \sum_{\bar z \in \Lambda_0} \psi_{\bar z}(z) = 1\;,\quad \forall z\in \R^{d+1}\;,$$
as well as its rescaled version $\psi^{n_0}_{\bar z}(\cdot) := \psi(2^{n_0\s}(\cdot-\bar z))$ which defines a partition of unity at scale $2^{-n_0}$:
$$ \sum_{\bar z \in \Lambda_{n_0}} \psi_{\bar z}^{n_0}(z) = 1\;,\quad \forall z\in \R^{d+1}\;.$$
We thus have for any test function $\varphi\in\ccB^r$ and any $x_0\in\T^d$
$$ \varphi_{(0,x_0)}\cdot I_{n_0} = \sum_{\substack{\bar z \in \Lambda_{n_0}\\ |\bar t| \le (R+1) 2^{-2n_0}\\ |\bar x| \le 3}} \varphi_{(0,x_0)}\cdot I_{n_0}\cdot \psi_{\bar z}^{n_0}\;.$$
For any $\bar z \in \Lambda_{n_0}$ and any $z\in B(\bar z,2^{-n_0})$, the function $\varphi I_{n_0} \psi^{n_0}_{\bar z}$ can be written as $2^{-n_0|\s|} \eta^{2^{-(n_0-1)}}_{z}$ for some function $\eta \in \ccB^r$ and up to some multiplicative constant which is uniformly bounded over all the parameters at stake. (Recentering the function at $z$ instead of $\bar z$ is convenient to recover $L^p$ norms later on). Using Jensen's inequality at the second line, we get
\begin{align*}
\big|\langle \xi,\varphi_{(0,x_0)}\cdot I_{n_0} \rangle \big| &\lesssim \sum_{\substack{\bar z \in \Lambda_{n_0}\\ |\bar t| \le (R+1) 2^{-2n_0}\\ |\bar x| \le 3}} \int_{z\in B(\bar z, 2^{-n_0})} 2^{n_0|\s|}\sup_{\eta\in \ccB^r} 2^{-n_0|\s|} \big|\langle \xi , \eta^{2^{-(n_0-1)}}_z \rangle\big| dz\\
&\lesssim \bigg(\sum_{\substack{\bar z \in \Lambda_{n_0}\\ |\bar t| \le (R+1) 2^{-2n_0}\\ |\bar x| \le 3}}2^{-n_0(|\s|-2)} \int_{z\in B(\bar z, 2^{-n_0})} 2^{n_0|\s|}\sup_{\eta\in \ccB^r} \Big(2^{-2n_0} \big|\langle \xi , \eta^{2^{-(n_0-1)}}_z \rangle\big|\Big)^p dz\bigg)^{\frac1{p}}\\
&\lesssim 2^{-2n_0(1-\frac1{p})} \bigg(\int_{z\in (-(R+2)2^{-2n_0},(R+2)2^{-2n_0})\times \T^d}\sup_{\eta\in \ccB^r} \big|\langle \xi , \eta^{2^{-(n_0-1)}}_z \rangle\big|^p dz\bigg)^{\frac1{p}}\\
&\lesssim 2^{-2n_0(1-\frac1{p}) - n_0 \nu} \| \xi \|_{\cB^{\nu,T}_{p}}\;,
\end{align*}
uniformly over all $n_0\ge 0$ such that $(R+2)2^{-2n_0}<T$, all $x_0\in\T^d$ and all $\varphi\in \ccB^r$. The asserted bound follows, so that the uniqueness part of the statement is proved.

\textit{Existence.} For all $n\in\Z$ and all $\bar z \in \Lambda_n$, we set
$$ \tilde{\psi}_{\bar z}^n(z') := \chi(2^{2n} t') \psi^n_{\bar z}(z')\;,\quad z'=(t',x')\in \R^{d+1}\;,$$
and we observe that $\sum_{n\in \Z}\sum_{\bar z \in \Lambda_n} \tilde{\psi}_{\bar z}^n(z') = 1$ for all $z'\in (0,\infty)\times \R^d$. We need to define $ \langle \xi,\varphi^\lambda_z \rangle$ when the support of $\varphi^\lambda_z$ overlaps the hyperplane $t=0$. Since $\xi$ vanishes on $(-\infty,0)\times \R^d$, it is natural to set
$$ \langle \xi,\varphi^\lambda_z \rangle := \sum_{n\in \Z} \sum_{\bar z \in \Lambda_n} \langle \xi , \varphi^\lambda_z \tilde{\psi}_{\bar z}^n \rangle\;,$$
for all $\lambda \in (0,1]$, all $z\in (-\lambda^2,3\lambda^2]\times \mathbb{T}^d$ and all $\varphi \in \ccB^r$. Let us show that
\begin{equation*}
\sup_{\lambda \in (0,1]} \Big\| \sup_{\varphi\in\ccB^r} \frac{\big|\langle \xi,\varphi^\lambda_z \rangle\big|}{\lambda^{\nu}} \Big\|_{L^p((-\lambda^2, 3\lambda^2]\times \T^d,dz)} < \infty\;.
\end{equation*}
holds whatever the sign of $\nu$. Notice that for $\nu < 0$ this is what we need, while for $\nu > 0$ this is stronger than what is required since $\varphi$ is not assumed to annihilate polynomials here.\\
Observe that $\varphi^\lambda_z \tilde{\psi}_{\bar z}^n$ vanishes as soon as $2^{-n} > \lambda$: indeed, the cutoff function $\chi(2^{2n}\cdot)$ is supported in $[a 2^{-2n},A 2^{-2n}]$, the function $\varphi^\lambda_z$ vanishes in $[4\lambda^2,\infty)\times \R^d$ and $4\lambda^2 < a \lambda^2 \le a 2^{-2n}$. Furthermore, for all $z'\in B(\bar z, 2^{-n})$, the function $\varphi^\lambda_z \tilde{\psi}_{\bar z}^n$ coincides with $2^{-n|\s|}\lambda^{-|\s|} \rho^{2^{-(n-1)}}_{z'}$ for some function $\rho \in \ccB^r$ up to a multiplicative factor uniformly bounded over all the parameters. Then, for every $n\ge 0$ such that $2^{-n}\le \lambda$, we have
\begin{align*}
&\Big\| \sum_{\bar z \in \Lambda_n} \sup_{\varphi\in\cB^r} \frac{\big|\langle \xi,\varphi^\lambda_z \tilde{\psi}_{\bar z}^n\rangle\big|}{\lambda^{\nu}} \Big\|_{L^p((-\lambda^2, 3\lambda^2)\times \T^d,dz)}\\
&\lesssim \Big\| \sum_{\substack{\bar z \in \Lambda_n\\|z-\bar z| \le \lambda + 2^{-n}\\ \bar t\in [(a-1) 2^{-2n},(A+1) 2^{-2n}]}} \int_{z'\in B(\bar z, 2^{-n|\s|})} 2^{n|\s|} \sup_{\rho\in\cB^r} \frac{\big|\langle \xi,\rho^{2^{-(n-1)}}_{z'} \rangle\big|}{\lambda^{\nu}} dz' 2^{-n|\s|} \lambda^{-|\s|} \Big\|_{L^p\big((-\lambda^2, 3\lambda^2)\times \T^d,dz\big)}\;,
\end{align*}
The number of non-zero contributions coming from the sum over $\bar z$ is of order $\lambda^d 2^{nd}$ uniformly over all the parameters. Hence, by Jensen's inequality we get
\begin{align*}
&\lesssim \lambda^{-2} 2^{-2n}\Big(\int_{z\in(-\lambda^2, 3\lambda^2)\times \T^d} \sum_{\substack{\bar z \in \Lambda_n\\|z-\bar z| \le \lambda + 2^{-n}\\ \bar t\in [(a-1) 2^{-2n},(A+1) 2^{-2n}]}} 2^{-nd} \lambda^{-d} \\
&\qquad\qquad\qquad\times\int_{z'\in B(\bar z, 2^{-n|\s|})} 2^{n|\s|} \Big(\sup_{\rho\in\ccB^r} \frac{\big|\langle \xi,\rho^{2^{-(n-1)}}_{z'} \rangle\big|}{\lambda^{\nu}}\Big)^p dz'\, dz\Big)^{\frac1{p}}\\
&\lesssim \lambda^{-2} 2^{-2n}\Big(\int_{z\in(-\lambda^2, 3\lambda^2)\times \T^d}  \int_{\substack{z'=(t',x')\\t'\in [3\cdot 2^{-2n},c'2^{-2n}]\\ |z'-z| \le \lambda+c'2^{-n}}} \lambda^{-d} 2^{2n} \Big(\sup_{\rho\in\cB^r} \frac{\big|\langle \xi,\rho^{2^{-(n-1)}}_{z'} \rangle\big|}{\lambda^{\nu}}\Big)^p dz' dz\Big)^{\frac1{p}}\;,
\end{align*}
for some $c'>0$. Since for every given $z'$ in the last integral, the integral over $z\in(-\lambda^2, 3\lambda^2)\times \T^d$ of the indicator of $|z'-z| \le \lambda+c'2^{-n}$ gives a term of order $\lambda^{|\s|}$ we deduce the following bound
\begin{align*}
&\lesssim \lambda^{-2\big(1-\frac1{p}\big)-\nu} 2^{-2n\big(1-\frac1{p}\big)-n\nu}\Big(\int_{z'\in [3\cdot 2^{-2n},c'2^{-2n}]\times\T^d} \Big(\sup_{\rho\in\ccB^r} \frac{\big|\langle \xi,\rho^{2^{-(n-1)}}_{z'} \rangle\big|}{2^{-n\nu}}\Big)^p dz'\Big)^{\frac1{p}}\;,
\end{align*}
uniformly over all $n\ge 0$ such that $2^{-n}\le \lambda$. Henceforth, we find
\begin{align*}
&\Big\| \sup_{\varphi\in\cB^r} \frac{\big|\langle \xi,\varphi^\lambda_z \rangle\big|}{\lambda^{\nu}} \Big\|_{L^p\big((-\lambda^2, 3\lambda^2)\times \T^d,dz\big)}\\
&\lesssim \sum_{n: 2^{-n}\le \lambda} \Big\| \sum_{\bar z \in \Lambda_n} \sup_{\varphi\in\cB^r} \frac{\big|\langle \xi,\varphi^\lambda_z \tilde{\psi}_{\bar z}^n\rangle\big|}{\lambda^{\nu}} \Big\|_{L^p\big((-\lambda^2, 3\lambda^2)\times \T^d,dz\big)}\\
&\lesssim \sum_{n: 2^{-n}\le \lambda}\lambda^{-2\big(1-\frac1{p}\big)-\nu} 2^{-2n\big(1-\frac1{p}\big)-n\nu} \times \sup_{n: 2^{-n} \le \lambda} \Big(\int_{z'\in [3\cdot 2^{-2n},c'2^{-2n}]\times\T^d} \Big(\sup_{\rho\in\ccB^r} \frac{\big|\langle \xi,\rho^{2^{-(n-1)}}_{z'} \rangle\big|}{2^{-n\nu}}\Big)^p dz'\Big)^{\frac1{p}}\\
&\lesssim 1\;,
\end{align*}
as desired.
\end{proof}
%\begin{lemma}\label{Lemma:SpecificConvol}
%In the context of Lemma \ref{Lemma:Extension}, suppose that for any $c'>0$ we have the additional bound
%\begin{equation}\label{Eq:SpecificConvol}
%\bigg(\sum_{m\ge 0} \Big\| \sup_{\varphi\in \ccB^r_1} \frac{\big|\langle \xi, \varphi^{2^{-m}}_z \rangle\big|}{2^{-m\eta'}} \Big\|_{L^p\big((3\cdot 2^{-2m},c' \cdot 2^{-2m})\times \T^3,dz\big)}^p\bigg)^{\frac1{p}} < \infty\;,
%\end{equation}
%then we have the following bound for every multiindex $k\in \N^{d+1}$
%\begin{equation}\label{Eq:SpecificConvol}
%\bigg(\sum_{n_0> n_T} \Big\| \sum_{m\le n_0+2} \frac{\big|\langle \xi,\partial^k P_m(z-\cdot) \rangle\big|}{2^{-n_0(\eta'-|k|+2)}} \Big\|_{L^p\big((2^{-2n_0},2^{-2(n_0-1)})\times \T^3,dz\big)}^p\bigg)^{\frac1{p}} < \infty\;.
%\end{equation}
%\end{lemma}
\begin{lemma}\label{Lemma:SpecificConvol}
In the context of Lemma \ref{Lemma:Extension}, take $\nu' < \nu$ such that $\nu' > -2(1-\frac1{p})$. Then we have the following bounds for every multiindex $k\in \N^{d+1}$
\begin{equation}\label{Eq:SpecificConvol}
\bigg(\sum_{n_0\ge 0} \Big\| \sum_{0\le m\le n_0+4} \frac{\big|\langle \xi,\partial^k P_m(z-\cdot) \rangle\big|}{2^{-n_0(\nu'-|k|+2)}} \Big\|_{L^p(n_0)}^p\bigg)^{\frac1{p}} < \infty\;,
\end{equation}
\begin{equation}\label{Eq:SpecificConvol2}
\bigg(\sum_{n_0\ge 0} \Big\| \frac{\big|\langle \xi,\partial^k P_-(z-\cdot) \rangle\big|}{2^{-n_0(\nu'-|k|+2)}} \Big\|_{L^p(n_0)}^p\bigg)^{\frac1{p}} < \infty\;.
\end{equation}
\end{lemma}

\begin{proof}
We only prove \eqref{Eq:SpecificConvol}, \eqref{Eq:SpecificConvol2} can be obtained by similar arguments.
We adapt the proof of Lemma \ref{Lemma:Extension} to this specific test function. First of all, we have for all $z\in (2^{-2n_0},2^{-2(n_0-1)})\times \T^d$ and all $0\le m\le n_0+4$
$$ \partial^k P_m(z-z') = \sum_{n \ge n_0} \sum_{\bar z \in \Lambda_n} \partial^k P_m(z-z') \tilde{\psi}^n_{\bar z}(z')\;,\quad \forall z'\in (0,\infty)\times\R^d\;.$$
We observe that the terms in the sum over $\bar z$ vanish except when $\bar t \in [(a-1)2^{-2n},(A+1) 2^{-2n}]$ and $|\bar x - x| < C2^{-m}$ for some constant $C>0$ depending on the sizes of the supports of $P_0$ and $\tilde{\psi}$. Furthermore, for all $z'' \in B(\bar z,2^{-n})$, the function $\partial^k P_m(z-\cdot) \tilde{\psi}^n_{\bar z}(\cdot)$ can be viewed as $2^{m(|k|+d)}2^{-n|\s|} \rho^{2^{-(n-1)}}_{z''}$ for some function $\rho \in \ccB^r$, up to a multiplicative constant which is uniformly bounded over all the parameters. This being given, we write
\begin{align*}
&\bigg\| \frac{\big|\langle \xi,\partial^k P_m(z-\cdot) \rangle\big|}{2^{-n_0(\nu'-|k|+2)}} \bigg\|_{L^p(n_0)}\lesssim 2^{(m-n_0)|k|} \sum_{n\ge n_0} 2^{-(n-n_0)(\nu'+2)}\\
&\qquad\times\bigg\|\sum_{\substack{\bar z \in \Lambda_n\\\bar t\in [(a-1)2^{-2n},(A+1) 2^{-2n}]\\|\bar x - x| \le C2^{-m}}} 2^{(m-n)d}  \int_{z'' \in B(\bar z, 2^{-n})} 2^{n|\s|} \sup_{\rho\in\ccB^r}\frac{\big|\langle \xi, \rho^{2^{-(n-1)}}_{z''}\rangle\big|}{2^{-n\nu'}} dz'' \bigg\|_{L^p(n_0)}\;.
\end{align*}
Then, using Jensen's inequality we get
\begin{align*}
&\Big\|\sum_{\substack{\bar z \in \Lambda_n\\\bar t\in [(a-1)2^{-2n},(A+1) 2^{-2n}]\\|\bar x - x| \le C2^{-m}}} 2^{(m-n)d}  \int_{z'' \in B(\bar z, 2^{-n})} 2^{n|\s|} \sup_{\rho\in\ccB^r}\frac{\big|\langle \xi, \rho^{2^{-(n-1)}}_{z''}\rangle\big|}{2^{-n\nu'}} dz'' \Big\|_{L^p(n_0)}\\
&\lesssim \bigg(\int_{z\in (2^{-2n_0}\wedge T,2^{-2(n_0-1)}\wedge T)\times \T^d} \sum_{\substack{\bar z \in \Lambda_n\\\bar t\in [(a-1)2^{-2n},(A+1) 2^{-2n}]\\|\bar x - x| \le C2^{-m}}} 2^{(m-n)d}  \int_{z'' \in B(\bar z, 2^{-n})} 2^{n|\s|} \sup_{\rho\in\ccB^r} \Big(\frac{\big|\langle \xi, \rho^{2^{-(n-1)}}_{z''}\rangle\big|}{2^{-n\nu'}}\Big)^p dz''\,dz\bigg)^{\frac1{p}}\\
&\lesssim \bigg(\int_{z'' \in [(a-2)2^{-2n},(A+2)2^{-2n}]\times \T^d} 2^{2(n-n_0)}  \sup_{\rho\in\ccB^r} \Big(\frac{\big|\langle \xi, \rho^{2^{-(n-1)}}_{z''}\rangle\big|}{2^{-n\nu'}} \Big)^p dz'' \bigg)^{\frac1{p}}\;,
\end{align*}
uniformly over all $0\le m\le n_0+4$ and all $n_0 \le n$. Observe that the last bound does not depend on $m$.\\
We now argue separately according as $k=0$ or $|k| > 0$. In the case $k=0$ and let us introduce $\nu'' \in (\nu',\nu)$. The quantity
$$ \sum_{0\le m\le n_0+4}\sum_{n\ge n_0} \frac1{n_0} 2^{-(n-n_0)(\nu''+2-2/p)}\;,$$
is uniformly bounded over all $n_0\ge 0$. Therefore, using Jensen's inequality on the sums over $m$ and $n$ we get
\begin{align*}
&\bigg(\sum_{n_0\ge 0} \Big\| \sum_{m\le n_0+4} \frac{\big|\langle \xi,\partial^k P_m(z-\cdot) \rangle\big|}{2^{-n_0(\nu'-|k|+2)}} \Big\|_{L^p(n_0)}^p\bigg)^{\frac1{p}}\\
&\lesssim\bigg(\sum_{n_0\ge 0} \sum_{0\le m\le n_0+4} \sum_{n\ge n_0}\frac1{n_0} 2^{-(n-n_0)(\nu''+2-\frac{2}{p})}\\
&\qquad \times\int_{z'' \in [(a-2)2^{-2n},(A+2)2^{-2n}]\times \T^d} \sup_{\rho\in\ccB^r} \Big(2^{-n_0(\nu''-\nu')} n_0 \frac{\big|\langle \xi, \rho^{2^{-(n-1)}}_{z''}\rangle\big|}{2^{-n\nu''}} \Big)^p dz''\bigg)^{\frac1{p}}\\
&\lesssim\bigg(\sum_{n\ge 0} \int_{z'' \in [(a-2)2^{-2n},(A+2)2^{-2n}]\times \T^d} \sup_{\rho\in\ccB^r} \Big(\frac{\big|\langle \xi, \rho^{2^{-(n-1)}}_{z''}\rangle\big|}{2^{-n\nu''}} \Big)^p dz'' \bigg)^{\frac1{p}}\;.
\end{align*}
By concavity of $x\mapsto x^{1/p}$ on $\R_+$, the last term is bounded by
\begin{align*}
&\sum_{n\ge 0} \bigg(\int_{z'' \in [\frac{(a-2)}{4}2^{-2n},\frac{(A+2)}{4}2^{-2n}]\times \T^d} \sup_{\rho\in\ccB^r} \Big(\frac{\big|\langle \xi, \rho^{2^{-n}}_{z''}\rangle\big|}{2^{-n\nu''}} \Big)^p dz'' \bigg)^{\frac1{p}}\\
&\lesssim \sum_{n\ge 0} 2^{-n(\nu-\nu'')} \sup_{n\ge 0} \bigg(\Big(\int_{z'' \in [3\cdot 2^{-2n},c'\cdot2^{-2n}]\times \T^d} \sup_{\rho\in\ccB^r} \Big(\frac{\big|\langle \xi, \rho^{2^{-n}}_{z''}\rangle\big|}{2^{-n\nu}} \Big)^p dz'' \bigg)^{\frac1{p}}\;,
\end{align*}
which is finite by assumption. The case $k>0$ is simpler: we don't need to introduce $\nu''$ since
$$ \sum_{m\le n_0+4}\sum_{n\ge n_0} 2^{(m-n_0)|k|} 2^{-(n-n_0)(\nu'+2-2/p)}\;,$$
is bounded uniformly over all $n_0\ge 0$. A similar calculation as before allows to complete the proof.
\end{proof}

\begin{proof}[Proof of Theorem \ref{Th:ReconstructionW} and Proposition \ref{Prop:Recons}]
First of all, we set $\langle \cR f,\varphi\rangle := 0$ whenever $\varphi$ is supported in $\{t<0\}\times\R^d$. Second, take $z=(t,x)$ and $\lambda\in (0,1]$ such that $t\in (3\lambda^2,T-\lambda^2)$, and observe that $f$ satisfies locally the bound of the unweighted space $\cD^\gamma_p$:
\begin{align}\label{Eq:BdLoc}
\Big\| \big| f(z') \big|_\zeta \Big\|_{L^p(D,dz')}+ \sup_{h\in B(0,\lambda)} \bigg\| \frac{\big| f(z'+h)-\Gamma_{z'+h,z'} f(z') \big|_\zeta}{|h|^{\gamma-\zeta}} \bigg\|_{L^p(D_h,dz')} \lesssim t^{\frac{\eta-\gamma}{2}} \$ f\$_{\eta,T,D}\;,
\end{align}
where
$$D= [t-2\lambda^2,t+\lambda^2]\times B(x,2\lambda)\;,\quad D_h = [t-2\lambda^2+|h|^2,t+\lambda^2-|h|^2]\times B(x,2\lambda-|h|)\;,$$
and $\$ f\$_{\eta,T,D}$ stands for the $\cD^{\gamma,\eta,T}_p$-norm where the integrals are restricted to the set $D$. A careful inspection of the proof of the reconstruction theorem~\cite[Th 3.1]{Recons} in the unweighted space $\cD^\gamma_p$ shows that for constructing the quantities $\langle \cR f, \varphi^\lambda_z\rangle$ for all $\varphi\in\ccB^r$, the finiteness of the l.h.s. of \eqref{Eq:BdLoc} suffices. This defines a distribution $\cR f$ on the set of all test functions supported in $\big((-\infty,0)\cup (0,T)\big)\times\R^d$: indeed, any such test function can be decomposed into a sum of \textit{finitely} many test functions of the form $\varphi^\lambda_z$ satisfying the assumption above and on which the action of $\cR f$ has been defined.
The reconstruction bound from~\cite[Th 3.1]{Recons} together with \eqref{Eq:BdLoc} ensures that \eqref{Eq:ReconsBound} is satisfied.\\
We now aim at applying Lemma \ref{Lemma:Extension}. If $\alpha\wedge\eta < 0$, then for all $\zeta\in\cA$ we have $\zeta-\alpha\wedge\eta \ge 0$ and therefore
\begin{align*}
\bigg\| \sup_{\varphi\in \ccB^r} \frac{\big|\langle \Pi_z f(z), \varphi_z^\lambda \rangle\big|}{\lambda^{\alpha\wedge\eta}} \bigg\|_{L^p((3\lambda^2,T-\lambda^2)\times\T^d,dz)} &\lesssim \sum_\zeta \bigg\| \sup_{\varphi\in \ccB^r} \frac{\big|\langle \Pi_z f_\zeta(z), \varphi_z^\lambda \rangle\big|}{\lambda^{\zeta} \, t^{\frac{\eta-\zeta}{2}}} \frac{\lambda^{\zeta-\alpha\wedge\eta}}{t^{\frac{\zeta-\eta}{2}}} \bigg\|_{L^p((3\lambda^2,T-\lambda^2)\times\T^d,dz)}\\
&\lesssim \$f\$_{\eta,T}\;.
\end{align*}
If $\alpha\wedge\eta =0$, then the same type of computation with $\alpha\wedge\eta$ replaced by $\bar\alpha <0$ still applies.\\
If $\alpha\wedge\eta > 0$, then $\min \cA = 0$ so that the same computation works with $\alpha\wedge\eta$ replaced by $0$. Furthermore, if the test function $\varphi$ belongs to $\ccB^r_{\lfloor \alpha\wedge\eta \rfloor}$ then $\langle \Pi_z f(z) , \varphi_z^\lambda \rangle = \sum_{\zeta\ge \alpha \wedge \eta} \langle \Pi_z f_\zeta (z),\varphi_z^\lambda \rangle$, and the same computation as above carries through with $\alpha\wedge \eta$. Finally, for any $c'>0$ we have:
\begin{align*}
&\Big\| \sup_{\varphi\in \ccB^r} \frac{\big|\langle \Pi_z f(z), \varphi^{2^{-m}}_z \rangle\big|}{2^{-m\eta}} \Big\|_{L^p\big((3\cdot 2^{-2m}\wedge (T-2^{-2m}),c' \cdot 2^{-2m}\wedge (T-2^{-2m}))\times \T^d,dz\big)}\\
&\lesssim \sum_{\zeta \in \cA} \Big\| \frac{|f(z)|_\zeta}{2^{-m(\eta-\zeta)}} \Big\|_{L^p\big((3\cdot 2^{-2m}\wedge (T-2^{-2m}),c' \cdot 2^{-2m}\wedge (T-2^{-2m}))\times \T^d,dz\big)} \\
&\lesssim \sum_{\zeta \in \cA} \Big\| \frac{|f(z)|_\zeta}{t^{\frac{\eta-\zeta}{2}}} \Big\|_{L^p\big((3\cdot 2^{-2m}\wedge (T-2^{-2m}),c' \cdot 2^{-2m}\wedge (T-2^{-2m}))\times \T^d,dz\big)}\;,
\end{align*}
so that bounding the $\ell^\infty$-norm by the $\ell^p$-norm, we get:
\begin{align*}
\sup_{m\ge 0} \Big\| \sup_{\varphi\in \ccB^r} \frac{\big|\langle \Pi_z f(z), \varphi^{2^{-m}}_z \rangle\big|}{2^{-m\eta}} \Big\|_{L^p\big((3\cdot 2^{-2m}\wedge (T-2^{-2m}),c' \cdot 2^{-2m}\wedge (T-2^{-2m}))\times \T^d,dz\big)} \lesssim \sum_{\zeta \in \cA} \Big\| \sup_{\varphi\in \ccB^r} \frac{|f(z)|_\zeta}{t^{\frac{\eta-\zeta}{2}}} \Big\|_{L^p((0,T)\times \T^d,dz)}\;.
\end{align*}
In any case, by combining the bounds we have just obtained with the reconstruction bound \eqref{Eq:ReconsBound}, we deduce that the conditions required in Lemma \ref{Lemma:Extension} are met, thus yielding the extension of $\cR f$ as an element of $\cB^{\bar\alpha,T}_p$ with $\bar\alpha$ as in the statement of the theorem.\\
Applying Lemma \ref{Lemma:SpecificConvol}, we deduce the statement of Proposition \ref{Prop:Recons} in the case of a single model.\\
Finally, we treat the case where we are given two models by using the bound already obtained in the unweighted case~\cite[Th 3.1]{Recons}: the bound \eqref{Eq:ReconsBoundTwo}, as well as the two-models counterpart of \eqref{Eq:ReconsConvol} and \eqref{Eq:ReconsConvol2}, easily follow using the same argument as above.
\end{proof}

\subsection{Embedding Theorem}\label{Subsec:Embeddings}

For classical Besov spaces, the difficulty of the proof of the embedding theorem varies according to the definition of the Besov-norm one opted for: when the norm is ``countable", the proof is simple as it essentially relies on the embedding properties of $\ell^p$-type spaces. In~\cite[Th 5.1]{Recons}, embedding theorems were obtained for the unweighted spaces $\cD^\gamma_{p,q}$. The main idea of the proof therein is the following: if one defines a space of averages $\bar\cD^\gamma_{p,q}$ (whose elements are defined on a countable set that approximates $\R\times\R^d$) endowed with a ``countable" norm, then the proof of the embedding theorem at the level of this space is simple. The important step is then the equivalence between the space of averages and the space $\cD^\gamma_{p,q}$.\\

We adapt this proof to our setting. In comparison with the original proof, the main technical difficulty comes from the weights near $t=0+$ which need some extra care. For simplicity, we assume that $T=1$ in this subsection and we drop the superscript $T$ in the spaces $\cD^{\gamma,\eta,T}_p$. This is a harmless assumption since the general case $T >0$ can be treated by considering the smallest $n_T\in \Z$ such that $T\ge 2^{-2n_T}$ and by considering the slightly modified grids $\Lambda_n = \{(k_0 2^{-2n}T, k_1 2^{-n}\sqrt T, \ldots, k_d 2^{-n} \sqrt T): k\in \Z^{d+1}\}$ for every $n\ge n_T$.\\

For every $n\ge 0$, 
%we introduce the grid
%$$ \Lambda_n := \{(k_0 2^{-2n}, k_1 2^{-n}, k_2 2^{-n}, k_3 2^{-n}):  k=(k_0,\ldots,k_3)\in \Z^4\}\;,$$
%as well as the set
we define
$$ \cE_n := \big\{ h\in \Lambda_n: 0 <  |h|_{\s} \le 2^{-n|\s|}\big\}\;,$$
where we recall that $\Lambda_n$ was defined in \eqref{eq:defLambda}.

For every $n_0\ge 0$, we introduce the following restriction of the grid
$$ \tilde\Lambda_n := \Lambda_n \cap [3\cdot 2^{-2n}, 1-2\cdot 2^{-2n}]\times \T^d\;,$$
as well as the associated ``boundary'':
$$ \partial \tilde\Lambda_n := \tilde\Lambda_n \cap \Big([3\cdot 2^{-2n}, 3\cdot 2^{-2(n-1)}]\times \T^d\Big)\;.$$
We then denote by $\ell^p_{n}(\tilde\Lambda_{n})$ the set of all sequences $u(z)$, $z\in\tilde\Lambda_{n}$ such that
$$ \| u \|_{\ell^p_{n}} = \Big(\sum_{z\in\tilde\Lambda_{n}} 2^{-n|\s|} |u(z)|^p \Big)^{\frac1{p}} < \infty\;.$$
We take a similar definition for $\ell^p_{n}(\partial\tilde\Lambda_{n})$.

With these notations at hand, we recast the definition of the space of averages in our context with weights. We let $\bar{\cD}^{\gamma,\eta}_{p}$ be the set of all sequences $(\averag{f}{n})_{n\ge 0}$ of maps $\averag{f}{n}:\tilde\Lambda_{n} \rightarrow \cT_{<\gamma}$ such that for all $\zeta \in \cA_\gamma$ we have:
\begin{enumerate}
\item Local bound:
$$ \bigg(\sum_{n\ge 0} \Big\| \frac{|\averag{f}{n}(z)|_\zeta}{t^{\frac{\eta-\zeta}{2}}} \Big\|_{\ell^p_{n}(\partial \tilde\Lambda_n)}^p\bigg)^{\frac1{p}} < \infty\;,$$
\item Translation bound:
$$ \sup_{n\ge 0} \sup_{h\in \cE_n} \Big\| \frac{|\averag{f}{n}(z+h)-\Gamma_{z+h,z} \averag{f}{n}(z)|_\zeta}{2^{-n(\gamma-\zeta)}t^{\frac{\eta-\gamma}{2}}} \Big\|_{\ell^p_{n}(\tilde\Lambda_n)} < \infty\;,$$
\item Consistency bound:
$$ \sup_{n\ge 0} \Big\| \frac{|\averag{f}{n}(z)-\averag{f}{n+1}(z)|_\zeta}{2^{-n(\gamma-\zeta)} t^{\frac{\eta-\gamma}{2}}} \Big\|_{\ell^p_{n}(\tilde\Lambda_n)} < \infty\;.$$
\end{enumerate}
We denote by $\$ \averag{f}{}\$$ the corresponding norm. We now follow the strategy of proof of~\cite[Th 5.1]{Recons} by adapting the intermediary results.\\
Observe that if we set $\cE_n^C := \{h\in \Lambda_n: 0< |h|_\s \le C 2^{-n}\}$, then we have for any given $C>0$ the bound
\begin{equation}\label{Eq:TransConsis}
\sup_{n\ge 0} \sup_{h\in \cE_n^C} \Big\| \tun_{\{z+h \in \tilde{\Lambda}_{n}\}}\frac{|\averag{f}{n+1}(z+h)-\Gamma_{z+h,z} \averag{f}{n}(z)|_\zeta}{2^{-n(\gamma-\zeta)}t^{\frac{\eta-\gamma}{2}}} \Big\|_{\ell^p_{n}(\tilde\Lambda_n)} < \infty\;.
\end{equation}
We first show that the spaces $\cD^{\gamma,\eta}_{p}$ and $\bar\cD^{\gamma,\eta}_{p}$ are essentially equivalent. Let us set
$$ B(z,r)^+ := B(z,r) \cap \{z'=(t',x') \in \R^{d+1}: t' \ge t\}\;.$$
Notice that the volume of $B(z,2^{-n})$ is $2^{d+1} 2^{-n|\s|}$, while the volume of $B(z,2^{-n})^+$ is $2^{d} 2^{-n|\s|}$.
\begin{proposition}\label{Prop:AveragEq}
Let $f\in\cD^{\gamma,\eta}_{p}$ and set for every $n\ge 0$
$$ \averag{f}{n}(z) := \int_{B(z,2^{-n})^+} 2^{-d} 2^{n|\s|} \Gamma_{z,z'} f(z') dz'\;,\quad z\in \tilde\Lambda_{n}\;.$$
Then $\averag{f}{} \in \bar{\cD}^{\gamma,\eta}_{p}$.\\
Conversely, let $\averag{f}{}\in\bar{\cD}^{\gamma,\eta}_{p}$ and set
$$ f_{n}(z) = \Gamma_{z,z_n} \averag{f}{n}(z_n)\;,$$
where $z_n$ is the nearest point to $z$ on the grid $\tilde\Lambda_{n}$. Then, for every $n_0\ge 0$, the sequence $(f_{n})_{n\ge n_0}$ converges in $L^p((3\cdot 2^{-2n_0},1)\times\T^d)$ to an element $f\in \cD^{\gamma,\eta}_p$.\\
If $\averag{f}{}$ is obtained from some $f\in\cD^{\gamma,\eta}_{p}$ as in the first part of the statement, then the sequence $f_n$ converges to the same map $f$.
\end{proposition}
\begin{proof}
We start with the first part of the proposition. To get the local bound on $\averag{f}{}$, we write
\begin{align*}
\Big\| \frac{|\averag{f}{n}(z)|_\zeta}{t^{\frac{\eta-\zeta}{2}}} \Big\|_{\ell^p_{n}(\partial \tilde\Lambda_n)} &\le \Big\| \int_{B(z,2^{-n})^+} 2^{-d} 2^{n|\s|}\frac{|\Gamma_{z,z'}f(z')|_\zeta}{t^{\frac{\eta-\zeta}{2}}} dz' \Big\|_{\ell^p_{n}(\partial \tilde\Lambda_n)}\\
&\lesssim \sum_{\nu \ge \zeta} \Big\| \int_{B(z,2^{-n})^+} 2^{-d} 2^{n|\s|}\frac{|f(z')|_\nu}{t^{\frac{\eta-\nu}{2}}} dz' \Big\|_{\ell^p_{n}(\partial \tilde\Lambda_n)}\;,
\end{align*}
where we have used the fact that $t$ is of order $2^{-2n}$ in the above integral. Applying Jensen's inequality on the integral over $z'$, we obtain the further bound
$$ \lesssim \sum_{\nu \ge \zeta} \bigg(\int_{z' \in \big((3\cdot 2^{-2n}, 13\cdot 2^{-2n})\cap (0,1)\big)\times \T^d} \Big(\frac{|f(z')|_\nu}{(t')^{\frac{\eta-\nu}{2}}}\Big)^p dz' \bigg)^{\frac1{p}}\;,$$
uniformly over all $n\ge 0$. Consequently, we find
$$ \Big(\sum_{n\ge 0} \Big\| \frac{|\averag{f}{n}(z)|_\zeta}{t^{\frac{\eta-\zeta}{2}}} \Big\|_{\ell^p_{n}(\partial \tilde\Lambda_n)}^p\Big)^{\frac1{p}}  \lesssim \sum_\nu\Big\| \frac{|f(z)|_\nu}{t^{\frac{\eta-\nu}{2}}} \Big\|_{L^p((0,1)\times\T^d})\;,$$
as required. We turn to the translation bound. For all $h\in \cE_n$ we write
$$ \averag{f}{n}(z+h) - \Gamma_{z+h,z} \averag{f}{n}(z) = \int_{u\in B(0,2^{-n})^+} 2^{-d} 2^{n|\s|} \Gamma_{z+h,z+h+u}\big(f(z+h+u) - \Gamma_{z+h+u,z+u} f(z+u)\big) du\;.$$
Therefore, using Jensen's inequality at the first line we get
\begin{align*}
&\Big\| \frac{|\averag{f}{n}(z+h)-\Gamma_{z+h,z} \averag{f}{n}(z)|_\zeta}{2^{-n(\gamma-\zeta)}t^{\frac{\eta-\gamma}{2}}} \Big\|_{\ell^p_{n}(\tilde\Lambda_n)}\\
&\lesssim \bigg(\sum_{z\in\tilde{\Lambda}_n} \int_{u\in B(0,2^{-n})^+} \Big(\frac{\Big|\Gamma_{z+h,z+h+u}\big(f(z+h+u) - \Gamma_{z+h+u,z+u} f(z+u)\big)\Big|_\zeta}{|h|^{\gamma-\zeta} t^{\frac{\eta-\gamma}{2}}}\Big)^p du\bigg)^{\frac1{p}}\\
&\lesssim \sum_{\nu\ge \zeta}\bigg(\int_{z\in (3\cdot 2^{-2n}, 1-2^{-2n})\times \T^d} \Big(\frac{\big|f(z+h) - \Gamma_{z+h,z} f(z)\big|_\nu}{|h|^{\gamma-\nu} t^{\frac{\eta-\gamma}{2}}}\Big)^p dz \bigg)^{\frac1{p}}\;.
\end{align*}
Therefore
$$ \Big(\sum_{n\ge 0} \Big\| \frac{|\averag{f}{n}(z+h)-\Gamma_{z+h,z} \averag{f}{n}(z)|_\zeta}{2^{-n(\gamma-\zeta)}t^{\frac{\eta-\gamma}{2}}} \Big\|_{\ell^p_{n}(\tilde\Lambda_n)}^p\Big)^{\frac1{p}} \lesssim \$f\$\;,$$
as required. The consistency bound is obtained similarly so we skip the details.\\
Let us now prove the second part of the statement. We first show that $(f_n)_{n\ge n_0}$ is a Cauchy sequence in $L^p((3\cdot 2^{-2n_0},1)\times\T^d)$. Fix $n_0\ge 0$. We have for every $n\ge n_0$:
\begin{align*}
&\Big\| \frac{\big|f_{n+1}(z) - f_n(z)\big|_\zeta}{t^{\frac{\eta-\zeta}{2}}} \Big\|_{L^p((3\cdot 2^{-2n_0},1)\times\T^d)}\\
&= \Big\| \frac{\big|\Gamma_{z,z_{n+1}}\big(\averag{f}{n+1}(z_{n+1}) - \Gamma_{z_{n+1},z_n}\averag{f}{n}(z_n)\big)\big|_\zeta}{t^{\frac{\eta-\zeta}{2}}} \Big\|_{L^p((3\cdot 2^{-2n_0},1)\times\T^d)}\\
&\lesssim \sum_{\nu \ge \zeta} 2^{-n(\nu-\zeta)} \Big(\int_{z\in (3\cdot 2^{-2n_0},1)\times\T^d} \Big(\frac{\big|\averag{f}{n+1}(z_{n+1}) - \Gamma_{z_{n+1},z_n}\averag{f}{n}(z_n)\big|_\nu}{t^{\frac{\eta-\zeta}{2}}}\Big)^p dz\Big)^{\frac1{p}}\;.
\end{align*}
At this point, we use the fact that there exists $C>0$ such that $|z_{n+1}-z_n|$ is bounded by $C 2^{-(n+1)}$ uniformly over all $z$. We thus get the further bound
\begin{align*}
&\lesssim \sum_{\nu \ge \zeta} 2^{-n(\nu-\zeta)} \Big(\sum_{z_n \in \tilde{\Lambda}_n\cap (3\cdot 2^{-2n_0},1)\times\T^d} 2^{-n|\s|} \sum_{h\in \cE_{n+1}^C} \Big(\frac{\big|\averag{f}{n+1}(z_n+h) - \Gamma_{z_{n}+h,z_n}\averag{f}{n}(z_n)\big|_\nu}{t^{\frac{\eta-\zeta}{2}}}\Big)^p \Big)^{\frac1{p}}\\
&\lesssim \sum_{\nu \ge \zeta} 2^{-n(\gamma-\zeta)} \sup_{h\in \cE_{n+1}^C} \Big(\sum_{z_n \in \tilde{\Lambda}_n\cap (3\cdot 2^{-2n_0},1)\times\T^d} 2^{-n|\s|} \Big(\frac{\big|\averag{f}{n+1}(z_n+h) - \Gamma_{z_{n}+h,z_n}\averag{f}{n}(z_n)\big|_\nu}{|h|^{\gamma-\nu} t^{\frac{\eta-\gamma}{2}} t^{\frac{\gamma-\zeta}{2}}}\Big)^p \Big)^{\frac1{p}}\\
&\lesssim 2^{-(n-n_0)(\gamma-\zeta)} \sum_{\nu \ge \zeta} \sup_{h\in \cE_{n+1}^C} \Big(\sum_{z_n \in \tilde{\Lambda}_n\cap (3\cdot 2^{-2n_0},1)\times\T^d} 2^{-n|\s|} \Big(\frac{\big|\averag{f}{n+1}(z_n+h) - \Gamma_{z_{n}+h,z_n}\averag{f}{n}(z_n)\big|_\nu}{|h|^{\gamma-\nu} t^{\frac{\eta-\gamma}{2}}}\Big)^p \Big)^{\frac1{p}}\;,
\end{align*}
where we have used the fact that $t\ge 2^{-2n_0}$ at the last line. We deduce that $(f_n)_{n\ge n_0}$ is a Cauchy sequence in $L^p([3\cdot 2^{-2n_0},1]\times\T^d)$, for every $n_0\ge 0$. We then get an element $f\in L^p((0,1)\times\T^3)$ and it remains to show that it actually belongs to $\cD^{\gamma,\eta}_p$. The local bound is already proved, let us focus on the translation bound. For every $h\in B(0,1)$ and every $z\in [3\cdot |h|^2,1-|h|^2]\times\T^d$, let $n_0\ge 0$ be the smallest integer such that $6\cdot 2^{-2n_0} \le 2|h|^2$. Then, we write
\begin{equation}\label{Eq:Decompfaveragf}\begin{split}
f(z+h) - \Gamma_{z+h,z} f(z) &= \big(f(z+h)-f_{n_0}(z+h)\big) + \big(f_{n_0}(z+h) - \Gamma_{z+h,z} f_{n_0}(z)\big)\\
&+ \big(\Gamma_{z+h,z}(f_{n_0}(z) - f(z))\big)\;.
\end{split}\end{equation}
We bound separately the three terms appearing on the r.h.s. Regarding the first term, we use the previous calculation to get
\begin{align*}
\Big\| \frac{\big|f(z+h)-f_{n_0}(z+h)\big|_\zeta}{|h|^{\gamma-\zeta} t^{\frac{\eta-\gamma}{2}}} \Big\|_{L^p((3\cdot |h|^2,1-|h|^2)\times\T^d)} &\le \sum_{n\ge n_0} \Big\| \frac{\big|f_{n+1}(z+h)-f_{n}(z+h)\big|_\zeta}{|h|^{\gamma-\zeta} t^{\frac{\eta-\gamma}{2}}} \Big\|_{L^p((3\cdot |h|^2,1-|h|^2)\times\T^d)} \\
&\le \sum_{n\ge n_0} \Big\| \frac{\big|f_{n+1}(z)-f_{n}(z)\big|_\zeta}{|h|^{\gamma-\zeta} t^{\frac{\eta-\gamma}{2}}} \Big\|_{L^p((3\cdot 2^{-2n_0},1)\times\T^d)}\\
&\le \sum_{n\ge n_0} 2^{-(n-n_0)(\gamma-\zeta)} \$ \averag{f}{}\$ \;,
\end{align*}
so it is bounded by a term of order $\$ \averag{f}{}\$$ as required. The bound of the third term of \eqref{Eq:Decompfaveragf} is similar. Let us now consider the second term of \eqref{Eq:Decompfaveragf}. We have
\begin{align*}
&\Big\| \frac{\big|f_{n_0}(z+h)-\Gamma_{z+h,z}f_{n_0}(z)\big|_\zeta}{|h|^{\gamma-\zeta} t^{\frac{\eta-\gamma}{2}}} \Big\|_{L^p([3\cdot |h|^2,1-|h|^2]\times\T^d)}\\
&\le \sum_{\nu \ge \zeta} \sup_{\tilde{h} \in \cE_{n_0}^C} \Big\| 2^{-n_0(\nu-\zeta)}\frac{\big|\averag{f}{n_0}(z_{n_0}+\tilde{h})-\Gamma_{z_{n_0}+\tilde{h},z_{n_0}}\averag{f}{n_0}(z_{n_0})\big|_\nu}{|h|^{\gamma-\zeta} t^{\frac{\eta-\gamma}{2}}} \Big\|_{L^p((3\cdot |h|^2,1-|h|^2)\times\T^d)} \\
&\le \sum_{\nu \ge \zeta} \sup_{\tilde{h} \in \cE_{n_0}^C} \Big\| \frac{\big|\averag{f}{n_0}(z+\tilde{h})-\Gamma_{z+\tilde{h},z}\averag{f}{n_0}(z)\big|_\nu}{|h|^{\gamma-\nu} t^{\frac{\eta-\gamma}{2}}} \Big\|_{\ell^p(\tilde{\Lambda}_{n_0})}\;,
\end{align*}
which is bounded by $\$ \averag{f}{}\$$ uniformly over all $n_0\ge 0$, as required. This completes the proof of the translation bound.\\
Finally, if $\averag{f}{}$ is constructed from some element $f\in \cD^{\gamma,\eta}_p$, then a simple computation shows that the sequence $(f- f_n)_{n\ge n_0}$ converges to $0$ in $L^p([3\cdot 2^{-2n_0},1]\times \T^d)$ for every $n_0\ge 0$. This completes the proof of the proposition.
\end{proof}

Let us state a useful bound for the sequel. For all $p \le p' \in [1,\infty]$, we have
\begin{equation}\label{Eq:Boundellp}
\Big\| u(z) \Big\|_{\ell^{p'}_n(\tilde\Lambda_n)} \le 2^{n|\s|(\frac1{p}-\frac1{p'})}  \Big\| u(z) \Big\|_{\ell^{p}_n(\tilde\Lambda_n)}\;,
\end{equation}
uniformly over all sequences $u(z), z\in \tilde\Lambda_n$. Of course, this remains true if $\tilde\Lambda_n$ is replaced by $\partial \tilde\Lambda_n$.\\

We have all the elements at hand to prove the embedding theorem.
\begin{proof}[Proof of Theorem \ref{Th:EmbeddingW}]
The first embedding is a direct consequence of the boundedness of the underlying space which implies the continuous inclusion $L^p \subset L^{p'}$ whenever $p' < p$.\\
The second embedding is more involved and relies on the spaces of averages $\bar{\cD}^{\gamma,\eta}_p$. If we establish the embedding at the level of these spaces, namely
$$ \$ \averag{f}{} \$_{\gamma',\eta',p'} \lesssim \$ \averag{f}{} \$_{\gamma,\eta,p}\;,$$
uniformly over all $\averag{f}{} \in \bar{\cD}^{\gamma,\eta}_p$, then the equivalence stated in Proposition \ref{Prop:AveragEq} yields the desired result.\\

Let us introduce the notation: $(\gamma,\eta,p) \leadsto (\gamma',\eta',p')$ if we have $\eta'-\eta = \gamma'-\gamma$ as well as
$$ \gamma' = \gamma - |\s| \big(\frac1{p} - \frac1{p'}\big)\;,\quad p'<\infty\;,$$
or
$$ \gamma' \le \gamma - \frac{|\s|}{p}\;,\quad p'=\infty\;.$$
Notice that, for all $\zeta < \gamma$, there is a unique $p(\zeta)\in [p,\infty]$ such that $(\gamma,\eta,p) \leadsto (\zeta,\eta+\zeta-\gamma,p(\zeta))$.\\

We let $\zeta_1 > \zeta_2 > \ldots$ be the elements of $\cA_\gamma$ listed in decreasing order. We are going to show the following two properties:\begin{enumerate}
\item[(1)] If $\gamma',\eta',p'$ are such that $\gamma' \in (\zeta_1,\gamma)$, $p' \in (p,\infty]$ and $(\gamma,\eta,p)\leadsto (\gamma',\eta',p')$, then $\averag{f}{} \in \bar{\cD}^{\gamma',\eta'}_{p'}$.
\item[(2)] If $\gamma' \in (\zeta_2,\zeta_1)$ and if $\eta'-\eta=\gamma'-\gamma$, then $\averag{f}{} \in \bar{\cD}^{\gamma',\eta'}_{p(\zeta_1)}$.
\end{enumerate}

Let us show that these two properties, combined with a simple recursion, yield the second embedding of the theorem. Let $p'' < p$, $\gamma'' < \gamma - |\s| \big(\frac1{p} - \frac1{p''})$ and $\eta'' = \eta + \gamma''-\gamma$: we aim at showing that $\bar{\cD}^{\gamma,\eta}_{p}$ is continuously embedded into $\bar{\cD}^{\gamma'',\eta''}_{p''}$. We distinguish four cases.

First, if $p'' < p({\zeta_1})$ then (1) shows that $\bar{\cD}^{\gamma,\eta}_{p}$ is continuously embedded into $\bar{\cD}^{\gamma',\eta'}_{p'}$ with $p'=p''$, $\gamma' = \gamma - |\s| \big(\frac1{p} - \frac1{p''})$ and $\eta'=\eta+\gamma'-\gamma$. The latter space is itself continuously embedded into $\bar{\cD}^{\gamma'',\eta''}_{p''}$ so that the desired embedding is proved.

Second, if $p'' = p({\zeta_1})$ then (2) shows that $\bar{\cD}^{\gamma,\eta}_{p}$ is continuously embedded into $\bar{\cD}^{\gamma',\eta'}_{p''}$ for all $\gamma' \in (\zeta_2,\zeta_1)$ and all $\eta'=\eta+\gamma'-\gamma$. If $\gamma'' \in (\zeta_2,\zeta_1)$ then the desired embedding is proved. If $\gamma'' < \zeta_2$, then it is a consequence of the continuous embedding of $\bar{\cD}^{\gamma',\eta'}_{p''}$ into $\bar{\cD}^{\gamma'',\eta''}_{p''}$ whenever $\gamma''-\gamma' = \eta''-\eta' < 0$.

Third, if $p(\zeta_2) > p'' > p({\zeta_1})$ then there exists $\gamma'\in (\zeta_2,\zeta_1)$ such that, taking $\eta'=\eta + \gamma'-\gamma$, we have $(\gamma',\eta',p(\zeta_1)) \leadsto (\gamma'',\eta'',p'')$ so that combining (2) and (1) the desired embedding follows.

Fourth, if $p'' \ge p(\zeta_2)$ then a simple recursion based on the three first cases yields the desired result.\\

Let us prove (1). We start with the local bound. Applying \eqref{Eq:Boundellp} and using the fact that $t$ is of order $2^{-2n}$ for all $z=(t,x) \in \partial \tilde\Lambda_n$, we get
\begin{align*}
\Big\| \frac{|\averag{f}{n}(z)|_\zeta}{t^{\frac{\eta'-\zeta}{2}}} \Big\|_{\ell^{p'}_{n}(\partial \tilde\Lambda_n)} \le 2^{n|\s|(\frac1{p}-\frac1{p'})} \Big\| \frac{|\averag{f}{n}(z)|_\zeta}{t^{\frac{\eta'-\zeta}{2}}} \Big\|_{\ell^{p}_{n}(\partial \tilde\Lambda_n)}\lesssim \Big\| \frac{|\averag{f}{n}(z)|_\zeta}{t^{\frac{\eta-\zeta}{2}}} \Big\|_{\ell^{p}_{n}(\partial \tilde\Lambda_n)}\;.
\end{align*}
Using the classical embedding from $\ell^p$ into $\ell^{p'}$, we deduce the following bound
\begin{align*}
\bigg(\sum_{n\ge 0} \Big\| \frac{|\averag{f}{n}(z)|_\zeta}{t^{\frac{\eta'-\zeta}{2}}} \Big\|_{\ell^{p'}_{n}(\partial \tilde\Lambda_n)}^{p'}\bigg)^{\frac1{p'}} \lesssim \bigg(\sum_{n\ge 0} \Big\| \frac{|\averag{f}{n}(z)|_\zeta}{t^{\frac{\eta-\zeta}{2}}} \Big\|_{\ell^{p}_{n}(\partial \tilde\Lambda_n)}^{p}\bigg)^{\frac1{p}}\;.
\end{align*}
We pass to the translation bound. Applying \eqref{Eq:Boundellp}, we obtain
\begin{align*}
\Big\| \frac{|\averag{f}{n}(z+h)-\Gamma_{z+h,z} \averag{f}{n}(z)|_\zeta}{2^{-n(\gamma'-\zeta)}t^{\frac{\eta'-\gamma'}{2}}} \Big\|_{\ell^{p'}_{n}(\tilde\Lambda_n)} &\lesssim 2^{n|\s|(\frac1{p}-\frac1{p'})} \Big\| \frac{|\averag{f}{n}(z+h)-\Gamma_{z+h,z} \averag{f}{n}(z)|_\zeta}{2^{-n(\gamma'-\zeta)}t^{\frac{\eta'-\gamma'}{2}}} \Big\|_{\ell^{p}_{n}(\tilde\Lambda_n)}\\
&\lesssim \Big\| \frac{|\averag{f}{n}(z+h)-\Gamma_{z+h,z} \averag{f}{n}(z)|_\zeta}{2^{-n(\gamma-\zeta)}t^{\frac{\eta-\gamma}{2}}} \Big\|_{\ell^{p}_{n}(\tilde\Lambda_n)}\;,
\end{align*}
as required. The consistency bound is obtained similarly. This completes the proof of (1).\\

Instead of proving (2) directly, we show the following:\begin{enumerate}
\item [(2')] For all $\epsilon > 0$ and all $\gamma'' \in (\zeta_2,\zeta_1)$, we have $\averag{f}{}\in \bar{\cD}^{\gamma'',\eta''}_{p(\zeta_1+\epsilon)}$ where $\eta'' = \eta + \gamma'' - \gamma$.
\end{enumerate}
If (2') holds true, then we can apply (1) and deduce that $\averag{f}{}\in \bar{\cD}^{\gamma(\epsilon),\eta(\epsilon)}_{p(\zeta_1)}$ with
$$ \gamma(\epsilon) = \gamma'' - |\s| \big(\frac1{p(\zeta_1+\epsilon)} - \frac1{p(\zeta_1)}\big)\;,\quad \eta(\epsilon) = \eta'' + \gamma(\epsilon) - \gamma''\;.$$
Since $\gamma(\epsilon) \uparrow \gamma''$ and $\eta(\epsilon) \uparrow \eta''$ as $\epsilon \downarrow 0$, Property (2) follows.\\
We are left with proving (2'). The local bound follows from exactly the same argument as in the proof of (1). Let us focus on the translation bound. We let $\averag{f}{}_{<\zeta_1}$ be the restriction of $\averag{f}{}$ to $\cT_{<\zeta_1}$. We have for all $\zeta < \zeta_1$:
\begin{align}\label{Eq:BdTranslationEmbed}
\big| \averag{f}{n}_{<\zeta_1}(z+h) - \Gamma_{z+h,z} \averag{f}{n}_{<\zeta_1}(z)\big|_\zeta \le& \big| \averag{f}{n}(z+h) - \Gamma_{z+h,z} \averag{f}{n}(z)\big|_\zeta+ \big|\Gamma_{z+h,z} \averag{f}{n}_{\zeta_1}(z)\big|_\zeta\;.
\end{align}
We will bound the contributions to the translation bound of these two terms separately. Applying \eqref{Eq:Boundellp} and since $\zeta_1+\epsilon - \gamma \le -|\s|(1/p - 1/p(\zeta_1+\epsilon))$, we get for all $h\in \cE_n$
\begin{align*}
\Big\| \frac{\big|\averag{f}{n}(z+h)-\Gamma_{z+h,z} \averag{f}{n}(z)\big|_\zeta}{|h|^{\gamma''-\zeta}t^{\frac{\eta''-\gamma''}{2}}} \Big\|_{\ell^{p(\zeta_1+\epsilon)}_{n}(\tilde{\Lambda}_n)} &\lesssim 2^{n|\s|\big(\frac1{p}-\frac1{p(\zeta_1+\epsilon)}\big)}\Big\| \frac{\big|\averag{f}{n}(z+h)-\Gamma_{z+h,z} \averag{f}{n}(z)\big|_\zeta}{|h|^{\gamma''-\zeta}t^{\frac{\eta-\gamma}{2}}} \Big\|_{\ell^{p}_{n}(\tilde{\Lambda}_n)}\\
&\lesssim 2^{n(\gamma''-\zeta_1-\epsilon)}\Big\| \frac{\big|\averag{f}{n}(z+h)-\Gamma_{z+h,z} \averag{f}{n}(z)\big|_\zeta}{|h|^{\gamma-\zeta}t^{\frac{\eta-\gamma}{2}}} \Big\|_{\ell^{p}_{n}(\tilde{\Lambda}_n)}\;,
\end{align*}
uniformly over all $n\ge 0$. This ensures that the supremum over $n$ of the l.h.s.~is bounded by a term of order $\$\bar f\$$.\\
We turn to the contribution coming from the second term of \eqref{Eq:BdTranslationEmbed}. We have
\begin{align*}
\Big\| \frac{\big|\Gamma_{z+h,z} \averag{f}{n}_{\zeta_1}(z)\big|_\zeta}{|h|^{\gamma''-\zeta}t^{\frac{\eta''-\gamma''}{2}}} \Big\|_{\ell^{p(\zeta_1+\epsilon)}_{n}(\tilde{\Lambda}_n)} \lesssim 2^{-n(\zeta_1-\gamma'')}\Big\| \frac{\big|\averag{f}{n}(z)\big|_{\zeta_1}}{t^{\frac{\eta-\gamma}{2}}} \Big\|_{\ell^{p(\zeta_1+\epsilon)}_{n}(\tilde{\Lambda}_n)}\;,
\end{align*}
uniformly over all $h\in\cE_n$ and all $n\ge 0$. At this point, we subdivide $\tilde{\Lambda}_n$ into the union of its components on $D_{n_0} = [3\cdot 2^{-n_0},3\cdot 2^{-(n_0-1)}]\times\T^d$ with $n_0=0,1,\ldots,n$, and we bound separately the corresponding $\ell^{p(\zeta_1+\epsilon)}_n$ norm:
$$ \Big\| \frac{\big|\averag{f}{n}(z)\big|_{\zeta_1}}{t^{\frac{\eta-\gamma}{2}}} \Big\|_{\ell^{p(\zeta_1+\epsilon)}_{n}(\tilde{\Lambda}_n \cap D_{n_0})}\;.$$
Fix such an $n_0$. For every $z\in \tilde{\Lambda}_{n}$, we let $v_z := \inf\{v\in \tilde{\Lambda}_{n-1}: v \ge z\}$ with respect to the lexicographic order and we use the following decomposition
$$ \averag{f}{n}_{\zeta_1}(z) = \averag{f}{n-1}_{\zeta_1}(v_z) + \averag{f}{n}_{\zeta_1}(z) - \averag{f}{n-1}_{\zeta_1}(v_z)\;.$$
We have $\| \averag{f}{n}_{\zeta_1}(z) / t^{\frac{\eta-\gamma}{2}} \|_{\ell^{p(\zeta_1+\epsilon)}_{n}(\tilde{\Lambda}_n \cap D_{n_0})} \le A_1(n) + A_2(n)$ where
\begin{align*}
A_1(n) = \Big(\sum_{z \in \tilde{\Lambda}_n \cap D_{n_0}} 2^{-n|\s|} \Big|\frac{\averag{f}{n-1}_{\zeta_1}(v_z)}{t^{\frac{\eta-\gamma}{2}}}\Big|^{p(\zeta_1+\epsilon)} \Big)^\frac{1}{p(\zeta_1+\epsilon)}\;,\\
A_2(n) = \Big(\sum_{z \in \tilde{\Lambda}_n \cap D_{n_0}} 2^{-n|\s|} \Big|\frac{\averag{f}{n}_{\zeta_1}(z)-\averag{f}{n-1}_{\zeta_1}(v_z)}{t^{\frac{\eta-\gamma}{2}}}\Big|^{p(\zeta_1+\epsilon)} \Big)^\frac{1}{p(\zeta_1+\epsilon)}\;.
\end{align*}
Since for every vertex $v =(s,y)\in \tilde{\Lambda}_{n-1}$, there are at most $2^{|\s|}$ vertices $z\in \tilde{\Lambda}_{n}$ such that $v_z = v$, and since $\eta-\gamma\le 0$, we get whenever $n>n_0$
\begin{align*}
A_1(n) &\le \Big(\sum_{v \in \tilde{\Lambda}_{n-1}\cap D_{n_0}} 2^{-(n-1)|\s|} \Big|\frac{\averag{f}{n-1}_{\zeta_1}(v)}{s^{\frac{\eta-\gamma}{2}}}\Big|^{p(\zeta_1+\epsilon)} \Big)^\frac{1}{p(\zeta_1+\epsilon)} = \Big\| \frac{\averag{f}{n-1}_{\zeta_1}(z)}{t^{\frac{\eta-\gamma}{2}}} \Big\|_{\ell^{p(\zeta_1+\epsilon)}_{n-1}(\tilde{\Lambda}_{n-1} \cap D_{n_0})}\;.
\end{align*}
Regarding $A_2(n)$, using the fact that $\zeta_1 = \max \cA_\gamma$ and \eqref{Eq:Boundellp} at the third line, we get whenever $n>n_0$:
\begin{align*}
A_2(n) &\lesssim \sup_{h\in\cE_{n}^C} \Big(\sum_{v \in \tilde{\Lambda}_{n-1}\cap D_{n_0}} 2^{-(n-1)|\s|} \Big|\frac{\averag{f}{n}_{\zeta_1}(v+h)-\averag{f}{n-1}_{\zeta_1}(v)}{s^{\frac{\eta-\gamma}{2}}}\Big|^{p(\zeta_1+\epsilon)} \Big)^\frac{1}{p(\zeta_1+\epsilon)}\\
&\lesssim 2^{-n(\gamma-\zeta_1)} \sup_{h\in\cE_{n}^C} \Big(\sum_{v \in \tilde{\Lambda}_{n-1} \cap D_{n_0}} 2^{-(n-1)|\s|} \Big|\frac{\big|\averag{f}{n}(z)-\Gamma_{z,z+h}\averag{f}{n-1}(z+h)\big|_{\zeta_1}}{2^{-n(\gamma-\zeta_1)} s^{\frac{\eta-\gamma}{2}}}\Big|^{p(\zeta_1+\epsilon)} \Big)^\frac{1}{p(\zeta_1+\epsilon)}\\
&\lesssim 2^{-n\epsilon} \sup_{h\in\cE_{n}^C} \Big(\sum_{v \in \tilde{\Lambda}_{n-1} \cap D_{n_0}} 2^{-(n-1)|\s|} \Big|\frac{\big|\averag{f}{n}(z)-\Gamma_{z,z+h}\averag{f}{n-1}(z+h)\big|_{\zeta_1}}{2^{-n(\gamma-\zeta_1)} s^{\frac{\eta-\gamma}{2}}}\Big|^{p} \Big)^\frac{1}{p}\;,
\end{align*}
Iterating this, we get
\begin{align*}
\Big\| \frac{\big|\averag{f}{n}(z)\big|_{\zeta_1}}{t^{\frac{\eta-\gamma}{2}}} \Big\|_{\ell^{p(\zeta_1+\epsilon)}_{n}(\tilde{\Lambda}_n \cap D_{n_0})} &\le \Big\| \frac{\big|\averag{f}{n_0}(z)\big|_{\zeta_1}}{t^{\frac{\eta-\gamma}{2}}} \Big\|_{\ell^{p(\zeta_1+\epsilon)}_{n_0}(\tilde{\Lambda}_{n_0} \cap D_{n_0})}\\
&+ C' \sum_{m= n_0}^{n-1} \sup_{h\in\cE_{m}^C} 2^{-m\epsilon}\Big(\sum_{v \in \tilde{\Lambda}_{m} \cap D_{n_0}} 2^{-m|\s|} \Big|\frac{\big|\averag{f}{m+1}(v)-\Gamma_{v,v+h}\averag{f}{m}(v+h)\big|_{\zeta_1}}{2^{-m(\gamma-\zeta_1)} s^{\frac{\eta-\gamma}{2}}}\Big|^{p} \Big)^\frac{1}{p}\;.
\end{align*}
Applying \eqref{Eq:Boundellp} and using the fact that $t$ is of order $2^{-n_0}$ in $D_{n_0}$, we find
$$ \Big\| \frac{\big|\averag{f}{n_0}(z)\big|_{\zeta_1}}{t^{\frac{\eta-\gamma}{2}}} \Big\|_{\ell^{p(\zeta_1+\epsilon)}_{n_0}(\tilde{\Lambda}_{n_0} \cap D_{n_0})} \lesssim 2^{-n_0\epsilon}\Big\| \frac{\big|\averag{f}{n_0}(z)\big|_{\zeta_1}}{t^{\frac{\eta-\zeta_1}{2}}} \Big\|_{\ell^{p}_{n_0}(\tilde{\Lambda}_{n_0} \cap D_{n_0})}\;,$$
uniformly over all $n_0\ge 0$. Putting everything together, we thus get
\begin{align*}
\Big\| \frac{\big|\averag{f}{n}(z)\big|_{\zeta_1}}{t^{\frac{\eta-\gamma}{2}}} \Big\|_{\ell^{p(\zeta_1+\epsilon)}_{n}(\tilde{\Lambda}_n)} \le \bigg(\sum_{n_0=0}^n \Big\| \frac{\big|\averag{f}{n}(z)\big|_{\zeta_1}}{t^{\frac{\eta-\gamma}{2}}} \Big\|_{\ell^{p(\zeta_1+\epsilon)}_{n}(\tilde{\Lambda}_n \cap D_{n_0})}^p\bigg)^{\frac1{p}} &\lesssim \$ \averag{f}{}\$\;,
\end{align*}
uniformly over all $n\ge 0$. We thus get the desired translation bound. The consistency bound is obtained similarly. This concludes the proof of (2').
\end{proof}

\subsection{Product}
%Let us recall that a sector $V$ is a `sub-regularity structure", stable under $\cG$.
%
%\begin{theorem}\label{Th:Mult}
%Let $f_i \in \mathcal{D}^{\gamma_i,\eta_i,T}_{p_i,\alpha_i}(V_i)$, $i=1,2$, where $V_i$ is a sector or regularity $\alpha_i$. Then $f:= f_1 f_2$ belongs to $\mathcal{D}^{\gamma,\eta,T}_p$, where
%$$\gamma = (\gamma_1 +\alpha_2)\wedge (\gamma_2 + \alpha_1)\;,\quad \eta = \eta_1+\eta_2\;,\quad\frac{1}{p} = \frac{1}{p_1} +\frac{1}{p_2} \;.$$
%If we are given two models $(\Pi,\Gamma)$ and $(\bar\Pi,\bar\Gamma)$, then we have the bound
%\begin{align*}
%\$ f_1 f_2 ; g_1 g_2\$ &\lesssim \|\Gamma\|^2 \| f_1-g_1\| \$f_2\$ + \|\Gamma-\bar\Gamma\| \Big(\$g_1\$ \$f_2\$ + \$g_2\$ \$f_1\$\Big) + \$f_1;g_1\$ \$f_2\$\\
%&+ \|\Gamma\| \$f_1\$ \$f_2;g_2\$ + \$g_1\$ \|f_2-g_2\| + \|\bar\Gamma\| \|f_1-g_1\| \$ g_2\$\;,
%\end{align*}
%uniformly over all $f_i \in \cD^{\gamma_i,\eta_i}_{p_i,\alpha_i}$ and all $g_i \in \bar{\cD}^{\gamma_i,\eta_i}_{p_i,\alpha_i}$.
%\end{theorem}

\begin{proof}[Proof of Theorem \ref{Th:Mult}]
Regarding the local bound, we have
$$ \Big\| \frac{|f|_\zeta(z)}{t^{\frac{\eta-\zeta}{2}}} \Big\|_{L^p((0,T)\times \T^d)} \lesssim \sum_{\zeta_1+\zeta_2=\zeta}\Big\| \frac{|f_1|_{\zeta_1}(z)}{t^{\frac{\eta_1-\zeta_1}{2}}}\frac{|f_2|_{\zeta_2}(z)}{t^{\frac{\eta_2-\zeta_2}{2}}} \Big\|_{L^p((0,T)\times \T^d)}\;,$$
so that H\"older's inequality yields the required bound.\\
We turn to the translation bound, and write
\begin{equation}\label{Eq:DecompMultipl}\begin{split}
f(z+h) - \Gamma_{z+h,z} f (z) =& -\big(f_1(z+h) - \Gamma_{z+h,z} f_1 (z))(f_2(z+h) - \Gamma_{z+h,z} f_2 (z)\big) \\
& +  \Gamma_{z+h,z}f_1(z)  \Gamma_{z+h,z} f_2(z) -  \Gamma_{z+h,z}\big(f_1(z) f_2(z)\big) \\
& + f_1(z+h) \big(f_2(z+h) - \Gamma_{z+h,z} f_2 (z)\big) \\
& + f_2(z+h) \big(f_1(z+h) - \Gamma_{z+h,z} f_1 (z)\big)\;,
\end{split}\end{equation}
and bound these four terms separately.

The bound of the first term follows from H\"older's inequality. Regarding the second term, we note that the $\gamma$-regularity of the sectors ensures the following identity:
$$\Gamma_{z+h,z}f_1(z)  \Gamma_{z+h,z} f_2(z) -  \Gamma_{z+h,z}(f_1(z) f_2(z)) = \sum_{\nu_1+\nu_2\geq \gamma} (\Gamma_{z+h,z}\cQ_{\nu_1}f_1(z)) (\Gamma_{z+h,z}\cQ_{\nu_2} f_2(z))\;.$$
Fix $\nu_1,\nu_2 < \gamma$ such that $\nu_1+\nu_2 \ge \gamma$. For any $\zeta_i \le \nu_i$ such that $\zeta=\zeta_1+\zeta_2 < \gamma$, we get
\begin{align*}
&\sup_{h\in B(0,1)} \bigg\| \frac{\big| (\Gamma_{z+h,z} \cQ_{\nu_1} f_1(z))_{\zeta_1} (\Gamma_{z+h,z}\cQ_{\nu_2} f_2(z) )_{\zeta_2} \big|}{|h|^{\gamma-\zeta_1-\zeta_2} t^{\frac{\eta-\gamma}{2}} } \bigg\|_{L^p((3|h|^2,T-|h|^2)\times\T^d)}\\
&\lesssim\sup_{h\in B(0,1)} \bigg\| \Big(\frac{|h|}{\sqrt t}\Big)^{\nu_1+\nu_2-\gamma}  \frac{\big| f_1(z)\big|_{\nu_1}}{t^{\frac{\eta_1-\nu_1}{2}}} \frac{\big| f_2(z)\big|_{\nu_2}}{t^{\frac{\eta_2-\nu_2}{2}}} \bigg\|_{L^p((3|h|^2,T-|h|^2)\times\T^d)}\\
&\lesssim \sup_{h\in B(0,1)} \Big\| \frac{|f_1(z)|_{\nu_1}}{t^{\frac{\eta_1-\zeta_1}{2}}} \Big\|_{L^{p_1}((3|h|^2,T-|h|^2)\times\T^d)} \Big\| \frac{|f_2(z)|_{\nu_2}}{t^{\frac{\eta_2-\zeta_2}{2}}} \Big\|_{L^{p_2}((3|h|^2,T-|h|^2)\times\T^d)} \\
&\lesssim \$f_1\$_{\eta_1,T}\$f_2\$_{\eta_2,T}\;,
\end{align*}
uniformly over all $f_1,f_2$.\\
We turn to the third term, we have for any $\zeta=\zeta_1+\zeta_2 < \gamma$:
\begin{align*}
&\sup_{h\in B(0,1)} \bigg\| \frac{\cQ_{\zeta_1} f_1(z+h) \cQ_{\zeta_2}(f_2(z+h) - \Gamma_{z+h,z} f_2 (z))}{|h|^{\gamma-\zeta_1-\zeta_2} t^{\frac{\eta-\gamma}{2}} } \bigg\|_{L^p((3|h|^2,T-|h|^2)\times\T^d)}\\
\leq &  \Big\| \frac{|f_1(z)|_{\zeta_1}}{t^{\frac{\eta_1-\zeta_1}{2}}} \Big\|_{L^{p_1}} \sup_{h\in B(0,1)} \bigg\| \frac{|f_2(z+h) - \Gamma_{z+h,z} f_2 (z) |_{\zeta_2} }{|h|^{\gamma_2-\zeta_2} t^{\frac{\eta_2-\gamma_2}{2}}} t^{\frac{\gamma-\gamma_2-\zeta_1}{2}} |h|^{\gamma_2+\zeta_1-\gamma}\bigg\|_{L^{p_2}((3|h|^2,T-|h|^2)\times\T^d)}\;.
\end{align*}
Since $|h|\le \sqrt{t}$ in the integral above and since $\gamma\le\gamma_2 +\alpha_1\le\gamma_2+\zeta_1$, we deduce that
$$ t^{\frac{\gamma-\gamma_2-\zeta_1}{2}} |h|^{\gamma_2+\zeta_1-\gamma} \le 1\;,$$
so that the last quantity is bounded by a term of order $\$f_1\$_{\eta_1,T}\$f_2\$_{\eta_2,T}$. By symmetry, the fourth term of \eqref{Eq:DecompMultipl} is bounded in exactly the same way as the third.\\
In the case where we have two models, the bound of the local terms derives from the same type of arguments as above. On the other hand, to control the translation term, we write
$$ f_1f_2(z+h) - g_1g_2(z+h) - \Gamma_{z+h,z}(f_1f_2)(z) + \bar\Gamma_{z+h,z}(g_1g_2)(z) = A + B\;,$$
where
\begin{align*}
A &= f_1f_2(z+h) - g_1g_2(z+h) - \Gamma_{z+h,z}f_1(z) \Gamma_{z+h,z}f_2(z) + \bar\Gamma_{z+h,z}g_1(z) \bar\Gamma_{z+h,z}g_2(z)\;,\\
B&= \Gamma_{z+h,z}f_1(z) \Gamma_{z+h,z}f_2(z) - \bar\Gamma_{z+h,z}g_1(z) \bar\Gamma_{z+h,z}g_2(z)- \Gamma_{z+h,z}(f_1f_2)(z) + \bar\Gamma_{z+h,z}(g_1g_2)(z)\;,
\end{align*}
and we bound separately $A$ and $B$. Regarding $A$, we write
\begin{align*}
A &= \big(f_1(z+h) - g_1(z+h) - \Gamma_{z+h,z} f_1(z) + \bar\Gamma_{z+h,z} g_1(z)\big) f_2(z+h)\\
&+ \Gamma_{z+h,z} f_1(z) \big(f_2(z+h)-g_2(z+h) - \Gamma_{z+h,z} f_2(z) + \bar\Gamma_{z+h,z} g_2(z)\big)\\
&+ \bar\Gamma_{z+h,z}\big(g_1(z)-f_1(z)\big) \big(\bar\Gamma_{z+h,z}g_2(z)-g_2(z+h)\big)\\
&+\big(\bar\Gamma_{z+h,z} f_1(z) - \Gamma_{z+h,z} f_1(z)\big)\big(\bar\Gamma_{z+h,z} g_2(z) - g_2(z+h)\big)\\
&+\big(g_1(z+h)-\bar\Gamma_{z+h,z} g_1(z)\big)\big(f_2(z+h)-g_2(z+h)\big)\;,
\end{align*}
and the bound of the terms on the r.h.s.~can be obtained using similar arguments as before. We turn to $B$, which can be written as the sum over $\nu_1+\nu_2 \ge \gamma$ of
\begin{equation}\label{Eq:GammaGammabar}\begin{split}
&\Big(\Gamma_{z+h,z} \cQ_{\nu_1} f_1(z)\Gamma_{z+h,z} \cQ_{\nu_2} f_2(z)-\bar\Gamma_{z+h,z} \cQ_{\nu_1} g_1(z)\bar\Gamma_{z+h,z} \cQ_{\nu_2} g_2(z)\Big)\\
&= \Gamma_{z+h,z} \cQ_{\nu_1} \big(f_1(z)-g_1(z)\big) \Gamma_{z+h,z} \cQ_{\nu_2} f_2(z)\\
&+ \big(\Gamma_{z+h,z} - \bar\Gamma_{z+h,z}\big)\cQ_{\nu_1} g_1(z) \Gamma_{z+h,z}\cQ_{\nu_2}f_2(z)\\
&- \bar\Gamma_{z+h,z} \cQ_{\nu_1} g_1(z) \big(\bar\Gamma_{z+h,z} - \Gamma_{z+h,z}\big) \cQ_{\nu_2} g_2(z)\;.
\end{split}\end{equation}
Let us bound the first term on the r.h.s. We have for all $\zeta_1+\zeta_2=\zeta < \gamma$
\begin{align*}
&\Big\| \frac{\big|\Gamma_{z+h,z} \cQ_{\nu_1} \big(f_1(z)-g_1(z)\big)\big|_{\zeta_1} \big|\Gamma_{z+h,z} \cQ_{\nu_2} f_2(z)\big|_{\zeta_2}}{t^{\frac{\eta-\gamma}{2}} |h|^{\gamma-\zeta}} \Big\|_{L^p((3|h|^2,T-|h|^2)\times\T^d)}\\
&\lesssim \Big\| \frac{\big|f_1(z)-g_1(z)\big|_{\nu_1}}{t^{\frac{\eta_1-\nu_1}{2}}} \Big\|_{L^{p_1}((3|h|^2,T-|h|^2)\times\T^d)} \Big\| \frac{\big|f_2(z)\big|_{\nu_2}}{t^{\frac{\eta_2-\nu_2}{2}}} \frac{|h|^{\nu_1+\nu_2-\gamma}}{t^{\frac{\nu_1+\nu_2-\gamma}{2}}} \Big\|_{L^{p_1}((3|h|^2,T-|h|^2)\times\T^d)}\;,
\end{align*}
which is bounded by a term of order $\|f_1-g_1\| \|f_2\|$ since $\nu_1+\nu_2-\gamma \ge 0$ and since $|h| \le \sqrt{t}$ in the integral above. The bound of the other two terms in \eqref{Eq:GammaGammabar} is obtained similarly.
\end{proof}

\subsection{Convolution with the heat kernel} 

Let us introduce some notations first. We set
\begin{equation}\label{Eq:DiffP}
P^{k,\gamma'}_{m,z,\bar z} (\cdot) := \partial^k P_m(z-\cdot) - \sum_{\ell\in\N^{d+1}:|k+\ell| < \gamma'} \frac{(z-\bar z)^\ell}{\ell!} \partial^{k+\ell} P_m(\bar z-\cdot)\;,
\end{equation}
as well as $P^{k,\gamma'}_{z,\bar z} = \sum_{m\ge 0} P^{k,\gamma'}_{m,z,\bar z}$. Recall that $\gamma'$ is not an integer. We introduce
$$ \partial \gamma' := \{\ell \in\N^{d+1}: |\ell| > \gamma', |\ell-e_{m(\ell)}| < \gamma'\}\;,$$
where $e_i$ is the unit vector of $\R^{d+1}$ in the direction $i\in\{0,\ldots,d\}$ and $m(\ell) := \inf\{i:\ell_i \ne 0\}$ for all $\ell\in\N^{d+1}$. By~\cite[Prop A.1]{Hairer2014}, there exists a signed measure $\mu^\ell(z-\bar z,du)$ on $\R^{d+1}$, supported in the set $\{u\in \R^{d+1}: u_i \in [0,z_i-\bar z_i]\}$ with total mass equal to $\frac{(z-\bar z)^\ell}{\ell!}$ and such that
\begin{equation}\label{Eq:KernelDecompo}
P^{k,\gamma'}_{m,z,\bar z} (\cdot) = \sum_{\ell:k+\ell \in \partial \gamma'} \int_{\R^{d+1}} \partial^{k+\ell} P_m(\bar z + u-\cdot) \mu^\ell(z-\bar z,du)\;.
\end{equation}
Recall also that $L^p(n_0)$ stands for the space $L^p((2^{-2n_0}\wedge T,2^{-2(n_0-1)\wedge T})\times \T^d,dz)$. Finally, we set for every $m\ge 0$:
\begin{align*}
\cP_m f(z) &:= \sum_{\zeta\in \cA_\gamma} \sum_{k\in\N^{d+1}:|k| < \zeta + 2} \frac{X^k}{k!} \langle \Pi_z \cQ_\zeta f(z),\partial^k P_m(z-\cdot)\rangle\\
&+ \sum_{k\in\N^{d+1}:|k| < \gamma + 2} \frac{X^k}{k!} \langle \cR f - \Pi_z f(z) , \partial^k P_m(z-\cdot)\rangle\;.
\end{align*}
This is convenient since for every $k\in\N^{d+1}$ we have
$$ \cQ_k \cP_+^\gamma f(z) = \sum_{m\ge 0} \cQ_k \cP_m f(z)\;.$$

\begin{proof}[Proof of Theorem \ref{Th:ConvSing}]
We subdivide the proof into three steps: first we bound the local terms of the $\cD^{\gamma,\eta,T}_{p}$-norm, second the translation terms and finally we establish the convolution identity. We only consider the case where we work with a single model. The bounds in the case where we have two models can easily be obtained using the following two identities:
\begin{align*}
\Pi_z \cQ_\zeta \tau - \bar\Pi_z \cQ_\zeta \bar\tau &= \Pi_z \cQ_\zeta (\tau-\bar\tau) + (\Pi_z-\bar\Pi_z)\cQ_\zeta \bar \tau\;,\\
\big(\Pi_{z+h} \cQ_\zeta \Gamma_{z+h,z} - \bar\Pi_{z+h} \cQ_\zeta \bar\Gamma_{z+h,z}\big) \bar \tau &= \Pi_{z+h} \cQ_\zeta (\Gamma_{z+h,z} - \bar\Gamma_{z+h,z}) \bar\tau + (\Pi_{z+h} - \bar\Pi_{z+h})\cQ_\zeta \bar\Gamma_{z+h,z} \bar\tau\;.
\end{align*}
For notational convenience, we take $T=1$ in the proof. It is plain that the proof carries through if $T$ is arbitrary.

\textit{First step: local terms.}\\
At non-integer levels $\zeta\in\cA_{\gamma}$, we have for all $z\in (0,T)\times \T^d$
$$ \frac{\big|\cP_+^\gamma f(z) \big|_{\zeta+2}}{t^{\frac{\eta'-\zeta-2}{2}}} = \frac{\big|\cI(f(z)) \big|_{\zeta+2}}{t^{\frac{\eta'-\zeta-2}{2}}} \le \frac{\big|f(z) \big|_{\zeta}}{t^{\frac{\eta-\zeta}{2}}}\;,$$
so that the required bound follows at once. Let us now consider integer levels: we fix $k\in \N^{d+1}$ such that $|k| < \gamma'$. We have
$$ \Big\| \sum_{m\ge 0} \frac{\big|\cQ_k(\cP_m f)(z)\big|}{t^{\frac{\eta'-|k|}{2}}} \Big\|_{L^p((0,1)\times\T^d)} \lesssim \bigg(\sum_{n_0\ge 1} \Big\| \sum_{m\ge 0} \frac{\big|\cQ_k(\cP_m f)(z)\big|}{2^{-n_0(\eta'-|k|)}} \Big\|_{L^p(n_0)}^p\bigg)^{\frac1{p}}\;.$$
The bound is carried out differently according to the relative values of $n_0$ and $m$. First, we assume that $n_0 +2 \ge m$ and we write
$$ k! \cQ_k(\cP_m f)(z) = \langle \cR f,\partial^k P_m(z-\cdot)\rangle - \sum_{\zeta \le |k|-2} \langle \Pi_z\cQ_\zeta f(z),\partial^k P_m(z-\cdot)\rangle \;.$$
By \eqref{Eq:ReconsConvol}, we immediately get the desired bound for the first term on the r.h.s. Regarding the second term, we only have to consider non-integer values of $\zeta$: indeed, for integer values of $\zeta$ the corresponding terms vanish since $P_m$ is assumed to annihilate polynomials. We write for all $\zeta\le |k|-2$ (necessarily $\zeta < |k|-2$ from the previous observation)
\begin{align*}
\Big\| \frac{\big|\langle \Pi_z\cQ_\zeta f(z) , \partial^k P_m(z-\cdot) \rangle\big|}{2^{-n_0(\eta'-|k|)}} \Big\|_{L^p(n_0)} &\lesssim \Big\| \frac{|f(z)|_\zeta 2^{-m(2-|k|+\zeta)}}{2^{-n_0(\eta'-|k|)}} \Big\|_{L^p(n_0)}\\
&\lesssim 2^{-(n_0-m)(|k|-2-\zeta)} \Big\| \frac{|f(z)|_\zeta}{t^{\frac{\eta-\zeta}{2}}} \Big\|_{L^p(n_0)} \;,
\end{align*}
uniformly over all $m\le n_0 +2$ and all $n_0\ge 1$. Hence
\begin{align*}
\bigg(\sum_{n_0\ge 1} \Big\| \sum_{m\le n_0+2} \frac{\big|\langle \Pi_z\cQ_\zeta f(z) , \partial^k P_m(z-\cdot) \rangle\big|}{2^{-n_0(\eta'-|k|)}} \Big\|_{L^p(n_0)}^p\bigg)^{\frac1{p}} &\lesssim \Big\| \frac{|f(z)|_\zeta }{t^{\frac{\eta-\zeta}{2}}} \Big\|_{L^p((0,1)\times\T^d,dz)}\;,
\end{align*}
as required.\\
We turn to the case where $n_0+2 < m$ and we write
$$ k! \cQ_k(\cP_m f)(z) = \langle \cR f - \Pi_z f(z),\partial^k P_m(z-\cdot)\rangle + \sum_{\zeta > |k|-2} \langle \Pi_z\cQ_\zeta f(z),\partial^k P_m(z-\cdot)\rangle \;.$$
The bound of the second term proceeds analogously to the bound of the second term in the previous case: the change of sign of $n_0+2-m$ is compensated by the change of sign of $|k|-2-\zeta$ so that the series in $m$ converges. To bound the first term, we use Theorem \ref{Th:ReconstructionW} to get
\begin{align*}
\Big\| \frac{\big|\langle \cR f - \Pi_z f(z) , \partial^k P_m(z-\cdot) \rangle\big|}{2^{-n_0(\eta'-|k|)}} \Big\|_{L^p(n_0)} &\lesssim 2^{-(m-n_0)(\gamma+2-|k|)} \Big\| \sup_{\varphi\in\ccB^r}\frac{\big|\langle \cR f - \Pi_z f(z) , \varphi^{2^{-m}}_z \rangle\big|}{2^{-m\gamma} 2^{-n_0(\eta-\gamma)}} \Big\|_{L^p(n_0)}\\
&\lesssim 2^{-(m-n_0)(\gamma'-|k|)} \$ f\$_{\eta,T,D} \;,
\end{align*}
uniformly over all $1 \le n_0 < m-2$, where $D=(2^{-(n_0+1)},2^{-(n_0-1)})\times\T^d$ and the notation $\$ f\$_{\eta,T,D}$ was introduced below \eqref{Eq:BdLoc}. Using the fact that $|k| < \gamma'$, we get the bound
\begin{align*}
\bigg(\sum_{n_0\ge 1} \Big\| \sum_{m> n_0+2} \frac{\big|\langle \cR f - \Pi_z f(z) , \partial^k P_m(z-\cdot) \rangle\big|}{2^{-n_0(\eta'-|k|)}} \Big\|_{L^p(n_0)}^p\bigg)^{\frac1{p}} &\lesssim \$ f \$\;,
\end{align*}
as required.\\

\textit{Second step: translation terms.} We first consider the case $\frac13\sqrt t \le 2^{-m}$. We write
\begin{align*}
k! \cQ_k\big((\cP_m f)(z+h) - \Gamma_{z+h,z} (\cP_m f)(z)\big) &= \langle \cR f , P^{k,\gamma'}_{m,z+h,z}\rangle - \langle \Pi_z f(z) , P^{k,\gamma'}_{m,z+h,z}\rangle\\
&- \sum_{\zeta \le |k|-2} \langle \Pi_{z+h} \cQ_\zeta(f(z+h)-\Gamma_{z+h,z} f(z)) , \partial^k P_m(z+h-\cdot)\rangle\;.
\end{align*}
We bound separately the three terms on the r.h.s. Let us introduce the convenient notation $L^p(n_0,h)$ that stands for the space
$$L^p((2^{-2n_0},2^{-2(n_0-1)})\cap(3|h|^2,1-|h|^2) \times \T^d,dz)\;.$$
We start with the second term. Using \eqref{Eq:KernelDecompo}, we find
\begin{align*}
&\Big\| \frac{\langle \Pi_z f(z) , P^{k,\gamma'}_{m,z+h,z}\rangle}{2^{-n_0(\eta'-\gamma')} |h|^{\gamma'-|k|}} \Big\|_{L^p(n_0,h)}\\
&\lesssim \sum_{\ell:k+\ell \in \partial \gamma'} \Big\| \int_{\R^{d+1}}\frac{\langle \Pi_z f(z) , \partial^{k+\ell} P_m(z + u-\cdot)\rangle}{2^{-n_0(\eta'-\gamma')} |h|^{\gamma'-|k|}}\mu^\ell(h,du) \Big\|_{L^p(n_0,h)}\\
&\lesssim \sum_{\ell:k+\ell \in \partial \gamma'} |h|^{|\ell|+|k|-\gamma'} 2^{-n_0(\gamma-\zeta)} 2^{-m(2-|k|-|\ell|+\zeta)}\Big\| \frac{|f(z)|_\zeta}{2^{-n_0(\eta-\zeta)}}\Big\|_{L^p(n_0,h)}\;,
\end{align*}
so that, since $\gamma'-|k|-|\ell| < 0$ and $|h| \le \sqrt t$, we get
\begin{align*}
\Big\| \sum_{m\le n_0+2} \frac{\langle \Pi_z f(z) , P^{k,\gamma'}_{m,z+h,z}\rangle}{2^{-n_0(\eta'-\gamma')} |h|^{\gamma'-|k|}} \Big\|_{L^p(n_0,h)} &\lesssim \Big\| \frac{|f(z)|_\zeta}{2^{-n_0(\eta-\zeta)}}\Big\|_{L^p(n_0,h)}\;,
\end{align*}
uniformly over all $n_0\ge 1$. Taking the $\ell^p(n_0\ge 1)$-norm of the latter, we find a bound of order $\$f\$$ as required. We pass to the first term. Using \eqref{Eq:KernelDecompo}, we find
\begin{align*}
&\Big\| \sum_{m\le n_0+2}\frac{\langle \cR f , P^{k,\gamma'}_{m,z+h,z}\rangle}{2^{-n_0(\eta'-\gamma')} |h|^{\gamma'-|k|}} \Big\|_{L^p(n_0,h)}\\
&\lesssim \sum_{\ell:k+\ell \in \partial \gamma'} \Big\| \int_{\R^{d+1}}\sum_{m\le n_0+2}\frac{\langle \cR f , \partial^{k+\ell} P_m(z + u-\cdot)\rangle}{2^{-n_0(\eta'-\gamma')} |h|^{\gamma'-|k|}}\mu^\ell(h,du) \Big\|_{L^p(n_0,h)}\\
&\lesssim \sum_{\ell:k+\ell \in \partial \gamma'} \Big(\frac{2^{-n_0}}{|h|}\Big)^{\gamma'-|k|-|\ell|} \Big\| \sum_{m\le n_0+2}\frac{\big|\langle \cR f , \partial^{k+\ell} P_m(z-\cdot)\rangle\big|}{2^{-n_0(\eta'-|k|-|\ell|)}} \Big\|_{L^p([2^{-2(n_0+1)},2^{-2(n_0-2)}\wedge 1]\times\T^d)}\\
&\lesssim \sum_{\ell:k+\ell \in \partial \gamma'} \Big\| \sum_{m\le n_0+2}\frac{\big|\langle \cR f , \partial^{k+\ell} P_m(z-\cdot)\rangle\big|}{2^{-n_0(\eta'-|k|-|\ell|)}} \Big\|_{L^p([2^{-2(n_0+1)},2^{-2(n_0-2)}\wedge 1]\times\T^d)}\;.
\end{align*}
Since $\gamma'-|k|-|\ell| < 0$ and $|h|\le 2^{-n_0}$, we use \eqref{Eq:ReconsConvol} to bound the $\ell^p(n_0\ge 1)$-norm of the previous quantity by a term of order $\$f\$$ as required. 
Regarding the third term, by the same argument as above we can disregard the integer values $\zeta$. Thus, for all $\zeta < |k| -2$ we have
\begin{align*}
&\Big\| \frac{\langle \Pi_{z+h} \cQ_\zeta(f(z+h)-\Gamma_{z+h,z} f(z)) , \partial^k P_m(z+h-\cdot)\rangle}{2^{-n_0(\eta'-\gamma')} |h|^{\gamma'-|k|}} \Big\|_{L^p(n_0,h)}\\
&\lesssim 2^{-m(2-|k|+\zeta)}|h|^{|k|-\zeta+\gamma-\gamma'} \Big\| \frac{\big| f(z+h)-\Gamma_{z+h,z} f(z) \big|_\zeta}{|h|^{\gamma-\zeta} 2^{-n_0(\eta-\gamma)}} \Big\|_{L^p(n_0,h)}\;,
\end{align*}
uniformly over all $n_0\ge m-2$ and all $|h| \le 2^{-n_0}$. Using the fact that $\zeta < |k|-2$, we deduce that the sum over all $m\le n_0+2$ is bounded by a term of order
$$ \Big\| \frac{\big| f(z+h)-\Gamma_{z+h,z} f(z) \big|_\zeta}{|h|^{\gamma-\zeta} 2^{-n_0(\eta-\gamma)}} \Big\|_{L^p(n_0,h)}\;,$$
so that the $\ell^p(n_0\ge 1)$-norm of the latter is bounded by a term of order $\$f\$$.\\

We now consider the case where $|h| \le 2^{-m} \le \frac13\sqrt{t}$, in which case we write
\begin{align*}
k! \cQ_k\big((\cP_m f)(z+h) - \Gamma_{z+h,z} (\cP_m f)(z)\big) &= \langle \cR f -\Pi_z f(z) , P^{k,\gamma'}_{m,z+h,z}\rangle\\
&- \sum_{\zeta \le |k|-2} \langle \Pi_{z+h} \cQ_\zeta(f(z+h)-\Gamma_{z+h,z} f(z)) , \partial^k P_m(z+h-\cdot)\rangle\;.
\end{align*}
Using \eqref{Eq:KernelDecompo} and the reconstruction bound \eqref{Eq:ReconsBound}, we get
\begin{align*}
&\Big\| \sum_{m:|h| \le 2^{-m} \le \frac13 2^{-n_0}} \frac{\langle \cR f -\Pi_{z} f(z), P^{k,\gamma'}_{m,z+h,z} \rangle}{2^{-n_0(\eta'-\gamma')} |h|^{\gamma'-|k|}} \Big\|_{L^p(n_0)}\\
&\lesssim \sum_{\ell:k+\ell \in \partial \gamma'} \Big\| \sum_{m:|h| \le 2^{-m} \le \frac13 2^{-n_0}} \int_{\R^{d+1}} \frac{\langle \cR f -\Pi_{z} f(z), \partial^{k+\ell} P_m(z + u-\cdot)\rangle}{2^{-n_0(\eta'-\gamma')} |h|^{\gamma'-|k|}}\mu^\ell(h,du) \Big\|_{L^p(n_0,h)}\\
&\lesssim \sum_{\ell:k+\ell \in \partial \gamma'} \sum_{m:|h| \le 2^{-m} \le \frac13 2^{-n_0}} 2^{-m(\gamma'-|k|-|\ell|)}|h|^{|k|+|\ell|-\gamma'} \$f\$_{(2^{-(n_0+1)},2^{-(n_0-2)}\wedge 1)\times\T^d}\\
&\lesssim \$f\$_{(2^{-(n_0+1)},2^{-(n_0-2)}\wedge 1)\times\T^d}\;,
\end{align*}
uniformly over all $n_0\ge 1$ such that $|h| \le 2^{-n_0} / 3$. Taking the $\ell^p(n_0\ge 1)$)norm, this yields a bound of order $\$f\$$ as required. We turn to the second term. By the same argument as above, we consider the non-integer values of $\zeta$ only, and get:
\begin{align*}
&\Big\| \sum_{m:|h| \le 2^{-m} \le \frac13 2^{-n_0}} \frac{\langle \Pi_{z+h} \cQ_\zeta(f(z+h)-\Gamma_{z+h,z} f(z)) , \partial^k P_m(z+h-\cdot) \rangle}{2^{-n_0(\eta'-\gamma')} |h|^{\gamma'-|k|}} \Big\|_{L^p(n_0,h)}\\
&\lesssim \sum_{m:|h| \le 2^{-m} \le \frac13 2^{-n_0}} 2^{-m(2-|k|+\zeta)} |h|^{|k|-\zeta+\gamma-\gamma'} \Big\|  \frac{|f(z+h)-\Gamma_{z+h,z} f(z)|_\zeta}{2^{-n_0(\eta-\gamma)} |h|^{\gamma-\zeta}} \Big\|_{L^p(n_0,h)}\\
&\lesssim  \Big\|  \frac{|f(z+h)-\Gamma_{z+h,z} f(z)|_\zeta}{2^{-n_0(\eta-\gamma)} |h|^{\gamma-\zeta}} \Big\|_{L^p(n_0,h)}\;,
\end{align*}
so that the $\ell^p(n_0\ge 1)$-norm of the latter is bounded by a term of order $\$f\$$.\\

Let us finally consider the case where $2^{-m} \le |h| \le \frac1{3} \sqrt{t}$. We write
\begin{align*}
k! \cQ_k\big((\cP_m f)(z+h) - \Gamma_{z+h,z} (\cP_m f)(z)\big) &= \langle \cR f -\Pi_{z+h} f(z+h), \partial^k P_m(z+h-\cdot) \rangle\\
&- \langle \cR f -\Pi_{z} f(z), \sum_{\ell:|k+\ell| < \gamma'} \frac{h^\ell}{\ell!} \partial^{k+\ell} P_m(z-\cdot) \rangle\\
&+ \sum_{\zeta > |k|-2} \langle \Pi_{z+h} \cQ_\zeta(f(z+h)-\Gamma_{z+h,z} f(z)) , \partial^k P_m(z+h-\cdot)\rangle\;.
\end{align*}
Regarding the first term, we set $D=(2^{-2(n_0+1)},2^{-2(n_0-1)})\times\T^d$ and we have
\begin{align*}
&\Big\| \sum_{m:2^{-m} \le |h|} \frac{\langle \cR f -\Pi_{z+h} f(z+h), \partial^k P_m(z+h-\cdot) \rangle}{2^{-n_0(\eta'-\gamma')} |h|^{\gamma'-|k|}} \Big\|_{L^p(n_0,h)}\\
&\lesssim \sum_{m:2^{-m} \le |h|} 2^{-m(2-|k|+\gamma)}|h|^{|k|-\gamma'} \$f\$_{\eta,T,D}\\
&\lesssim  \$f\$_{\eta,T,D}\;,
\end{align*}
so that the corresponding $\ell^p$-norm is bounded by a term of order $\$f\$$. The bound of the second term is very similar. Regarding the third term, we have
\begin{align*}
&\Big\| \sum_{m:2^{-m} \le |h|} \frac{\langle \Pi_{z+h} \cQ_\zeta(f(z+h)-\Gamma_{z+h,z} f(z)) , \partial^k P_m(z+h-\cdot)\rangle}{2^{-n_0(\eta'-\gamma')} |h|^{\gamma'-|k|}} \Big\|_{L^p(n_0,h)}\\
&\lesssim \sum_{m:2^{-m} \le |h|} 2^{-m(2-|k|+\zeta)}|h|^{|k|-\zeta+\gamma-\gamma'} \Big\| \frac{|f(z+h)-\Gamma_{z+h,z} f(z)|_\zeta}{2^{-n_0(\eta-\gamma)} |h|^{\gamma-\zeta}} \Big\|_{L^p(n_0,h)}\\
&\lesssim \Big\| \frac{|f(z+h)-\Gamma_{z+h,z} f(z)|_\zeta}{2^{-n_0(\eta-\gamma)} |h|^{\gamma-\zeta}} \Big\|_{L^p(n_0,h)}\;,
\end{align*}
so that the corresponding $\ell^p$-norm is bounded by a term of order $\$f\$$.\\

\textit{Third step: convolution identity.}
The element $\cP_+ f \in \cD^{\gamma',\eta',T}_{p}$ can always be restricted to $\cD^{\gamma'',\eta'',T}_{p}$ for any given $\gamma'' \in (0,1)$, and for $\eta''= \eta'+(\gamma''-\gamma')$. Using the uniqueness part of Theorem \ref{Th:ReconstructionW}, we deduce that the identity $P_+ * \cR f = \cR \cP_+ f$ holds as soon as we have
$$ \sup_{\lambda \in (0,1]} \Big\| \sup_{\eta \in \ccB^r} \frac{\big|\langle P_+ *\cR f - \Pi_z \cP_+ f(z) , \eta^\lambda_z\rangle\big|}{\lambda^{\gamma''} t^{\frac{\eta''-\gamma''}{2}}} \Big\|_{L^p((3\lambda^2,T-\lambda^2)\times \T^d)} < \infty\;.$$
A simple computation shows that
$$ \langle P_+ *\cR f - \Pi_z \cP_+ f(z) , \eta^\lambda_z\rangle = \sum_{m\ge 0} \langle \cR f -\Pi_z f(z) , \int \eta^\lambda_z (z+h) P^{0,\gamma''}_{m,z+h,z} dh\rangle\;,$$
One has to distinguish three cases according as $\frac13 \sqrt{t} \le 2^{-m}$, $\lambda \le 2^{-m} \le \frac13 \sqrt {t}$ and $2^{-m} \le \lambda \le \frac13 \sqrt{t}$. In every case, the bound is virtually the same as the bound of the translation terms presented above so we do not provide the details.
\end{proof}
\begin{proof}[Proof of Theorem \ref{Th:ConvSmooth}]
Recall that $P_-$ is a compactly supported, smooth function. Using the bound \eqref{Eq:ReconsConvol2} that was proved earlier, we get
$$ \Big\| \frac{\langle \cR f, \partial^k P_-(z-\cdot) \rangle}{t^{\frac{\eta'-|k|}{2}}} \Big\|_{L^p((0,T)\times\T^d,dz)} \lesssim \$f\$\;,$$
for all $k\in \N^{d+1}$. This yields the required bounds on the local terms of the $\cD^{\gamma',\eta',T}_p$-norm. Regarding the translation terms, we observe that
$$ k!\cQ_k \big(\cP_- f(z+h) - \Gamma_{z+h,z} \cP_- f(z)\big) = \langle \cR f , P^{k,\gamma'}_{-,z+h,z}\rangle\;,$$
where $P^{k,\gamma'}_{-,z,z+h}$ is the function defined in \eqref{Eq:DiffP} upon replacing $P_m$ by $P_-$. Combining \eqref{Eq:KernelDecompo} and the bound already obtained above, we easily deduce that the bound on the translation terms is satisfied.\\
To get the identity $\cR \cP_- f = P_- * \cR f$, we restrict $\cP_- f$ to $\cT_{\gamma''}$ with $\gamma'' \in (0,1)$ and we observe that
$$ \Pi_z \cP_- f(z) = (P_- * \cR f)(z)\;,\quad z\in(0,T)\times\T^d\;.$$
Since the r.h.s.~defines a smooth function, the uniqueness part of Theorem \ref{Th:ReconstructionW} ensures that $\cR \cP_- f$ coincides with $P_- * \cR f$ as required.
\end{proof}

\subsection{Convolution of the shift}
\begin{proof}[Proof of Lemma \ref{Lemma:ConvolShift}]
Recall the notation $L^2(n_0)$. Any function $h\in L^2((0,T)\times\T^d)$ can be viewed as an element of $\cB^{-\kappa/3}_{2}((-\infty,T]\times \T^d)$ for some small $\kappa > 0$. Applying Lemma \ref{Lemma:SpecificConvol} and Equation \eqref{Eq:SpecificConvol2}, we deduce that
\begin{equation}\label{Eq:BoundConvolShift}
\bigg(\sum_{n_0\ge n_T} \Big\| \sum_{m\le n_0+4} \frac{\big|\langle h,\partial^k P_m(z-\cdot) \rangle\big|}{2^{-n_0(-\frac{\kappa}{2}-|k|+2)}} \Big\|_{L^2(n_0)}^2\bigg)^{\frac1{2}} < \infty\;.
\end{equation}
and
\begin{equation}\label{Eq:BoundConvolShift2}
\bigg(\sum_{n_0\ge n_T} \Big\| \frac{\big|\langle h,\partial^k P_-(z-\cdot) \rangle\big|}{2^{-n_0(-\frac{\kappa}{2}-|k|+2)}} \Big\|_{L^2(n_0)}^2\bigg)^{\frac1{2}} < \infty\;.
\end{equation}
We now prove that $\cP h$ belongs to $\cD^{\gamma,\gamma,T}_{2}$ with $\gamma = 2-\kappa$. Regarding the local terms, first observe that
$$ k! \big|\cP h(z)\big|_k = \big| \langle h, \partial^k P_-(z-\cdot) \rangle\big| + \sum_{m\ge 0} \big| \langle h, \partial^k P_m(z-\cdot) \rangle\big|\;.$$
We distinguish the cases $n_0+2 < m$ and $n_0+2 \ge m$ where $n_0$ is the integer such that $t\in [2^{-2n_0},2^{-2(n_0-1)}]$. Assume first that $n_0+2 < m$. We have
\begin{align*}
\Big\| \frac{\langle h , \partial^k P_m(z-\cdot) \rangle}{t^{\frac{\gamma-|k|}{2}}} \Big\|_{L^2(n_0)}\lesssim \frac{2^{-m(2-|k|-\frac{\kappa}{2})}}{2^{-n_0(\gamma-|k|)}}\Big\| \frac{ \langle h , \partial^k P_m(z-\cdot) \rangle}{2^{-m(2-|k|-\frac{\kappa}{2})}} \Big\|_{L^2(n_0)}\;,
\end{align*}
so that
\begin{align*}
\bigg(\sum_{n_0\ge n_T} \Big\| \sum_{m>n_0+2} \frac{\big|\langle h , \partial^k P_m(z-\cdot) \rangle\big|}{t^{\frac{\gamma-|k|}{2}}} \Big\|_{L^2(n_0)}^2 \bigg)^\frac12 &\lesssim \sum_{m\ge 0} \bigg(\sum_{m-2 >n_0\ge n_T} \Big\|  \frac{\big|\langle h , \partial^k P_m(z-\cdot) \rangle\big|}{t^{\frac{\gamma-|k|}{2}}} \Big\|_{L^2(n_0)}^2 \bigg)^\frac12\\
&\lesssim \sum_{m\ge 0} 2^{-m\frac{\kappa}{2}} \bigg(\sum_{m-2 >n_0\ge n_T} \Big\| \frac{ \langle h , \partial^k P_m(z-\cdot) \rangle}{2^{-m(2-|k|-\frac{\kappa}{2})}} \Big\|_{L^2(n_0)}^2 \bigg)^\frac12\\
&\lesssim \sum_{m\ge 0} 2^{-m\frac{\kappa}{2}} \sup_{m\ge 0} \Big\| \sup_{\varphi\in\ccB^r} \frac{ \big|\langle h , \varphi^{2^{-m}}_z \rangle\big|}{2^{m\frac{\kappa}{2}}} \Big\|_{L^2((0,T)\times\T^d)}\;,
\end{align*}
which is bounded as required. The computation is similar for $P_-$. On the other hand, when $m \le n_0+2$, \eqref{Eq:BoundConvolShift} and \eqref{Eq:BoundConvolShift2} yield the desired bound.\\
To treat translation terms in the norm, one proceeds similarly. Actually, the proof is very similar to that of Theorem \ref{Th:ConvSing}: one has to distinguish three cases according to the relative values of $|h|$, $\sqrt t$ and $2^{-n_0}$, but the arguments are essentially the same.
\end{proof}

%\begin{thebibliography}{99}
%
%\end{thebibliography}

%\bibliographystyle{plain}
%\bibliography{library_additive}

\end{document}